\documentclass[11pt]{article} 
\usepackage{a4,amssymb, amsmath, amsthm,amstext,amsfonts}
\usepackage{color}
\usepackage{float}
\usepackage{graphicx}
\usepackage{esint}
\usepackage{a4wide}
\usepackage{babel}
\usepackage{bbm}
\usepackage{mathrsfs}
\usepackage{upgreek} 
\usepackage{setspace}
\usepackage{multirow}
\usepackage{placeins}
\usepackage{array}
\usepackage{hyperref}
\usepackage{algorithmic}
\usepackage{algorithm}
\usepackage{bm}
\usepackage{nicefrac}
\usepackage{subfig}
\usepackage{wrapfig}
\usepackage{graphicx}

\usepackage[title]{appendix}
  
\newtheorem{theorem}{Theorem}[section]
\newtheorem{corollary}[theorem]{Corollary}
\newtheorem{lemma}[theorem]{Lemma}
\newtheorem{proposition}[theorem]{Proposition}

\theoremstyle{definition}
\newtheorem{definition}[theorem]{Definition}

\theoremstyle{remark}
\newtheorem{remark}[theorem]{Remark}

\newcommand{\R}{\mathbb{R}}
\newcommand{\C}{\mathbb{C}}
\newcommand{\D}{\mathbb{D}}

\begin{document}

\providecommand{\keywords}[1]{{\noindent \textit{Key words:}} #1}
\providecommand{\msc}[1]{{\noindent \textit{Mathematics Subject Classification:}} #1}

\title{Higher-order time domain boundary elements  for elastodynamics --  graded meshes and \textit{hp} versions\\\vskip 0.8cm}
\author{ Alessandra Aimi\thanks{Department of Mathematical, Physical and Computer Sciences,
University of Parma, Parco Area delle Scienze, 53/A, 43124, Parma,
Italy, email: alessandra.aimi@unipr.it, giulia.dicredico@unipr.it. \newline Members of the INDAM-GNCS Research Group, Italy.} \and Giulia Di Credico${}^\ast$\thanks{Engineering Mathematics, University of Innsbruck, Innsbruck, Austria,  email: heiko.gimperlein@uibk.ac.at.} \and Heiko Gimperlein${}^\dagger$ \and Ernst P.~Stephan\thanks{Institute of Applied Mathematics, Leibniz University Hannover, 30167 Hannover, Germany, email: stephan@ifam.uni-hannover.de.}}\date{}

\maketitle \vskip 0.5cm
\begin{abstract}
\noindent  The solution to the elastodynamic equation in the exterior of a polyhedral domain or a screen exhibits singular behavior from the corners and edges. The detailed expansion of the singularities implies quasi-optimal estimates for piecewise polynomial approximations of the Dirichlet trace of the solution and the traction. The results are applied to $hp$ and graded versions of the  time domain boundary element method for the weakly singular and the hypersingular integral equations. Numerical examples confirm the theoretical results for the Dirichlet and Neumann problems  for screens and for polygonal domains in 2d. They exhibit the expected quasi-optimal convergence rates and the singular behavior of the solutions.
\end{abstract}

\msc{65M38 (primary); 65M15; 74S15; 35L67 (secondary)}\\
\keywords{time domain boundary element method, graded mesh, $hp$ version, elastodynamics.}

\section{Introduction}

Solutions to elliptic and parabolic boundary value problems in polyhedral domains exhibit singularities in a neighborhood of the corners and edges. Numerical approximations by finite or boundary element methods take into account the nonsmooth behavior with local mesh refinements or higher polynomial degrees to recover optimal convergence rates. The resulting $h$, $p$ and $hp$ methods have been studied for several decades, see e.g.~\cite{schwab} for finite elements and \cite{gwinsteph} for boundary elements.\\

For hyperbolic equations in conical or wedge domains the singular behavior of the solution has been clarified by Plamenevski\v{\i} and collaborators since the late 1990's \cite{kokotov, kokotov2, matyu, plamenevskii}. The explicit singular expansions were used by M\"{u}ller and Schwab to prove optimal convergence rates for a finite element method on algebraically graded meshes for the wave and elastodynamic equations in polygonal domains in $\mathbb{R}^2$ \cite{mueller, mueller2}. Corresponding results for the wave equation in $\mathbb{R}^3$ were obtained by two of the authors, leading to approximation results for boundary element methods (TDBEM) on graded meshes \cite{graded}, $hp$ versions \cite{hp} and the efficiency of a posteriori error estimates for adaptive refinement procedures \cite{adaptive}.\\
 
In this article we initiate the study of $h$, $p$ and $hp$ time domain boundary element methods for the Dirichlet and Neumann problems of elastodynamics in a polyhedral domain $\Omega \subset \mathbb{R}^n$, $n=2,3$. Based on the approach by Plamenevski\v{\i} and singular expansions for the time independent Lam\'{e} equation, we obtain a detailed description of the singularities of the solution for the model 3d geometries of a wedge and a cone, as well as 2d polygonal domains. The expansions imply quasi-optimal convergence rates for piecewise polynomial approximations on graded meshes and by $hp$ versions. \\

To be specific, we formulate the set-up and results for exterior problems. Let $\Gamma \subset \mathbb{R}^n$, $n=2,3$, be a screen or closed surface and denote by $\Omega$ the connected exterior $\Omega \subset \mathbb{R}^n$ of $\Gamma$. This article considers the dynamics of a linear elastic body with \textit{Lam\'e parameters}  $\lambda,\mu>0$ and mass density $\rho$, as described by the time dependent elastodynamic equation
\begin{equation}\label{Navierintro}
(\lambda+\mu)\nabla(\nabla \cdot \textbf{u})+\mu\Delta\textbf{u}-\rho\ddot{\textbf{u}}=0, \quad  \textbf{x}\in\Omega,\: t\in(0,T]\ .
\end{equation}
We impose homogeneous initial conditions $\textbf{u}(0,\textbf{x}) = \partial_t \textbf{u}(t, \textbf{x}) = 0$ and consider either Dirichlet boundary conditions, $\textbf{u} = \textbf{g}$, or Neumann boundary conditions involving the traction, $\textbf{p}(\textbf{u}) = \textbf{h}$. \\

To solve \eqref{Navierintro} numerically, we formulate it as an equivalent time dependent integral equation on $\Gamma$. For Dirichlet boundary conditions we study 
\begin{align} \label{BIEintro}
\mathcal{V}\pmb{\Phi}(\textbf{x},t)&=\left(\mathcal{K}+\frac{1}{2}\right) \textbf{g}(\textbf{x},t), \qquad (\textbf{x},t)\in\Gamma \times [0,T]\ ,
\end{align}
involving the weakly singular integral operator $\mathcal{V}$ and the double layer integral operator $\mathcal{K}$. $\mathcal{V}$ and $\mathcal{K}$ are defined from a fundamental solution $\textbf{G}$ to \eqref{Navierintro} and its traction $\textbf{p}_{\pmb{\xi}}(\textbf{G})$
 \begin{align*}\mathcal{V} \pmb{\Phi}(\textbf{x},t) &= \int_0^t\int_{\Gamma}\textbf{G}(\textbf{x},\pmb{\xi};t,\tau)\pmb{\Phi}(\pmb{\xi},\tau)d\Gamma_{\pmb{\xi}}d\tau\ , \\ \mathcal{K} \pmb{\Phi}(\textbf{x},t) &= \int_0^t\int_{\Gamma}\textbf{p}_{\pmb{\xi}}(\textbf{G})(\textbf{x},\pmb{\xi};t,\tau)^T\pmb{\Phi}(\pmb{\xi},\tau)d\Gamma_{\pmb{\xi}}d\tau\ .\end{align*}
The Neumann problem is similarly formulated as an equation for the hypersingular integral operator $\mathcal{W}$, see \eqref{explicit BIE hypersingular}. The weak formulation of these integral equations is approximated using Galerkin boundary elements $\pmb{\Phi}_{h,\Delta t} \in \left(V_{\Delta t,q}\otimes X^{-1}_{h,p}\right)^n$, based on tensor products of piecewise polynomial functions on a quasi-uniform or graded mesh in space and a uniform mesh in time. \\

The convergence rate of the error is determined by the singularities of the solution of \eqref{Navierintro} at non-smooth boundary points of the domain $\Omega$. Near an edge or a cone point of the boundary $\Gamma \subset \mathbb{R}^3$ we obtain a singular expansion of the solution into a leading part of explicit singular functions plus smoother remainder terms. {Expansions in a wedge, respectively a cone, are obtained in \eqref{wedgeexpand} and \eqref{coneexpand}: if we treat the variable along the edge as a parameter, the expansion in a wedge reduces to the case of a polygon in 2d, where in a neighborhood of a vertex it takes the form} 
\begin{align*}
\textbf{u}(t,\textbf{x}) &= \chi(r)r^{\nu^*} \bold{a}(t,\phi) + \textbf{u}_{0}(t,r, \phi) \ ,\\
\textbf{p}(\textbf{u})(t,\textbf{x}) &=\chi(r)r^{\nu^*-1} \bold{b}(t,\phi) + \pmb{\phi}_{0}(t,r, \phi) \ .
\end{align*}
{Here, $(r,\phi)$ are polar coordinates centered at the vertex, the exponent $\nu^*$ is determined by the opening angle $\omega$ at the vertex and by the elastic parameters,}  and $\textbf{u}_{0}$, $\pmb{\phi}_{0}$ are remainder terms of lower order. 
In particular, for a fixed time $t$ the solution to \eqref{Navierintro} admits an explicit singular expansion with the same behavior as the time independent Lam\'{e} equation. 

This asymptotic expansion of the solution $\textbf{u}$ and the traction $\textbf{p}(\textbf{u})$ gives rise to quasi-optimal convergence rates {in space-time anisotropic Sobolev norms. See \eqref{sobdef} for the definition of the Sobolev space $H^r_\sigma(\mathbb{R}^+,\widetilde{H}^s({\Gamma}))$ and \eqref{sobnormdef} for the definition of the norm $\|\cdot\|_{r,s,\Gamma,\ast} $.} We consider the approximation error of the solution on graded meshes, as defined in \eqref{gradedmesh}, in Corollary \ref{approxcor2}a) and the $hp$ version on quasi-uniform meshes in Corollary \ref{approxcor1}a). There the approximation error is determined by an exponent $\tilde{\alpha}$, which depends on the geometry (wedge, cone) and the elastic parameters, see \eqref{alphadef}: 

\vspace{0.3cm}

\noindent \textbf{Theorem.} Let $\varepsilon>0$ and $\sigma>0$.
a) Let $\pmb{ \Phi}$ be the solution to the single layer integral equation  \eqref{BIEintro} and  $\pmb{ \Phi}_{h,\Delta t}^{\tilde{\beta}}  \in \left(V_{\Delta t,q}\otimes X^{-1}_{h,0}\right)^n$ the best approximation to $\pmb{ \Phi}$ in the norm of ${H}^{r}_\sigma(\R^+, \widetilde{H}^{-\frac{1}{2}}(\Gamma))^n$ on a ${\tilde{\beta}}$-graded spatial mesh  with $\Delta t \lesssim h_1$. Then for $p=1,2,3,\dots$ $\|\pmb{ \Phi}-\pmb{ \Phi}_{h, \Delta t}^{\tilde{\beta}}\|_{r,-\frac{1}{2}, \Gamma, \ast} \leq C_{{\tilde{\beta}},\varepsilon} h^{\min\{{\tilde{\beta}} \tilde{\alpha}-\varepsilon, \frac{3}{2}\}}$. \\

\noindent b) Let $\pmb{\Phi}$ be the solution to the single layer integral equation \eqref{BIEintro} and  $\pmb{\Phi}_{h,\Delta t} \in \left(V_{\Delta t,p}\otimes X^{-1}_{h,p}\right)^n$ the best approximation  in the norm of ${H}^{r}_\sigma(\R^+, \widetilde{H}^{-\frac{1}{2}}(\Gamma))^n$ to $\pmb{\Phi}$  on a quasiuniform spatial mesh  with $\Delta t \lesssim h$. Then for $p=0,1,2,\dots$ $$\|\pmb{\Phi} - \pmb{\Phi}_{h,\Delta t}\|_{r, -\frac{1}{2}, \Gamma, \ast} \lesssim \left(\frac{h}{(p+1)^2}\right)^{\tilde{\alpha}{-\varepsilon}} + \left(\frac{\Delta t}{p+1}\right)^{p+1-r}+ \left(\frac{h}{p+1}\right)^{\frac{1}{2}+\eta}\ ,$$
 where {$r \in [0,p+1)$}  and $\pmb{\phi}_0 \in {H}^{p+1}_\sigma(\R^+, \widetilde{H}^{\eta}(\Gamma))^n$ is the regular part of the singular expansion of $\pmb{p}(\bold{u})|_\Gamma$.

\vspace{0.3cm}

Corresponding results in the case of a 2d polygon are obtained as a result of the edge problem. Corollaries \ref{approxcor2}b)  and \ref{approxcor1}b) contain analogous results for the hypersingular integral equation of the Neumann problem. As the analysis is local on $\Gamma$, the extension to the single layer and hypersingular integral equations for interior problems is immediate.\\

Numerical experiments are presented for the weakly singular and hypersingular integral operators in polygonal and crack geometries in $\mathbb{R}^2$. They achieve the predicted convergence rates on graded meshes and for the $hp$ version. Furthermore, they confirm the leading singular exponents of the solution, and the $hp$ version on a geometrically graded mesh \eqref{geometric points} exhibits faster than algebraic convergence. \\

Boundary element methods for time dependent problems have attracted much recent interest, see \cite{costabel,review,hd,sayas} for an overview. They are of particular relevance for problems which cannot be reduced to the frequency domain, such as nonlinear problems or problems involving a broad range of frequencies \cite{contact}. While their application to elasticity has long been studied in engineering \cite{antes}, their analysis for elastodynamic scattering and crack problems was initiated by B\'{e}cache and Ha Duong in \cite{Becache1993, Becache1994}. Recent developments include space-time Galerkin and convolution quadrature methods, fast discretizations \cite{Aimi20,hsiao1,schanz2,schanz}, as well as more complex elastic behavior \cite{hsiao2}.

For the time independent Lam\'{e} equation in  singular domains, such as with a crack, detailed asymptotic expansions have been studied extensively, partly motivated by applications to computing quantities of interest like stress intensity factors, see e.g.~\cite{bs,dcy,grisvard87,grisvard89,oy}. Using such expansions, von Petersdorff \cite{disspetersdorff} derived quasi-optimal error estimates for boundary elements on graded meshes. The $hp$ version on geometrically graded meshes  was studied in \cite{ms}. 
Sharp $hp$-explicit estimates on smooth open surfaces with quasiuniform meshes are due to Bespalov \cite{besp}, following earlier work of Bespalov and Heuer for the Laplace and Lam\'{e} equations \cite{heuer01, heuer2}.  \\

\noindent \emph{Structure of this article:} Section \ref{sec 2} reviews the {Dirichlet and Neumann boundary value problems for \eqref{Navierintro}} and their formulation as boundary integral equations {in terms of the weakly singular, respectively hypersingular operators. {Proposition} \ref{wellposedness} establishes the well-posedness of these equations.} The regularity of solutions to the elastodynamic problem is addressed in Section \ref{regularitysection}, see also Appendix B for the theoretical setting used to formulate the results. Taking their traces we get corresponding results for the solutions of the integral equations. In Subsection \ref{regularitywedge} the solution of the elastodynamic problem in a wedge is analyzed, in Subsection \ref{regularitycone} in a cone. Special consideration is given to 2d problems in Subsection \ref{regularity2d}. The BEM discretization and time integration are discussed in Section \ref{sec 3}. In Section \ref{approxsection} approximation results are derived, both for the $h$ version TDBEM on graded meshes and the $hp$ version. Both a circular wedge and a cone geometry are considered. The 2d case of a polygon corresponds to the theoretical error estimates for the numerical results in Section \ref{sec:numer}. {Section  \ref{sec:algo} discusses algorithmic aspects 
of the implementation.} Appendix A introduces the relevant Sobolev space setting for the error analysis together with the mapping properties of the integral operators and the associated weak formulations. {Appendix B describes crucial theoretical ingredients for the analysis of the elastodynamic problem in a wedge and in a cone.} In Appendix C we collect some additional auxiliary results for the error analysis.\\ 


\noindent {\emph{Notation:}  For vectors/vector fields (written in bold letters) the operators and
norms are understood component\-wise and not marked additionally. We write $f \lesssim g$ provided there exists a constant $C$ such that $f \leq Cg$. If the constant
$C$ is allowed to depend on a parameter $\sigma$, we write $f \lesssim_\sigma g$.}

\section{Model problem and boundary integral equations}\label{sec 2}

We consider elastic wave propagation in a Lipschitz domain $\Omega=\mathbb{R}^n\setminus \overline{\Omega'}$ exterior to the bounded domain $\Omega'$, with piecewise smooth boundary $\Gamma = \partial \Omega$, $n=2$ or $3$. As a limiting case, also screen problems in $\Omega = \mathbb{R}^n\setminus \overline{\Gamma}$ are considered, outside an open arc $\Gamma \subset \mathbb{R}^2$ or open surface $\Gamma \subset \mathbb{R}^3$. In the absence of external body forces the displacement field $\textbf{u}(\textbf{x},t)=(u_1,\dots,u_n)^{\top}(\textbf{x},t)$, $\textbf{x}=(x_1,\dots, x_n)^{\top}\in\mathbb{R}^n$, satisfies the \textit{elastodynamic equation}:
\begin{equation}\label{Navier equation}
(\lambda+\mu)\nabla(\nabla \cdot \textbf{u})+\mu\Delta\textbf{u}-\varrho\ddot{\textbf{u}}=0, \quad  \textbf{x}\in\Omega,\: t\in(0,T],
\end{equation}
where $\lambda,\mu>0$ are the \textit{Lam\'e parameters} and $\varrho$ represents the mass density. Upper dots indicate the derivative with respect to time, and we later in particular consider $T=\infty$. 
Using the \textit{Hooke tensor} $C_{ih}^{kl}=\lambda\delta_{ih}\delta_{kl}+\mu(\delta_{ik}\delta_{hl}+\delta_{il}\delta_{hk})$, $i,h,k,l=1, \dots, n$, we rewrite equation \eqref{Navier equation} in components as  
\begin{equation}\label{components with hooke tensor}
\sum_{h,k,l=1}^n\dfrac{\partial}{\partial x_h}\left( C_{ih}^{kl}\dfrac{\partial u_k}{\partial x_l}(\textbf{x},t)\right)-\varrho\ddot{u}_i(\textbf{x},t)=0,\quad \textbf{x}\in\Omega,\: t\in (0,T],\:i=1,\dots,n.
\end{equation}
%
We also define the traction $\textbf{p}=(p_1,\dots,p_n)^{\top}$ along $\Gamma$,
$$
p_i(\textbf{x},t)=p_i(\textbf{u})(\textbf{x},t)=\sum_{h,k,l=1}^nC_{ih}^{kl}\dfrac{\partial u_k}{\partial x_l}(\textbf{x},t)n_{\textbf{x} h},\quad \textbf{x}\in\Gamma,\:t\in(0,T],\: i=1,\dots,n,
$$
where $\textbf{n}_{\textbf{x}}$ is the unit normal vector to $\Gamma$ calculated in $\textbf{x}$, pointing from $\Omega$ to $\Omega'$. To emphasize that $\textbf{p}$ is defined on $\Gamma$, we also use the notation $\textbf{p}|_\Gamma$.
Equation \eqref{Navier equation} is equipped with initial vanishing conditions \eqref{initial vanishing condition} and a Dirichlet boundary condition on $\Gamma$, modelling a \textit{soft scattering} by the boundary:
\begin{align}
& \textbf{u}(\textbf{x},0)=\dot{\textbf{u}}(\textbf{x},0)=0, \quad \textbf{x}\in\Omega,\label{initial vanishing condition}\\
& \textbf{u}(\textbf{x},t)=\textbf{g}(\textbf{x},t),\qquad \quad  (\textbf{x},t)\in\Sigma:=\Gamma\times(0,T].\label{dirichlet condition}
\end{align}
In addition to \eqref{dirichlet condition}, also \textit{hard scattering} is considered, corresponding to a prescribed Neumann boundary condition 
\begin{align}
& \textbf{p}(\textbf{u})(\textbf{x},t)=\textbf{h}(\textbf{x},t),\quad (\textbf{x},t)\in\Sigma:=\Gamma\times(0,T].\label{neumann condition} 
\end{align}
We remark that the unknown $\textbf{u}$ can be written as the sum of two displacements $\textbf{u}=\textbf{u}_{\mathtt{P}}+\textbf{u}_{\mathtt{S}}$ {(Chapter V of \cite{Eringen1975}):} the term $\textbf{u}_{\mathtt{P}}$, called \textit{primary wave}, spreads in $\Omega$ with phase speed $c_{\mathtt{P}}=\sqrt{(\lambda+2\mu)/\varrho}>0$,
while $\textbf{u}_{\mathtt{S}}$, called \textit{secondary wave}, propagates in $\Omega$ with phase speed $c_{\mathtt{S}}=\sqrt{\mu/\varrho}>0$.
\subsection{Representation formula and direct boundary integral formulation}
If pure Dirichlet conditions \eqref{dirichlet condition} are imposed, to describe the unknown $\textbf{u}$ in $\Omega\times (0,T]$ we consider the following \textit{direct integral representation formula}:
\begin{align}\label{representation formula}
u_i(\textbf{x},t)=\sum_{j=1}^n\int_0^t\int_{\Gamma}G_{ij}(\textbf{x},\pmb{\xi};t,\tau)p_j(\pmb{\xi},\tau)d\Gamma_{\pmb{\xi}}d\tau-\sum_{j=1}^n\int_0^t\int_{\Gamma}
\sum_{h,k,l=1}^nC_{jh}^{kl}\dfrac{\partial G_{ik}}{\partial \xi_l
}(\textbf{x},\pmb{\xi};t,\tau)u_j(\pmb{\xi},\tau)n_{\pmb{\xi} h}d\Gamma_{\pmb{\xi}}d\tau,&\nonumber\\\quad(\textbf{x},t)\in\Omega\times (0,T],\:i=1,\dots,n,&
\end{align} 
where the traction $\textbf{p}$ is unknown on the boundary $\Gamma$. This formula is compactly written as 
$$\textbf{u}(\textbf{x},t)=\mathcal{V}\textbf{p}(\textbf{x},t)-\mathcal{K}\textbf{u}(\textbf{x},t),\quad \quad(\textbf{x},t)\in\Omega\times (0,T],$$
with the space-time \textit{single layer integral operator} $\mathcal{V} = (V_{ij})_{i,j = 1}^n$  and the \textit{double layer integral operator} $\mathcal{K} = (K_{ij})_{i,j = 1}^n$.


%
The second order tensor $\textbf{G} = (G_{ij})_{i,j = 1}^n$ in formula \eqref{representation formula} is the fundamental solution of the considered differential problem: in 2d
\begin{align}
G_{ij}(\textbf{x},\pmb{\xi};t,\tau):=& \dfrac{H[c_{\mathtt{P}}(t-\tau)-r]}{2\pi\varrho c_{\mathtt{P}}}\left\lbrace \dfrac{r_i r_j}{r^4}\dfrac{2 c^2_{\mathtt{P}}(t-\tau)^2-r^2}{\sqrt{c_{\mathtt{P}}^2(t-\tau)^2-r^2}}-\dfrac{\delta_{ij}}{r^2}\sqrt{c^2_{\mathtt{P}}(t-\tau)^2-r^2}\right\rbrace \nonumber\\
-& \dfrac{H[c_{\mathtt{S}}(t-\tau)-r]}{2\pi\varrho c_{\mathtt{S}}}\left\lbrace \dfrac{r_i r_j}{r^4}\dfrac{2 c^2_{\mathtt{S}}(t-\tau)^2-r^2}{\sqrt{c_{\mathtt{S}}^2(t-\tau)^2-r^2}}-\dfrac{\delta_{ij}}{r^2}\dfrac{c^2_{\mathtt{S}}(t-\tau)^2}{\sqrt{c^2_{\mathtt{S}}(t-\tau)^2-r^2}}\right\rbrace,\quad i,j=1,2,\label{fundamental solution}
\end{align}
while in 3d
\begin{align}
G_{ij}(\textbf{x},\pmb{\xi};t,\tau):=&  \frac{t-\tau}{4\pi\varrho r^2} \left(\frac{r_i r_j}{r^3} - \frac{\delta_{ij}}{r}\right)(H[c_{\mathtt{P}}(t-\tau)-r]-H[c_{\mathtt{S}}(t-\tau)-r]) \nonumber \\ &+ \frac{r_i r_j}{4\pi\varrho r^{3}} \left(c_{\mathtt{P}}^{-2}\delta(c_{\mathtt{P}}(t-\tau)-r)-c_{\mathtt{S}}^{-2}\delta(c_{\mathtt{S}}(t-\tau)-r) \right)\nonumber \\& + \frac{\delta_{ij}}{4\pi\varrho r c_{\mathtt{S}}^2} \delta(c_{\mathtt{S}}(t-\tau)-r), \qquad i,j=1,2,3.\label{fundamental solution3d}
\end{align}
Here we set the vector $\textbf{r}=(r_1,\dots,r_n)^\top=\textbf{x}-\pmb{\xi}=(x_1-\xi_1,\dots,x_n-\xi_n)^\top$, $r = |\textbf{r}|$, $H$ is the Heaviside function and $\delta$ the Dirac distribution.\\
Exploiting the Dirichlet boundary condition \eqref{dirichlet condition}, we obtain the following \textit{boundary integral equation}:
\begin{equation}\label{explicit BIE}
\mathcal{V}\pmb{\Phi}(\textbf{x},t)=\left(\mathcal{K}+\frac{1}{2}\right) \textbf{g}(\textbf{x},t),\quad(\textbf{x},t)\in\Sigma,
\end{equation}
with solution $\pmb{\Phi} = \textbf{p}|_\Gamma$. This solution can then be used in the representation formula \eqref{representation formula}.

In case of \textit{hard scattering problems}, namely with assigned condition \eqref{neumann condition}, the unknown displacement can be calculated in $\Omega$ considering the representation formula \eqref{representation formula} with the Hooke tensor applied:
\begin{align}
\sum_{h,k,l=1}^nC_{ih}^{kl}\dfrac{\partial u_k}{\partial x_l}(\textbf{x},t)n_{\textbf{x} h}=\sum_{j=1}^n\sum_{h,k,l=1}^n\int_0^t\int_{\Gamma}C_{ih}^{kl} \dfrac{\partial G_{jk}}{\partial x_l} (\textbf{x},\pmb{\xi};t,\tau)p_j(\pmb{\xi},\tau)n_{\textbf{x} h}d\Gamma_{\pmb{\xi}}d\tau &\nonumber\\
-\sum_{j=1}^n\sum_{h,k,l=1}^n\sum_{h',k',l'=1}^n\int_0^t\int_{\Gamma}C_{ih}^{kl}C_{jh'}^{k'l'} \frac{\partial G_{kk'}}{\partial x_l \partial \xi_{l'}} (\textbf{x},\pmb{\xi};t,\tau)u_j(\pmb{\xi},\tau) n_{\pmb{\xi} h'} n_{\textbf{x} h}d\Gamma_{\pmb{\xi}}d\tau,&\nonumber\\\quad(\textbf{x},t)\in\Omega\times (0,T],\:k=1,\dots,n,& \label{representation formula double layer}
\end{align} 
where the the displacement $\textbf{u}$ is unknown on the boundary $\Gamma$. The related compact notation is
$$\textbf{p}(\textbf{x},t)=\mathcal{K'}\textbf{p}(\textbf{x},t)-\mathcal{W}\textbf{u}(\textbf{x},t),\quad \quad(\textbf{x},t)\in\Omega\times (0,T],$$
where the operator $\mathcal{K'}= (K'_{ij})_{i,j = 1}^n$ is the \textit{adjoint double layer operator} and $\mathcal{W}= (W_{ij})_{i,j = 1}^n$ is the space-time \textit{hypersingular integral operator}.


Letting $\textbf{x}\in\Omega$ tend to $\Gamma$ in \eqref{representation formula double layer}, we obtain the time dependent boundary integral equation
\begin{equation}\label{explicit BIE hypersingular}
\mathcal{W}\pmb{\Psi}(\textbf{x},t)=\left(\mathcal{K'}-\frac{1}{2}\right) \textbf{h}(\textbf{x},t),\quad(\textbf{x},t)\in\Sigma, 
\end{equation}
with solution $\pmb{\Psi} = \textbf{u}|_\Gamma$ depending on the Neumann condition $\textbf{p}(\textbf{u})=\textbf{h}$ as prescribed in \eqref{neumann condition}. 
Therefore, our purpose is the numerical solution of the system \eqref{explicit BIE hypersingular} through the approximation of $\pmb{\Psi}$, which can then be used in the representation formula \eqref{representation formula}. \\

The Galerkin approximations to the integral equations \eqref{explicit BIE} and \eqref{explicit BIE hypersingular} are based on their weak formulations. The weak formulation of  \eqref{explicit BIE} {in the space-time cylinder $\Sigma$ is given in terms of the bilinear form
\begin{equation}\label{bilinearBD}
B_{D,\Sigma}(\pmb{\Phi},\pmb{ \tilde{\Phi}}) := \langle{\mathcal{V} \partial_t{\pmb{\Phi}}},\pmb{ \tilde{\Phi}}\rangle_{{L^2(\Sigma)}}.
\end{equation}}


\noindent \emph{Find $\pmb{\Phi}\in H^1_{\sigma}((0,T],\widetilde{H}^{-\frac{1}{2}}(\Gamma))^n$, such that}
\begin{equation}\label{energetic weak formulation}
{B_{D,\Sigma}(\pmb{\Phi},\pmb{ \tilde{\Phi}})}=\langle\partial_t{\left(\mathcal{K}+1/2\right){\textbf{g}}},\pmb{ \tilde{\Phi}}\rangle_{{L^2(\Sigma)}},
\end{equation}
\textit{for all $\pmb{\tilde{\Phi}}=(\tilde{\Phi}_1,\dots, \tilde{\Phi}_n)^{\top} \in H^1_{\sigma}((0,T],\widetilde{H}^{-\frac{1}{2}}(\Gamma))^n$.}\\
\\
Similarly, the weak formulation of  \eqref{explicit BIE hypersingular} is given {in terms of the bilinear form 
\begin{equation}\label{bilinearBN}
B_{N,\Sigma}(\pmb{\Psi},\pmb{ \tilde{\Psi}}) :=\langle{\mathcal{W} \partial_t{\pmb{\Psi}}},\pmb{\tilde{\Psi}}\rangle_{{L^2(\Sigma)}}.
\end{equation}} \\ 

\noindent \emph{Find $\pmb{\Psi}\in H^1_{\sigma}((0,T],\widetilde{H}^{\frac{1}{2}}(\Gamma))^n$,  such that}
\begin{equation}\label{hypersingeq}
{B_{N,\Sigma}(\pmb{\Psi},\pmb{ \tilde{\Psi}})} =\langle\partial_t{\left(\mathcal{K'}-1/2\right){\textbf{h}}},\pmb{ \tilde{\Psi}}\rangle_{{L^2(\Sigma)}},
\end{equation}
\textit{for all $\pmb{\tilde{\Psi}}=(\tilde{\Psi}_1,\dots, \tilde{\Psi}_n)^{\top} \in H^1_{\sigma}((0,T],\widetilde{H}^{-\frac{1}{2}}(\Gamma))^n$.}\\

{As in previous works the theoretical analysis requires a $\sigma$-dependent weight in the inner product for $T=\infty$, see \eqref{eq:bilinearform}. Then the} boundary integral equation \eqref{energetic weak formulation} for the Dirichlet problem {in the infinite space-time cylinder $\Gamma \times \mathbb{R}^+$} is well-posed, as follows from the coercivity and continuity of $\mathcal{V}$ shown in Appendix A, together with a proper setting of the functional spaces. Corresponding results for the hypersingular operator $\mathcal{W}$ in formulation \eqref{hypersingeq} go back to \cite{Becache1993, Becache1994}, where the 2d case is analyzed. The results easily generalize to 3d, for example, following the arguments in Appendix A.  

\begin{proposition}\label{wellposedness}{Let $\sigma>0$, $r \in \R$.}\\
a) Assume that ${\mathbf{g}} \in H^{r+1}_{\sigma}(\mathbb{R}^+,H^{\frac{1}{2}}(\Gamma))^n$. Then there exists a unique solution $\pmb{ \Phi} \in H^r_{\sigma}(\mathbb{R}^+,\widetilde{H}^{-\frac{1}{2}}(\Gamma))^n$  of \eqref{energetic weak formulation} and
\begin{equation}
\|\pmb{ \Phi}\|_{r, -\frac{1}{2}, \Gamma, \ast} \lesssim_\sigma \|{\mathbf{g}}\|_{r+1, \frac{1}{2}, \Gamma}\ .
\end{equation}
b) Assume that ${\mathbf{h}}\in H^{r+1}_{\sigma}(\mathbb{R}^+,H^{-\frac{1}{2}}(\Gamma))^n$. Then there exists a unique solution $\pmb{ \Psi} \in H^{r}_{\sigma}(\mathbb{R}^+,\widetilde{H}^{\frac{1}{2}}(\Gamma))^n$  of \eqref{hypersingeq} and
\begin{equation}
\|\pmb{ \Psi}\|_{r,\frac{1}{2}, \Gamma, \ast}\lesssim_\sigma \|{\mathbf{h}}\|_{r+1,-\frac{1}{2}, \Gamma} \ .
\end{equation}
{The proof for $r=0$ follows from Proposition \ref{DPbounds} and the mapping properties of $\mathcal{K}, \mathcal{K}'$, as found in \cite{chud}. The result for general $r$ then follows by the result for $r=0$ by differentiating the equation $r$ times, and complex interpolation for non-integer $r$. }
\end{proposition}

\section{Regularity of solutions to the Dirichlet problem}\label{regularitysection}


In this section we obtain precise results for the singular behaviour of the solution to the original initial-boundary value problem {of elastodynamics} with Dirichlet conditions \eqref{components with hooke tensor} - \eqref{dirichlet condition} for {two model geometries, the circular cone and the wedge.} The {decomposition results for the solution of} the differential equation lead {(by taking traces)} to decompositions also for the solutions of the integral equations in singular terms and more regular remainders. The problem with Neumann conditions can be dealt with by appropriate modifications; therefore this is omitted for brevity. {The analysis is local and therefore applies to both exterior and interior problems. While we treat arbitrary  polygonal domains in $\mathbb{R}^2$, an extension to arbitrary  polyhedral domains in $\mathbb{R}^3$ would require the extension of the analysis recalled in Appendix \ref{sec:polygonallame} to general corner singularities. Results in this generality are not currently available in the analysis literature and beyond the scope of this article.}\\

Subsection \ref{regularity2d} outlines the asymptotics of solutions near a vertex in a polygonal domain in $\mathbb{R}^2$, corresponding to the numerical experiments in Section \ref{sec:numer}. It includes a detailed discussion of the singular exponents for both the elastodynamic boundary problem and the scalar wave equation. {First, we consider the time-independent case (Proposition \ref{th1.1}). The results for the time-dependent case in a polygon} follow from the analysis for a wedge in $\mathbb{R}^3$, see  Corollary \ref{polygonlemma} in Subsection \ref{regularitywedge}, by explicit calculation of the singular exponents and the singular functions. Theorem \ref{wedgelemma} in  Subsection \ref{regularitywedge} presents the abstract asymptotic expansion for the solution in a wedge.  {It turns out that the singular exponents are the same as in the time-independent case, but the coefficients of the singular functions depend additionally on time. The behavior of the solution in a wedge is obtained by applying a partial Fourier transform along the edge and in time. Then the leading term of the resulting system \eqref{principalpartsystem} decouples into a 2d elastic system for the plane components of the elastodynamic field and into a scalar inhomogeneous wave equation \eqref{wavez2} for the $z$ component along the edge. In Theorem \ref{maintheorem} we therefore recall our results for the wave equation in a wedge. Then we apply Dauge's approach \cite{dauge} to the full system \eqref{principalpartsystem} with parameter $(\xi,\tau)$ by inserting the expansions \eqref{decomp1Lame}  and \eqref{decomp1dir}  of the time-independent, elliptic situation. In this way we obtain the expansion \eqref{eq44}  and via inverse Fourier transform the expansion \eqref{wedgeexpand} for the time-dependent problem.}

 The solution of the elastodynamic boundary problem in a circular cone is discussed in Subsection \ref{regularitycone}. {We consider the elastodynamic system in spherical coordinates. For fixed time $t$ we derive rotationally symmetric solutions \eqref{usym}.} Its asymptotic expansion is obtained in Theorem \ref{conefinalthm}.

We denote model geometries by $\mathbb{D}$. For ease of reference to the work of Plamenevski\v{\i} and coauthors, as well as to Appendix B and to \cite{hp}, this section adopts some of the notation from the analysis community, rather than the notation commonly found in numerical works. In particular, the  $\sigma>0$ from other sections in the article is here called $\gamma$, singular exponents {$\lambda_k$ are denoted by $i\lambda_k$}, and the definition of the Fourier transform and its inverse are interchanged.

\subsection{Behavior of solutions in a 2d sector}\label{regularity2d}


In the 2d case, for the inhomogeneous elastodynamic equation in a polygonal interior or exterior domain $ \Omega $, we introduce the radial and tangential components of $\bold{u}, u_r=r^{\nu^*} \varphi_r(\phi, t)$ and $ u_{\phi}=r^{\nu^*}\varphi_{\phi}(\phi,t)$ locally near a vertex of interior opening angle $\omega$. The system then becomes
\begin{align}\label{2dsystem1}
&\mu \partial^2_{\phi}\varphi_{r}+(\lambda+2\mu)((\nu^*)^2-1)\varphi_{r}+((\lambda+\mu)\nu^*-(\lambda+3\mu))\partial_{\phi}\varphi_{\phi}-r^{2-\nu^*}F_r = \varrho r^{2}\partial_t^2 u_r \ , \\& \label{2dsystem2}
(\lambda+2\mu) \partial^2_{\phi}\varphi_{\phi}+\mu((\nu^*)^2-1)\varphi_{\phi}+((\lambda+\mu)\nu^*+(\lambda+3\mu))\partial_{\phi}\varphi_{r}-r^{2-\nu^*}F_\phi = \varrho r^{2}\partial_t^2 u_\phi\ .
\end{align}
The time independent solutions of this system with right hand side $(F_r,F_\phi)=(0,0)$ are given by $(\cos(1+\nu^*)\phi, -\sin(1+\nu^*)\phi)^T, (\sin(1+\nu^*)\phi, \cos(1+\nu^*)\phi)^T,  (\cos (1-\nu^*)\phi, -\bar{\nu}\sin(1-\nu^*)\phi)^T,(\sin(1-\nu^*)\phi, \bar{\nu}\cos(1-\nu^*)\phi)^T $ with $ \bar{\nu}=\frac{3+\nu^*-4\nu}{3-\nu^*-4\nu}$ where
$\nu= \frac{\lambda}{2(\lambda+\mu)} $ is the Poisson number.

We briefly review the time independent problem with Dirichlet conditions $u_r(\pm\omega/2)=u_{\varphi}(\pm\omega/2)=0$: with arbitrary constants $A,B,C,D$ we obtain
\begin{equation*}
A \cos (1+\nu^*)\omega/2\pm B\sin(1-\nu^*)\omega/2 +C \cos(1-\nu^*)\omega/2\pm D\sin(1-\nu^*)\omega/2=0
\end{equation*}
\begin{equation*}
\mp A \sin(1+\nu^*)\omega/2+B\cos(1+\nu^*)\omega/2 \mp \bar{\nu}C \sin(1-\nu^*)\omega/2+\bar{\nu}D\cos(1-\nu^*)\omega/2=0\ ,
\end{equation*}
and therefore  the plane strain condition
\begin{equation}\label{singdirichlet}
\sin\nu^*\omega=\pm\frac{\bar{\nu}-1}{\bar{\nu} +1}\sin\omega \ \ \text{ with }\ \  \frac{\bar{\nu}-1}{\bar{\nu} +1}=\frac{\nu^*}{3-4\nu}.
\end{equation}
Since one can proceed analogously for Neumann boundary conditions one gets the following theorem for the time-independent problem.
\begin{proposition}\label{th1.1}
Let $\bold{f} \in H^{s-1}(\Omega)^2$ and $s>0$, $s  \notin \mathrm{Re}\ \nu^{*}_{jk} $ with $ \nu^{*}_{jk} $ as in \eqref{eq1.11},  \eqref{eq1.12}. Then the weak solution $\bold{u} \in H^{1}(\Omega)^2$ of the time-independent equations \eqref{2dsystem1}, \eqref{2dsystem2} admits with $C^{\infty}$ cut-off functions $\chi_{j}$ near the vertex $t_{j}$ with interior opening angle $\omega_j$ the decomposition
\begin{equation}\label{decomp1Lame} 
\bold{u}=\bold{u}_{0} + \sum_{\mathrm{Re}\ \nu^{*}_{jk}<s} a_{jk}^\ast \bold{S}^{*}_{jk}(r, \phi) \chi_j(r)
\end{equation}
with a regular part $  \bold{u}_{0}  \in H^{1+s}(\Omega)^2$, $a_{jk} \in \mathbb{C} $ and the singularity functions
\begin{equation}\label{eq1.10}
\bold{S}^{*}_{jk}(r, \phi)  = \begin{cases} r^{\nu^{*}_{jk} }\pmb{\varphi}_{jk}^*(\phi) \text{ for } \nu^{*}_{jk} \notin \mathbb{N}, \\
 r^{\nu^{*}_{jk} } \ln r\  \pmb{\varphi}_{jk}^*(\phi) + r^{\nu^{*}_{jk} }\pmb{\tilde{\varphi}}_{jk}^*(\phi)  \text{ for } \nu^{*}_{jk} \in \mathbb{N},\end{cases}
\end{equation}
Here the singular exponents $  \nu^{*}_{jk} \in \mathbb{C}  $ with $\mathrm{Re}\ \nu^{*}_{jk} > 0$ are solutions of the following equations depending on the kind of boundary conditions 
at the two sides meeting at the corner $ t_j $
\begin{equation}\label{eq1.11}
\text{Dirichlet: }  \sin\nu^{*}_{jk} \omega_j =\pm \textstyle{\frac{\nu^{*}_{jk}}{k^*}} \sin \omega_j
\end{equation}
\begin{equation}\label{eq1.12}
\text{Neumann: } \sin\nu^{*}_{jk} \omega_j =\pm \nu^{*}_{jk} \sin \omega_j 
\end{equation}
The functions $  \pmb{\varphi}_{jk} $ with the components $( \varphi_{jk})_r $ in r-direction and   $  (\varphi_{jk})_{\phi} $ in $ \phi $-direction  are of the form
\begin{equation}
(\varphi_{jk}^*)_{r}=A \cos(1+\nu^{*}_{jk})\phi + B \sin(1+\nu^{*}_{jk})\phi + C \cos(1-\nu^{*}_{jk})\phi +D  \sin(1-\nu^{*}_{jk})\phi
\end{equation}
\begin{equation}
(\varphi_{jk}^*)_{\phi}=-A \sin(1+\nu^{*}_{jk})\phi + B \cos(1+\nu^{*}_{jk})\phi - \gamma_{jk} C \sin(1-\nu^{*}_{jk})\phi + \gamma_{jk} D \cos(1-\nu^{*}_{jk})\phi
\end{equation}
with constants $A, B, C, D  \in \mathbb{C} $ depending on the type of boundary conditions at the corner and the constants $$ \gamma_{jk}=\textstyle{\frac{3+\nu^{*}_{jk}-4 \nu}{3- \nu^{*}_{jk}-4 \nu}},\ k^{*} =3-4 \nu\ .$$ 
\end{proposition}

As remarked in \cite{grisvardcr}, p.~73, for Dirichlet boundary conditions there exist two leading real roots of the equation \eqref{eq1.11} in $(0,1)$.

\begin{remark}\label{angledependence}
For a crack, i.e.~$\omega_j=2\pi$ for Dirichlet and Neumann boundary conditions $\nu^{*}_{j1}=1/2 $.

More generally, we can use \eqref{eq1.11} to study the leading singular exponents for the solution of the Dirichlet problem near an angle $\omega$ when $\omega \to 0$, respectively $\omega \to 2\pi$. 

{To do so, note that for the leading singular exponent $\nu^{*} = \nu^{*}_{j1}$}
\begin{equation}
\sin\nu^* \omega=\frac{\nu^*}{k^{\ast}}\sin\omega = \frac{\nu^* \omega}{k^\ast} + o(\omega) 
\end{equation}
for $\omega \to 0$, or $$\frac{\sin\nu^* \omega}{\nu^* \omega} \to \frac{1}{k^\ast}. 
$$
We conclude $\nu^* = \frac{c}{\omega} + O(1)$, where $c$ satisfies $\frac{\sin c}{c} = \frac{1}{k^\ast}$.

For the corresponding exterior angle, $\omega = 2 \pi - \varepsilon$ with $\varepsilon \to 0$, we set $\nu^* = \frac{1}{2} + \tilde{\nu}(\varepsilon)$. Then 
$\sin\nu^*\omega= \sin\left((\frac{1}{2}+\tilde{\nu}(\varepsilon))(2\pi - \varepsilon)\right)$, and Taylor expanding for $\varepsilon, \tilde{\nu}(\varepsilon) \to 0$ leads to 
$$\sin\nu^*\omega = -2\pi \tilde{\nu}(\varepsilon) + \frac{\varepsilon }{2} + o(\varepsilon)\ .$$ 
On the other hand, from equation \eqref{eq1.11} $\sin\nu^*\omega = \frac{\nu^*}{k^{\ast}}\sin\omega = - \frac{\nu^*} {k^{\ast}}\varepsilon 
+ o(\varepsilon)$, so that 
$-2\pi \tilde{\nu}(\varepsilon) + \frac{\varepsilon }{2} = - \frac{1}{2k^{\ast}}\varepsilon 
+ o(\varepsilon)$,
or $\tilde{\nu}(\varepsilon) = \frac{\varepsilon}{4\pi}\left(1 + \frac{1}{k^\ast}\right)+ o(\varepsilon)$ and $$\nu^* = \frac{1}{2}+ \frac{\varepsilon}{4\pi}\left(1 + \frac{1}{k^\ast}\right)+ o(\varepsilon)\ .$$

Figure \ref{fig:mesh_expected_exponent} numerically illustrates $\nu^\ast$ as a function of $\omega$, when $\lambda=2$, $\mu=1$ and $\rho=1$. It confirms the above analysis.
\end{remark}

In the next section we also require a corresponding description of the singularities for the scalar wave equation \cite{hp}
$$
 \varrho\partial_t^2 u = (\partial_x^2 + \partial_y^2) u - F \ .
$$
in $\Omega$ with Dirichlet or Neumann boundary conditions. 

Again, we first describe the singularities for the well-studied time independent problem. In this case near the vertex $t_{j}$ with interior opening angle $\omega_j$ the weak solution $u$ admits   the decomposition
\begin{equation}\label{decomp1dir} 
u=u_{0} + \sum_{ \nu_{jk}<s} a_{jk} S_{jk}(r, \phi) \chi(r)
\end{equation}
with $C^{\infty}$ cut-off functions $\chi_{j}$, a regular part $  u_{0}  \in H^{1+s}(\Omega), a_{jk} \in \mathbb{C} $ and the singularity functions
\begin{equation}\label{eq1.10dir}
S_{jk}(r, \phi) = \begin{cases} r^{\nu_{jk} }\varphi_{jk}(\phi) \text{ for } \nu_{jk} \notin \mathbb{N}, \\
 r^{\nu_{jk} } \ln r\  \varphi_{jk}(\phi) + r^{\nu_{jk} }{\tilde{\varphi}}_{jk}(\phi)  \text{ for } \nu_{jk} \in \mathbb{N},\end{cases}
\end{equation}
where $\nu_{jk} = \frac{k\pi}{\omega_j}$. For Dirichlet boundary conditions $\varphi_{jk,D} = \sin(\nu_{jk}\phi)$, $k \in \mathbb{N}$, while for Neumann boundary conditions $\varphi_{jk,N} = \cos(\nu_{jk}\phi)$, $k \in \mathbb{N}_0$.


\subsection{Behavior of solutions in a wedge}\label{regularitywedge}

The behavior of solutions in a wedge of opening angle $\omega$, {$\mathbb{D} = \mathbb{K} \times \mathbb{R}$ with $\mathbb{K}= \{(r,\phi) : r>0, \ \phi \in (0,\omega)\}$,} generalizes the discussion in Section \ref{regularity2d} from dimension $n=2$ to $n=3$. {As long as we discuss this model geometry with only one non-smooth subset $\{\mathbf{0}\}\times \mathbb{R}$ of $\partial \mathbb{D}$, we omit the index numbering the non-smooth subsets ($j$ in Subsection \ref{regularity2d}).}

We here consider the elastodynamic system \eqref{Navier equation} in the space-time cylinder {$\mathcal{Q} = \mathbb{D} \times \mathbb{R}$}  with a right hand side $\bold{f}$
\begin{equation}
L(\partial_x, \partial_y, \partial_z, \partial_t)\bold{u}:=-(\lambda+\mu)\nabla(\nabla \cdot \textbf{u})-\mu\Delta\textbf{u}+\varrho\ddot{\textbf{u}}=\bold{f}\ .
\end{equation}

Applying a partial Fourier transform $\mathcal{F}_{(z,t)\mapsto (\xi,\tau)}$  along the edge and in time, the equation becomes 
\begin{equation}\label{eq1.42}
 L(\partial_x, \partial_y,-i\xi, -i\tau)\hat{\bold{u}}(x,y,\xi,\tau)=\hat{\bold{f}}(x,y, \xi,\tau) ,
\end{equation}
posed in the sector $\mathbb{K}$. 

More precisely, the operator $L$ here takes the form
\begin{align}\nonumber
&L(\partial_x, \partial_y,\partial_z, \partial_t) =\\ \label{Fourierwedgesystemtime} & \textstyle{\begin{pmatrix}-(\lambda+2\mu) \partial_x^2 - \mu (\partial_y^2 + \partial_z^2) + \varrho \partial_t^2& -(\lambda + \mu) \partial_x \partial_y & -(\lambda+\mu) \partial_x \partial_z \\  -(\lambda+\mu) \partial_x \partial_y& -(\lambda+2\mu) \partial_y^2 - \mu (\partial_x^2 + \partial_z^2) + \varrho \partial_t^2&  -(\lambda + \mu) \partial_y \partial_z \\ -(\lambda + \mu) \partial_x \partial_z & -(\lambda + \mu) \partial_y \partial_z & -(\lambda+2\mu) \partial_z^2 - \mu (\partial_x^2 + \partial_y^2) + \varrho \partial_t^2\end{pmatrix}} \ . 
\end{align}
The Fourier transform $\mathcal{F}_{(z,t)\mapsto (\xi,\tau)}$ transforms the system into 
\begin{align}\nonumber
&L(\partial_x, \partial_y,-i\xi, -i\tau) =\\& \label{Fourierwedgesystem}\textstyle{\begin{pmatrix}-(\lambda+2\mu) \partial_x^2 - \mu \partial_y^2 + \mu \xi^2 - \varrho \tau^2& -(\lambda + \mu) \partial_x \partial_y & i(\lambda+\mu) \xi \partial_x \\  -(\lambda+\mu) \partial_x \partial_y& -(\lambda+2\mu) \partial_y^2 - \mu \partial_x^2 +\mu\xi^2 - \varrho \tau^2&  i(\lambda + \mu) \xi \partial_y  \\ i(\lambda + \mu) \xi \partial_x  & i(\lambda + \mu) \xi \partial_y &  -\mu (\partial_x^2 + \partial_y^2) +(\lambda+2\mu) \xi^2- \varrho \tau^2\end{pmatrix}}\ . 
\end{align}
With $\zeta^2 = (\mu \xi^2 - \varrho \tau^2)^{-1}$, we obtain
\begin{align} 
&M(\partial_x,\partial_y,\xi, \tau) = \zeta^2 L(\zeta^{-1} \partial_x,\zeta^{-1}\partial_y,-i\xi, -i\tau) = L_0 + L_1 + L_2\nonumber \\& =\begin{pmatrix}-(\lambda+2\mu) \partial_x^2 - \mu \partial_y^2 & -(\lambda + \mu) \partial_x \partial_y & 0 \\ -(\lambda+\mu) \partial_x \partial_y& -(\lambda+2\mu) \partial_y^2 - \mu \partial_x^2 & 0  \\ 0  & 0 & - \mu (\partial_x^2 + \partial_y^2) \end{pmatrix}\nonumber \\&\qquad +\begin{pmatrix}0&0 & i(\lambda+\mu) \xi \zeta \partial_x \\ 0& 0&  i(\lambda + \mu) \xi \zeta \partial_y  \\ i(\lambda + \mu) \xi \zeta \partial_x  & i(\lambda + \mu) \xi \zeta \partial_y & 0\end{pmatrix}+\begin{pmatrix}1 & 0 &0 \\0&   1& 0  \\ 0  & 0 & \zeta^2 [(\lambda+2\mu) \xi^2- \varrho \tau^2]\end{pmatrix}\ .\label{principalpartsystem}
\end{align}

The principal part $  L_0 $ of  the operator $M$ in \eqref{principalpartsystem} is 
\begin{equation}\label{eq1.46}
L_0 := -\begin{pmatrix}  \Delta^{*}_{x,y} &  0 \\0 & \mu\Delta_{x,y}\end{pmatrix},
\end{equation}
and \eqref{eq1.42} becomes \begin{equation}\label{Meq}M \bold{v} = \zeta^2 \hat{\bold{f}} =:\bold{k}_{(\zeta)}\ .\end{equation}
{We study this equation in rescaled variables $\bold{v}(\tilde{x},\tilde{y}) = \hat{\bold{u}}(x,y,\xi,\tau)$, with $(\tilde{x},\tilde{y}) = \zeta^{-1}(x,y)$ and $\tilde{r}= |(\tilde{x},\tilde{y})| = r/\zeta$, and in this way obtain uniform assertions for $\hat{\bold{u}}$ in $\zeta$ below.}

The leading term $ L_0  $ decomposes  into the Laplace operator $\Delta_{x,y}$ (in direction of the edge) and into the 
two-dimensional elasticity operator $\Delta^{*}_{x,y}$ on the cross section $\mathbb{K}$. $L_0$ decouples the equations for the components $(v_{x}, v_{y})$ and $v_{z}$ into a 2d elastic system for the plane components of ${\mathbf{v}}$, discussed in Section  \ref{regularity2d}, and a scalar problem for the $z$-component, both posed in the sector $\mathbb{K}$.  


The singularities for $M$ result from the singularities of $ L_0 $ plus correction terms of higher regularity, which come from the differential
operators of lower order. {For time-independent problems this is shown in Proposition 16.8 and equation (5.9) in \cite{dauge}, as well as in \cite{disspetersdorff}.}

For the Dirichlet problem 
the singularities for $L_0$ follow directly from Proposition \ref{th1.1}, giving for $\hat{\bold{u}}$ the expansion \eqref{eq1.123} for $p=0$. Here the singularities $\bold{S}_{k,0} = (0,0,S_k)$, $\bold{S}^*_{k,0}$ are those in  \eqref{eq1.10dir}, respectively \eqref{eq1.10}. {(Recall that we omit the index $j$ numbering the vertices in Subsection \ref{regularity2d}.)}

The singularities for the whole operator $M$ are then obtained as follows. First, one moves the lower-order terms in the operator to the right hand side of the differential equation and repeats this process.

The additional correction terms  $\bold{S}_{k,\ell}$, $\bold{S}^*_{k,\ell}$  for $\ell>0$ are defined recursively as 
\begin{equation}\label{itersing}
\bold{S}_{k,1} = -R L_1 \bold{S}_{k,0},\qquad \bold{S}_{k,\ell} = -R L_2 \bold{S}_{k,\ell-2}-R L_1 \bold{S}_{k,\ell-1} \qquad (\ell>1),
\end{equation}
and correspondingly for $\bold{S}_{k,\ell}^*$. Here $R = (R_\Delta^*, R_\Delta)$ is the solution operator for $\Delta$, respectively $\Delta^*$.

More explicitly, we obtain
$$L_1 \bold{S}_{k,0} = \begin{pmatrix} i (\lambda + \mu) \xi \zeta \nabla S_{k}\\ 0\end{pmatrix}\ ,$$
and we make the ansatz $\bold{S}_{k,1} = (\bold{B}_{k,1}, A_{k,1})$ with a scalar function $A_{k,1}$ in the edge direction and a two-component vector $\bold{B}_{k,1}$ for the components in the cross section.

Then  $$A_{k,1} = -(\lambda+\mu)R_\Delta 0 = 0,\qquad \bold{B}_{k,1} = -(\lambda+\mu)\xi \zeta R_\Delta^*\nabla \bold{S}_{k}\ .$$

Corresponding formulas can be derived for the higher singular functions $\bold{S}_{k,\ell}$. They satisfy $\bold{S}_{k,\ell}^\ast(r,\phi) \sim r^{\nu_k^*+\ell} \pmb{\varphi}_{k,\ell}^\ast(\phi)$, respectively $\bold{S}_{k,\ell}(r,\phi) \sim r^{\nu_k+\ell} \pmb{\varphi}_{k,\ell}(\phi)$, with $\pmb{\varphi}_{k,0}= (0,0,\varphi_{k})$ from \eqref{eq1.10dir}. This is abstractly described in \cite{matyu}, p.~495, relying on Proposition 3.9 in \cite{nazarov}, and explicit formulas are not easily derived for the wedge. While only the leading terms are given explicitly, and confirmed in our numerical experiments, the general structure of the singular functions is sufficient for the error analysis in Section \ref{approxsection}.

For the time-dependent situation we first consider the third equation, for $u_z$ in \eqref{Fourierwedgesystemtime}, which up to operators of lower order in $x$ and $y$ is simply the wave equation {in the wedge geometry $\mathbb{D} \times \mathbb{R}$. As above, $\mathbb{D} = \mathbb{K}  \times \mathbb{R} \subset\mathbb{R}^3$ and $\mathbb{K} $ is the sector $\{(r,\phi) : r>0, \ \phi \in (0,\omega)\}$.} {Using \eqref{Fourierwedgesystem} in cylindrical coordinates and taking the Fourier transform $\mathcal{F}_{(z,t)\mapsto (\xi,\tau)}$, we obtain} 
\begin{equation}\label{wavez2}-\Delta_{x,y} \hat u_z(r,\phi,\xi,\tau) + \left(\frac{\lambda+2\mu}{\mu}\xi^2 - \frac{\varrho}{\mu} \tau^2\right) \hat u_z(r,\phi,\xi,\tau) = \mu^{-1}\hat{k}\ ,\end{equation}
{up to lower order terms.} Here $k$ is the third component of $\bold{k}_{(\zeta)}$. To find  the behavior of the solutions of \eqref{wavez2}, after rescaling $\tau,\xi$ it suffices to study the wave equation \begin{equation}\label{auxwave}-\Delta_{x,y} \hat{u}_z - (\tau^2-\xi^2)\hat{u}_z = \hat{k}\ .\end{equation} 
The approach in \cite{hp} makes an ansatz $$\hat{u}_z  =r^{i\lambda_{-k}}\varphi_{-k}(\phi) \rho_{-k}(r\eta)= r^{i\lambda_{-k}} \sin\left(i\lambda_{-k} \phi \right) \rho_{-k}(r\eta)$$ with 
$\eta^2 =  \xi^2 - \tau^2$ and reduces \eqref{wavez2} for $\hat{k}=0$ to a Bessel differential equation:
$$r^2 \eta^2 \rho_{-k}''(r\eta) + \left(2i\lambda_{-k}+1\right) r\eta \rho_{-k}'(r\eta) + r^2\eta^2 \rho_{-k}(r\eta) = 0\ .$$
For the edge with Dirichlet or Neumann boundary conditions, $i\lambda_{-k} = \frac{\pi k}{\omega}$.
The  solution of the Bessel differential equation can be given explicitly in terms of a Bessel function as in \cite{hp}:
\begin{equation*}
 \rho_{-k}(t \tau) = c \  (r \tau)^{i\lambda_{-k}}  K_{i\lambda_{-k}}(i r \tau).
\end{equation*}
The resulting asymptotic expansion obtained {for $\rho_{-k}(t\tau)$ in Theorem 14 from \cite{graded}} corresponds to the expansion of $\hat{u}_z$. Theorem \ref{conetheorem} describes the general singular behavior in the space-time cylinder $\mathcal{Q} = \mathbb{D}\times \mathbb{R}$. The above arguments lead to the following more precise expansion in Theorem \ref{maintheorem} for the wedge {$\mathbb{D}= \mathbb{K}\times \mathbb{R}$}, involving the following special solutions $w_{-k,B}$ of the Dirichlet {($B=D$)} or Neumann  {($B=N$)} problem with $\varphi_{k,B}$ as at the end of Section \ref{regularity2d} (see \cite[(3.5)]{kokotov3}, respectively \cite[(4.4)]{kokotov}): 
$$w_{-k,B}(r,\phi, \xi, \overline{\tau}) = \frac{2^{1-i\lambda_{k,B}}}{\Gamma(i\lambda_{k,B})}(ir\sqrt{-|\xi|^2+\overline{\tau}^2})^{i\lambda_{k,B}} K_{i\lambda_{k,B}}(ir\sqrt{-|\xi|^2+\overline{\tau}^2}) r^{-i \lambda_{k,B}} \varphi_{k,B}(\phi) \ .$$ 
{We recall the following theorem for the wave equation in the wedge, which gives an expansion of the solution in terms of singular functions (Theorem 14 in \cite{graded}, $n=3$, $d=1$ in their notation).}
\begin{theorem}[\cite{graded}]\label{maintheorem}
Let $\beta\leq 1$ and $\gamma>0$, $(f,g) \in \mathcal{R}H_{\beta, q}(\mathcal{Q}, \gamma)$,  
and assume that the line $\mathrm{Im} \ \lambda = \beta-1$ does not intersect the spectrum of $\mathcal{A}_B$ from \eqref{pencildef}. 
Further, define $$J_{\beta,B} = \left\{k: 0> \mathrm{Im}\ \lambda_{k,{B}} > \beta-1\right\}\cup A\ ,$$ with $A = \{0\}$ for $\beta\leq 0$ and $A=\emptyset$ otherwise.\\
 If $u$ is a strong solution to the inhomogeneous wave equation with homogeneous Dirichlet or Neumann boundary conditions ($B=D$, resp.~$N$), 
then near the edge $u$ is of the form 
$$\sum_{j \in J_{\beta,B}} \Gamma(1+\nu_{j,B}) r^{i \lambda_{j,B}} \varphi_{j,B}(\phi)\sum_{m=0}^{N_j}\frac{(\partial_t^2-\Delta_z)^m(ir)^{2m}}{2^{2m}m! \Gamma(m+\nu_{j,B}+1)} \left(\mathcal{F}^{-1}_{(\xi,\tau) \to (z,t)}c_{j,B}(r, \phi,{\xi,\tau})\right) + \check{u}_0(r,\phi,z,t)\ ,$$
assuming that {$i \lambda_{j,B} = \nu_{j,B} = \frac{\pi}{\omega} \not \in \mathbb{N}_0$}. Here $N_j$ sufficiently large, and depending on the boundary conditions 
$$c_{j,D}(\xi,\tau) = \langle \hat{f}(\cdot, \xi, \tau), w_{-j,D}(\cdot, \xi, \overline{\tau})\rangle_{L^2(\mathbb{K})} + (\hat{g}(\cdot, \xi,\tau), \partial_\nu w_{-j,D}(\cdot, \xi, \overline{\tau}))_{L^2(\partial \mathbb{K})};$$ 
$$c_{j,N}(\xi,\tau) = \langle \hat{f}(\cdot, \xi, \tau), w_{-j,N}(\cdot, \xi, \overline{\tau})\rangle_{L^2(\mathbb{K})} + (\hat{h}(\cdot, \xi,\tau), w_{-j,N}(\cdot, \xi, \overline{\tau}))_{L^2(\partial \mathbb{K})}.$$ 
{The regularity of $c_{j,B}$ is determined by the right hand side, and the remainder $\check{u}_0$ is less singular than $u$,} in the sense that  $\|\check{u}_0\|_{DV_{\beta, q}(\mathcal{Q}; \gamma)} \lesssim \|(f,g)\|_{\mathcal{R}H_{\beta, q}(\mathcal{Q}, \gamma)}$ for the Dirichlet problem, with analogous results in the Neumann case. We refer to Appendix B for the definition of the weighted spaces $DV_\beta(\mathcal{Q}, \gamma), \mathcal{R}H_{\beta, q}(\mathcal{Q}, \gamma)$.  If $i \lambda_{j,B} \in \mathbb{N}_0$, additional terms $r^{i \lambda_{j,B}} \log(r)$ appear. 
\end{theorem}

{While Theorem \ref{maintheorem} is for homogeneous Dirichlet or Neumann boundary conditions, it is readily translated into inhomogeneous boundary conditions, as for elliptic problems \cite[Section 5]{petersdorff2}: For Dirichlet boundary conditions $u = g$, choose an extension $\widetilde{g}$ in the domain with Dirichlet trace $g$. The function $U = u-\widetilde{g}$ then satisfies homogeneous Dirchlet boundary conditions  $U=0$. 
Theorem \ref{maintheorem} then assures an asymptotic expansion of $U$, and therefore of $u = U + \widetilde{g}$.}\\
{An analogous argument applies to Neumann boundary conditions, using an extension $\widetilde{g}$ with the given Neumann trace.}

In particular, we mention the leading term of the expansion for the Dirichlet problem:
\begin{corollary}
Let $\gamma>0$, $\beta<1$, and  assume that $i\lambda_1 = \frac{\pi}{\omega}$ is the only eigenvalue in the strip $\beta -1 \leq \mathrm{Im}\ \lambda \leq0$. Then for $(f,g) \in \mathcal{R}V_\beta(\mathcal{Q},\gamma)$ the solution $u \in DV_1(\mathcal{Q},\gamma)$ of the inhomogenous boundary problem admits the representation 
$$u(r,\phi,z,t) = \chi(r) r^{\pi/\omega} \varphi(\phi) Xc(r,\phi,z,t) + u_0(r,\phi,z,t),$$ 
where $u_0 \in DV_\beta(\mathcal{Q},\gamma)$, $\gamma>\gamma_0$, $\chi$ is a cut-off function, $X$ as in \eqref{Xdef}, and \begin{equation}\label{cdef}c(r,\phi,z,t) = \int \left\{\langle f(t'), W(t-t')\rangle_{\mathbb{D}} + \langle g(t'), \partial_\nu W(t-t')\rangle_{\partial \mathbb{D}}\right\}dt'\ .\end{equation}
Here, $$W(r,\phi,z,t) = \mathcal{F}^{-1}_{(\xi,\tau)\to (z,t)} w(r, \phi,\xi,\tau)$$ and $w$ solves   \eqref{auxwave} with Dirichlet boundary condition $w|_{\partial \mathbb{K}} =0$.
\end{corollary}
Near the edge, the function $w$ behaves like $r^{\frac{\pi}{\omega}}\varphi(\phi)$ from \eqref{dualfunctions}. 


Now the expansions \eqref{decomp1Lame} and \eqref{decomp1dir} can be applied to $\bold{v}$ in \eqref{Meq}, yielding with $(\tilde{x},\tilde{y}) = \zeta^{-1}(x,y)$ and $\tilde{r}= |(\tilde{x},\tilde{y})| = r/\zeta$, 
\begin{align}
 \bold{v}&= \bold{v}_{0}+ \chi(\tilde{r}) \Big( \sum_{\mathrm{Re}\ \nu_{k}<s} a_{k,(\zeta)} \sum_{0 \leq \ell<s-\mathrm{Re}\ \nu_{k}}  \bold{S}_{k,\ell}(\tilde{r}, \phi)+ \sum_{\mathrm{Re}\ \nu^{*}_{k}<s} a^{*}_{k, (\zeta)} \sum_{0 \leq \ell<s-\mathrm{Re}\ \nu^{*}_{k}} \bold{S}^{*}_{k,\ell}(\tilde{r}, \phi) \Big)\label{eq1.123} 
\end{align}
with $\bold{v}_{0} \in H^{s+1}(K)^n$,  $a_{k,(\zeta)}, a^{*}_{k,(\zeta)}\in \mathbb{C}$ for fixed $\zeta$. {Here, as before, the singular functions $\bold{S}_{k,\ell}$
are to leading order those of the wave equation, in the third component $(0,0,S_{k})$, while $\bold{S}_{k,\ell}^*$ are to leading order those of the 2d elastostatic system \eqref{eq1.10dir}.}  
 In the following we consider the case of large $\zeta$ (see \cite{dauge}). We transform \eqref{eq1.123} back in the coordinates $ \xi,x,y$. When $\bold{S}_{k,\ell}$ and $\bold{S}^*_{k,\ell}$ have no log term, then $\bold{S}_{k,\ell}(\zeta^{-1}r,\phi) = \zeta^{-\nu_k - \ell}\bold{S}_{k,\ell}(r,\phi)$ and correspondingly for $\bold{S}^*_{k,\ell}$. Using that $\hat{c}_{k,(\zeta)} = \zeta^{-\nu_k} a_{k,(\zeta)}$ and $\hat{c}_{k,(\zeta)}^\ast = \zeta^{-\nu_k^*} a_{k,(\zeta)}^*$ we obtain $a_{k,(\zeta)} \sum_{0 \leq \ell<s-\mathrm{Re}\ \nu_{k}} \bold{S}_{k,\ell}(\zeta^{-1}r,\phi) = \sum_{0 \leq \ell<s-\mathrm{Re}\ \nu_{k}} \zeta^{-\ell}\hat{c}_{k,(\zeta)} \bold{S}_{k,\ell}(r,\phi)$ and correspondingly for $\bold{S}^*_{k,\ell}$. With $\bold{v}(\tilde{x},\tilde{y}) = \hat{\bold{u}}(x,y,\xi,\tau)$ and $\bold{v}_0(\tilde{x},\tilde{y}) = \hat{\bold{u}}_0(x,y,\xi,\tau)$ we obtain
\begin{align}
\hat{\bold{u}}(x,y,\xi,\tau)&=\hat{\bold{u}}_{0}(x,y,\xi,\tau)+ \chi(r/\zeta) \Big( \sum_{\mathrm{Re}\ \nu_{k}<s} a_{k,(\zeta)} \sum_{0 \leq \ell<s-\mathrm{Re}\ \nu_{k}}  \bold{S}_{k,\ell}(\zeta^{-1}r, \phi) \nonumber\\ & \quad+ \sum_{\mathrm{Re}\ \nu^{*}_{k}<s} a^{*}_{k,(\zeta)} \sum_{0 \leq \ell<s-\mathrm{Re}\ \nu^{*}_{k}} \bold{S}^{*}_{k,\ell}(\zeta^{-1}r, \phi) \Big)\nonumber \\& = \hat{\bold{u}}_{0}(x,y,\xi,\tau)+ \chi(r/\zeta) \Big( \sum_{\mathrm{Re}\ \nu_{k}<s} \sum_{0 \leq \ell<s-\mathrm{Re}\ \nu_{k}}  \zeta^{-\ell}\hat{c}_{k,(q)}\bold{S}_{k,\ell}(\zeta^{-1}r, \phi) \nonumber\\ & \quad+ \sum_{\mathrm{Re}\ \nu^{*}_{k}<s} \sum_{0 \leq \ell<s-\mathrm{Re}\ \nu^{*}_{k}} \zeta^{-\ell}\hat{c}^*_{k,(\zeta)} \bold{S}^{*}_{k,\ell}(\zeta^{-1}r, \phi) \Big) \ .\label{eq44}
\end{align}
In the notation of Appendix B, we obtain by applying the inverse Fourier transform $\mathcal{F}_{(\xi,\tau) \mapsto (z,t)}^{{-1}}$
\begin{align}\bold{u}(x,y,z,t)&=\bold{u}_{0}(x,y,z,t)+ \sum_{\mathrm{Re}\ \nu_{k}<s} \sum_{0 \leq \ell<s-\mathrm{Re}\ \nu_{k}}
(X c_{k,\ell})(\bold{y},z,t) \bold{S}_{k,\ell}(r, \phi)\nonumber \\
&\qquad + \sum_{\mathrm{Re}\ \nu^{*}_{k}<s} \sum_{0 \leq \ell<s-\mathrm{Re}\ \nu_{k}}
(X c_{k,\ell}^*)(\bold{y},z,t) \bold{S}^{*}_{k,\ell}(r, \phi).
\label{wedgeexpand}
\end{align}
Here, $\hat{c}_{k,\ell} = \zeta^{-\ell}\hat{c}_{k,(\zeta)}$, $\hat{c}_{k,\ell}^\ast = \zeta^{-\ell}\hat{c}_{k,(\zeta)}^\ast$, with $\zeta^2 = (\mu \xi^2 - \varrho \tau^2)^{-1}$ as before. As in Appendix B, the smoothing operator $X$ is given by $$X c(\bold{y},z,t) = \mathcal{F}^{-1}_{(\xi,\tau) \to (z,t)} \chi({\sqrt{|\xi|^2 + |\tau|^2}}\bold{y}) \hat{c}(\xi,\tau)$$ for $\hat{c} = \hat{c}_{k,\ell}, \hat{c}_{k,\ell}^*$.
The regularity of $\bold{u}_0$ and of the edge functions $ c_{k,p}, c_{k,p}^{*} $ follows corresponding to the case of the scalar wave equation in Theorem \ref{maintheorem}, generalizing the results of \cite{hp} to elastodynamics.

Altogether, we obtain the following theorem, formulated corresponding to Theorem \ref{conetheorem} in Appendix B.

\begin{theorem}\label{wedgelemma} 
Let $\gamma>0$, $q \in \mathbb{N}_0$, $\beta \in (\beta_{r+1},\beta_r)$ with $0<\beta_r-\beta<1$, 
$(\bold{f},\bold{g}) \in \mathcal{R}V_{\beta,q}(\mathcal{Q},\gamma)$ and assume that the orthogonality condition 
\eqref{orthorelation} holds for all $\nu_k, \nu_k^*$ with $\mathrm{Re}\ \nu_k, \mathrm{Re}\ \nu_k^* \in [1-\beta_r, 1-\beta_1]$.  Then the solution of the initial-boundary value problem \eqref{components with hooke tensor} - \eqref{dirichlet condition} admits the expansion
\eqref{wedgeexpand} {in terms of the singular functions $\bold{S}_{k,\ell}$, $\bold{S}^*_{k,\ell}$ constructed from \eqref{eq1.10dir}, respectively \eqref{eq1.10}.} 
{Further, in \eqref{wedgeexpand}} $s<\mathrm{min}\{\mathrm{Re}\ \nu_k, \mathrm{Re}\ \nu_k^\ast\}+{\ell}+1+\beta$ for all $k$ and  $\bold{u}_0 \in DV_{\beta,q}(\mathcal{Q}, \gamma)$. 
\end{theorem}

{By considering the coordinate $z$ along the edge as a parameter, we recover and refine the results for  polygonal domains in 2d from Section \ref{regularity2d}. More precisely, we obtain for the solution of the elastodynamic problem \eqref{2dsystem1}-\eqref{2dsystem2}:}
\begin{corollary}\label{polygonlemma} 
Let $\gamma>0$, $q \in \mathbb{N}_0$,  $\beta \in (\beta_{r+1},\beta_r)$ with $0<\beta_r-\beta<1$, 
$(\bold{f},\bold{g}) \in \mathcal{R}V_{\beta,q}(\mathcal{Q},\gamma)$ and assume that the orthogonality condition 
\eqref{orthorelation} holds for all $\nu_k^*$ with $\mathrm{Re}\ \nu_k^* \in [1-\beta_r, 1-\beta_1]$.   Then in the neighborhood of a vertex $t_j$ with interior opening angle $\omega_j$ the solution to \eqref{components with hooke tensor} - \eqref{dirichlet condition} admits the expansion 
\begin{equation}\label{polygexpand}\bold{u}(x,y,t) =\bold{u}_0(r,\phi,t) + \sum_{k,\ell} (Xc_{k,\ell}^*)(t) \bold{S}_{k,\ell}^*(r,\phi),\end{equation}
where $s<\mathrm{min}\{\mathrm{Re}\ \nu_k, \mathrm{Re}\ \nu_k^\ast\}+{\ell}+1+\beta$ for all $k$ and $\bold{u}_0 \in DV_{\beta,q}(\mathcal{Q}, \gamma)$.
\end{corollary}
Corollary \ref{polygonlemma} recovers Theorem 3.5 in \cite{mueller2}; note that the sum in $k$ in both \eqref{wedgeexpand} and \eqref{polygexpand} implicitly includes multiplicities of the eigenvalues.

We finally recall embedding theorems $DV_{\beta,q}(\mathcal{Q}, \gamma) \subset H^r_\sigma(\mathbb{R}^+, H^s(\mathbb{D}))^n$ \cite{mueller}. Corollary \ref{polygonlemma} then says that given parameters $\beta$, $\gamma$, the solution may be written as the
sum of a remainder term $\bold{u}_0 \in DV_{\beta,q}(\mathcal{Q}, \gamma) \subset H^r_\sigma(\mathbb{R}^+, H^s(\mathbb{D}))^n$ and, depending on the order $s$, a finite number of  singular functions $\bold{S}_{k,\ell}^*$.

\subsection{Behaviour of solutions in a circular cone}\label{regularitycone}

We consider the elastodynamic system in spherical coordinates $(r,\theta, \phi)$ with origin at the apex. It takes the form
\begin{equation}\label{cone1}
(\lambda+\mu)\partial_{r}(\bold{\nabla}\cdot \bold{u})+\mu[\nabla^2u_{r}-\frac{2u_r}{r^2}-\frac{2}{r^2\sin\theta}\partial_{\theta}(u_{\theta}\sin\theta)-\frac{2}{r^2\sin\theta}\partial_{\phi}u_{\phi}] + f_{r}=\varrho\partial_t^2 u_r
\end{equation}
\begin{equation}\label{cone2}
\frac{(\lambda+\mu)}{r}\partial_{\theta}(\bold{\nabla}\cdot \bold{u})+\mu[\nabla^2u_{\theta}+\frac{2}{r^2}\partial_{\theta}u_{r}-\frac{u_{\theta}}{r^2\sin^2\theta}
-\frac{2\cos\theta}{r^2 \sin^2\theta}\partial_{\phi}u_{\phi}] + f_{\theta}=\varrho\partial_t^2 u_\theta
\end{equation}
\begin{equation}\label{cone3}
\frac{(\lambda+\mu)}{r \sin\theta}\partial_{\phi}(\bold{\nabla}\cdot \bold{u})+\mu[\nabla^2u_{\phi}+\frac{2}{r^2\sin\theta}\partial_{\phi}u_{r}
+\frac{2cos\theta}{r^2\sin^2\theta}\partial_{\phi}u_{\theta}-\frac{u_{\phi}}{r^2\sin^2\theta}] + f_{\phi}=\varrho\partial_t^2 u_\phi
\end{equation}
with
\begin{equation}
\bold{\nabla}\cdot \bold{u} = \frac{1}{ r^2}\partial_{r}(r^2 u_r)+\frac{1}{r \sin(\theta)}\partial_{\theta}(\sin(\theta) u_{\theta})+\frac{1}{r \sin(\theta)}\partial_{\phi}u_\phi
\end{equation}
\begin{equation}
\nabla^2 u_{i}=\frac{1}{ r^2}\partial_{r}(r^2\partial_{r} u_i)+\frac{1}{r^2 \sin(\theta)}\partial_{\theta}(\sin(\theta) \partial_\theta u_{i})+\frac{1}{r^2 \sin(\theta)} \partial_{\phi}^2u_i, \quad \text{with}\ i=r,\phi,\theta.
\end{equation}
Note that we include a force term $\bold{f}=(f_r, f_\phi, f_\theta)^{\top}$ in the domain. 

We denote by $\bold{x}$ the point with spherical coordinates $(r,\phi,\theta)$. The local orthonormal basis vectors are 
\begin{align*}
\bold{e}_r &= (\sin(\theta)\cos(\phi), \sin(\theta)\sin(\phi),\cos(\theta))^{\top},\\
\bold{e}_\theta &= (\cos(\theta)\cos(\phi), \cos(\theta)\sin(\phi),-\sin(\theta))^{\top},\\
\bold{e}_\phi&=(-\sin(\phi), \cos(\phi), 0)^{\top},
\end{align*}
and we write the components of an arbitrary vector $\bold{u}$ in this basis as $\bold{u} = u_r \bold{e}_r + u_\theta \bold{e}_\theta + u_\phi \bold{e}_\phi$. 

Any vector field symmetric under rotations in $\phi$ will take the form 
$$\bold{u}(\bold{x},t) = u_r(r,\theta,t)\bold{e}_r + u_\theta(r,\theta,t)\bold{e}_\theta =: (\bold{e}_r, \bold{e}_\theta)^{\top} \tilde{\bold{u}}(r,\theta,t)\ .$$ 
{First we consider the  system \eqref{cone1} - \eqref{cone3} for fixed $t$.} Beagles and S\"{a}ndig \cite{bs} use Papkovich-Neuber potentials to construct solutions from the ansatz  
\begin{equation} \label{ansatzbs}
\bold{u} = 4(1-\nu) \bold{B} - \nabla(\bold{x}\cdot \bold{B} + B_4)
\end{equation}
with Poisson's ratio $\nu$ and where the components of $\bold{B} = (B_1,B_2,B_3)^{\top}$ and $B_4$ are harmonic functions. In spherical coordinates \eqref{ansatzbs} becomes
\begin{align} 
\bold{u} &= (u_r, u_\theta,u_\phi)^{\top} = (3-4\nu)(\bold{B}\cdot \bold{e}_r, \bold{B}\cdot \bold{e}_\theta, \bold{B}\cdot \bold{e}_\phi)^{\top} \nonumber\\&\qquad - (r \bold{e}_r \cdot \partial_r\bold{B} + \partial_r B_4, \bold{e}_r \cdot \partial_\theta \bold{B} + \frac{1}{r} \partial_\theta B_4, \frac{1}{\sin(\theta)} \bold{e}_r \cdot \partial_\phi \bold{B} + \frac{1}{r\sin(\theta)}\partial_\phi B_4)^{\top}.\label{ansatzbs2}
\end{align}
Set $B_1=B_2 =0$, $$B_3 = c_1 r^\alpha P_\alpha(\cos(\theta)), \ B_4 = c_2 r^{\alpha+1}P_{\alpha+1}(\cos(\theta)),$$ where $P_\alpha(\cos(\theta))$ are Legendre functions of the first kind  and $\alpha >0$ will be specified below. Substituting this ansatz into \eqref{ansatzbs2} gives the general form of the rotationally symmetric solutions to \eqref{cone1} - \eqref{cone3} {at fixed time $t$},
\begin{equation}\label{usym}
\bold{u}(r,\theta) = c_1 r^{\alpha} 
\begin{pmatrix}A_{11}(\alpha,\theta)\\ A_{21}(\alpha,\theta)\end{pmatrix}  + c_2 r^\alpha \begin{pmatrix}B_{11}(\alpha,\theta)\\B_{21}(\alpha,\theta) \end{pmatrix} 
\end{equation}
with $(A_{11}(\alpha,\theta), A_{21}(\alpha,\theta))= ((3-4\nu - \alpha)P_\alpha \cos(\theta),  P_\alpha' \cos(\theta) \sin(\theta) - (3-4\nu) P_\alpha \sin(\theta))$ as well as $(B_{11}(\alpha,\theta),B_{21}(\alpha,\theta)) = ( -(\alpha+1) P_{\alpha+1}, \sin(\theta) P_{\alpha+1}')$.

Using the Mellin transform with respect to $r$, 
$$\widetilde{w}(\alpha,\theta,\phi)= \frac{1}{\sqrt{2\pi}} \int_0^\infty r^{-\alpha-1} w(r,\theta,\phi) \ dr$$ 
the system \eqref{cone1} - \eqref{cone3} with Dirichlet boundary conditions transforms into a parameter-dependent boundary value problem. 
{The exponents $\alpha$ are given by the roots of the equation} 
$$\mathrm{det}
\begin{pmatrix}A_{11}(\alpha,\omega) & B_{11}(\alpha,\omega)\\ A_{21}(\alpha,\omega) & B_{21}(\alpha,\omega)\end{pmatrix} = 0\ ,$$
where $\omega$ is the opening angle. The vanishing of the determinant is equivalent to the following transcendental equation for $\alpha$:
\begin{equation}\label{transcendentalcone}
0 = \frac{-(\alpha+1)}{\sin(\omega)} \left(P_\alpha^2 \cos(\omega)(\alpha+4\nu-3) + P_\alpha P_{\alpha+1} (3-4\nu - \cos^2(\omega) (2\alpha+1)) + P_{\alpha+1}^2 \cos(\omega)(\alpha+1)\right)\ .
\end{equation}
Imposing homogeneous Dirichlet conditions on ${\bf u}$ in \eqref{usym} determines the coefficients $c_1$, $c_2$ and hence the corresponding eigenfunction. 
For numerical results for $\alpha_\ell$ and their dependence on $\omega$, see \cite{bs}.

{Now we apply the partial Fourier transform $\mathcal{F}_{t\to\tau}$ to the system \eqref{cone1} - \eqref{cone3} and obtain the following  parameter dependent Lam\'{e} equation in the cone $\mathbb{K}$ with opening angle $\omega$, 
\begin{equation}\label{eq55}
(\lambda+\mu)\nabla(\nabla \cdot \hat{\textbf{u}})+\mu\Delta\hat{\textbf{u}} + \tau^2 \hat{\textbf{u}}=\hat{\bold{f}}, \quad  \textbf{x}\in\mathbb{K},
\end{equation}
with Dirichlet boundary condition $\hat{\textbf{u}}|_{\partial\mathbb{K}} = \hat{\bold{g}}$. 
Let $\hat{\bold{f}} \in H^0_\beta(\mathbb{K})^n$, $\hat{\bold{g}} \in H^{3/2}_\beta(\partial \mathbb{K})^n$ 
Assume that no eigenvalues of the pencil $\mathcal{A}_D$ from \eqref{pencildef}, more concretely no roots of \eqref{transcendentalcone}, lie on the lines
\begin{equation}\label{alphaeqn}\mathrm{Re}\ \alpha = -\beta + \frac{1}{2} =: h\ \quad \mathrm{Re}\ \alpha = -\beta' + \frac{1}{2} =: h'.\end{equation}
We apply the framework of Appendix B, especially Section B.1. We observe that the eigenfunctions $A_{11}(\alpha, \theta)$, $A_{21}(\alpha, \theta)$, $B_{11}(\alpha, \theta)$, $B_{21}(\alpha, \theta)$ with $\alpha$ from \eqref{transcendentalcone}  for the homogeneous Dirichlet problem  are just the eigenfunctions $\pmb{\varphi}_\ell^{(k)}$ in the power-like solution \eqref{powerlike} of the homogeneous Dirichlet boundary value problem \eqref{111}, \eqref{111b}. Now equation \eqref{eq55} with Dirichlet boundary conditions is just \eqref{119}, \eqref{119b} in Appendix B. We can therefore apply Theorem \ref{thm95} in Appendix B with $i \lambda_\ell = \alpha_\ell$ and $\mathrm{Re}\ \alpha_\ell = -\mathrm{Im}\ \lambda_\ell$. Now if} $h < \mathrm{Re}\ \alpha_\ell< h'$, then there holds the following result as a consequence of Theorem \ref{thm95} (with inhomogeneous Dirichlet data ${\bf g}$): the solution of \eqref{cone1} - \eqref{cone3} 
has the expansion 
\begin{equation}\label{symconebs}
\hat{\bold{u}}(r,\theta,\phi,\tau) = \chi(p r) \sum_{\ell} \sum_{k,j}\hat{c}^{(k,j)}_\ell(\phi,\tau) \bold{u}_\ell^{(k,j)}(r,\theta) + \hat{\bold{u}}_0(r,\theta,\phi,\tau)\ ,
\end{equation}
with $\hat{\bold{u}}_0 \in DH_{\beta'}(\mathbb{K},\tau)$, $\bold{u}_\ell^{(k,j)}$ as in \eqref{usym} with $\alpha= \alpha_\ell$ a root of \eqref{transcendentalcone} and $-h < \mathrm{Re}\ \alpha_\ell < - h'$. The sum extends over the index $k$ of the roots $\alpha_\ell$. The coefficients $\hat{c}^{(k,j)}_\ell$ in the expansion \eqref{symconebs} can be computed from the results by Maz'ya and Plamenevski\v{\i}, see \cite{bs}.\\

Taking an inverse Fourier transform from $\tau$ to $t$, the results by Matyukevich and Plamenevski\v{\i} \cite{matyu} in Section \ref{sec:polygonallame} give through Theorem \ref{conetheorem}  the following result, using the function spaces in \eqref{DVQ}, \eqref{RVQ}: 
\begin{theorem} \label{conefinalthm}Let $\gamma>0$ and $\beta \in (\beta_{r+1},\beta_r)$ with $0<\beta_r-\beta<1$, $(\bold{f},\bold{g}) \in \mathcal{R}V_\beta(\mathcal{Q},\gamma)$ and assume that the orthogonality condition \eqref{orthorelation} holds for all $\alpha_\ell$ with $\mathrm{Re}\ \alpha_\ell \in [\frac{1}{2}-\beta_1, \frac{1}{2}-\beta_r]$. Then the solution of \eqref{cone1} - \eqref{cone3} with Dirichlet condition $\bold{u}|_{\partial \mathcal{Q}} = \bold{g}$ admits an expansion 
\begin{equation}\label{coneexpand}\bold{u}(r,\theta,\phi,t) = \sum_\ell \sum_{k,j} \tilde{c}_\ell^{k,j}(\phi,t) \bold{u}_\ell^{k,j}(r,\theta) +\bold{u}_0(r,\theta,\phi,t),\end{equation}
where $\bold{u}_0 \in DH_\beta(\mathcal{Q}, \gamma)$, with $\bold{u}_\ell^{k,j}$ from \eqref{usym} and the variable coefficients $\tilde{c}_\ell^{k,j}$ as in Theorem \ref{conetheorem}. The sum in $\ell$ is over all $\alpha_\ell$ with $\mathrm{Re}\ \alpha_\ell =\frac{1}{2}-\beta_r$, while the sum over $k,j$ extends over all the generalized eigenfunctions $\bold{u}_\ell^{k,j}$ of the form \eqref{usym} corresponding to $\alpha_\ell$.
\end{theorem}


{Analogous to Corollary \ref{polygonlemma} for the wedge, Theorem \ref{conefinalthm} for the cone} says that the solution may be written as the
sum of a  remainder term $\bold{u}_0 \in DV_\beta(\mathcal{Q}, \gamma) \subset H^r_\sigma(\mathbb{R}^+, H^s(\mathbb{D}))^n$ and, depending on the order $s$, a finite number of  singular functions $\bold{u}_\ell^{k,j}$.

\section{BEM discretization}\label{sec 3} 

To solve the energetic weak formulations \eqref{energetic weak formulation} and \eqref{hypersingeq} in a discretized form,  we  consider a uniform decomposition of the time interval $[0,T]$ with time step $\Delta t=\nicefrac{T}{N_{\Delta t}}$, $N_{\Delta t}\in\mathbb{N}^{+}$, generated by the $N_{\Delta t}+1$ times $t_{n}=n\Delta t$, $n=0,\ldots,N_{\Delta t}$. We define the corresponding space ${V_{\Delta t,s}}$ of piecewise polynomial functions of degree $s$ in time
 (continuous and vanishing at $t=0$ if $s\geq 1$).

For the space discretization in 2d, we introduce a boundary mesh constituted by a set of straight line segments $\mathcal{T}=\left\lbrace e_1,... ,e_M \right\rbrace$ such that $h_i:=length(e_i)\leqslant h$, $e_i\cap e_j=\emptyset$ if $i\neq j$ and $\cup_{i=1}^M \overline{e}_i=\overline{\Gamma}$ if $\Gamma$ is polygonal, or a suitably fine approximation of $\Gamma$ otherwise. 
In 3d, we assume that $\Gamma$ is triangulated by  $\mathcal{T}=\{e_1,\cdots,e_{M}\}$, with $h_i:=diam(e_i)\leqslant h$, $e_i\cap e_j=\emptyset$ if $i\neq j$ and if $\overline{e_i}\cap \overline{e_j} \neq \emptyset$, the intersection either an edge or a vertex of both triangles. 

On $\mathcal{T}$ we define  $\mathcal{P}_p$ as the space of polynomials of degree $p$, and consider the spaces of piecewise polynomial functions 
$$X^{-1}_{h,p}=\left\lbrace w\in L^2(\Gamma)\: : \: w\vert_{e_i}\in \mathcal{P}_p,\: e_i\in \mathcal{T}  \right\rbrace\subset \widetilde{H}^{-1/2}(\Gamma)$$
and 
$$X^{0}_{h,p}=\left\lbrace w\in C^0(\Gamma)\: : \: w\vert_{e_i}\in \mathcal{P}_p,\: e_i\in \mathcal{T}  \right\rbrace \subset \widetilde{H}^{1/2}(\Gamma).$$

The Galerkin approximations of \eqref{energetic weak formulation}, \eqref{hypersingeq} corresponding to these discrete spaces read, with $B_{D/N,\Sigma}$ as in \eqref{bilinearBD}, \eqref{bilinearBN}: \\

\noindent \textit{Find $\pmb{\Phi}_{h,\Delta t} \in \left({V_{\Delta t,s_p}}\otimes X^{-1}_{h,p}\right)^n$ such that}
\begin{equation}\label{energetic weak formulationh}
B_{D,\Sigma}({\pmb{\Phi}_{h,\Delta t}},\pmb{ \tilde{\Phi}}_{h,\Delta t})=\langle\partial_t{\left(\mathcal{K'}+1/2\right) {\textbf{g}}},\pmb{ \tilde{\Phi}}_{h,\Delta t}\rangle_{{L^2(\Sigma)}},
\end{equation}
\textit{for all $\pmb{\tilde{\Phi}}_{h,\Delta t} \in \left({V_{\Delta t,s_p}}\otimes X^{-1}_{h,p}\right)^n$.}\\

\noindent \textit{Find $\pmb{\Psi}_{h,\Delta t}\in \left({V_{\Delta t,s_q}}\otimes X^{0}_{h,q}\right)^n$ such that}
\begin{equation}\label{hypersingeqh}
B_{N,\Sigma}({\pmb{\Psi}_{h,\Delta t}},\pmb{ \tilde{\Psi}}_{h,\Delta t})=\langle\partial_t{\left(\mathcal{K}-1/2\right){\textbf{h}}},\pmb{ \tilde{\Psi}}_{h,\Delta t}\rangle_{{L^2(\Sigma)}},
\end{equation}
\textit{for all $\pmb{\tilde{\Psi}}_{h,\Delta t} \in \left({V_{\Delta t,s_q}}\otimes X^{0}_{h,q}\right)^n$.}\\

\begin{remark}
Due to the continuity and coercivity of the bilinear forms \eqref{energetic weak formulation} {(Proposition \ref{DPbounds})}, respectively \eqref{hypersingeq} \cite{Becache1993}, the discretized equations \eqref{energetic weak formulationh}, respectively \eqref{hypersingeqh}, admit a unique solution. Stability and a priori error estimates for the numerical error follow as in \cite{bh}. {The intention of this article is to show that the use of graded meshes and of higher-order polynomials leads to improved approximation rates for the solution. This is the subject of Section \ref{approxsection}.}
\end{remark}

\begin{figure}[h!]
\centering
\subfloat[]{{\includegraphics[height=3.5cm, width=3.5cm]{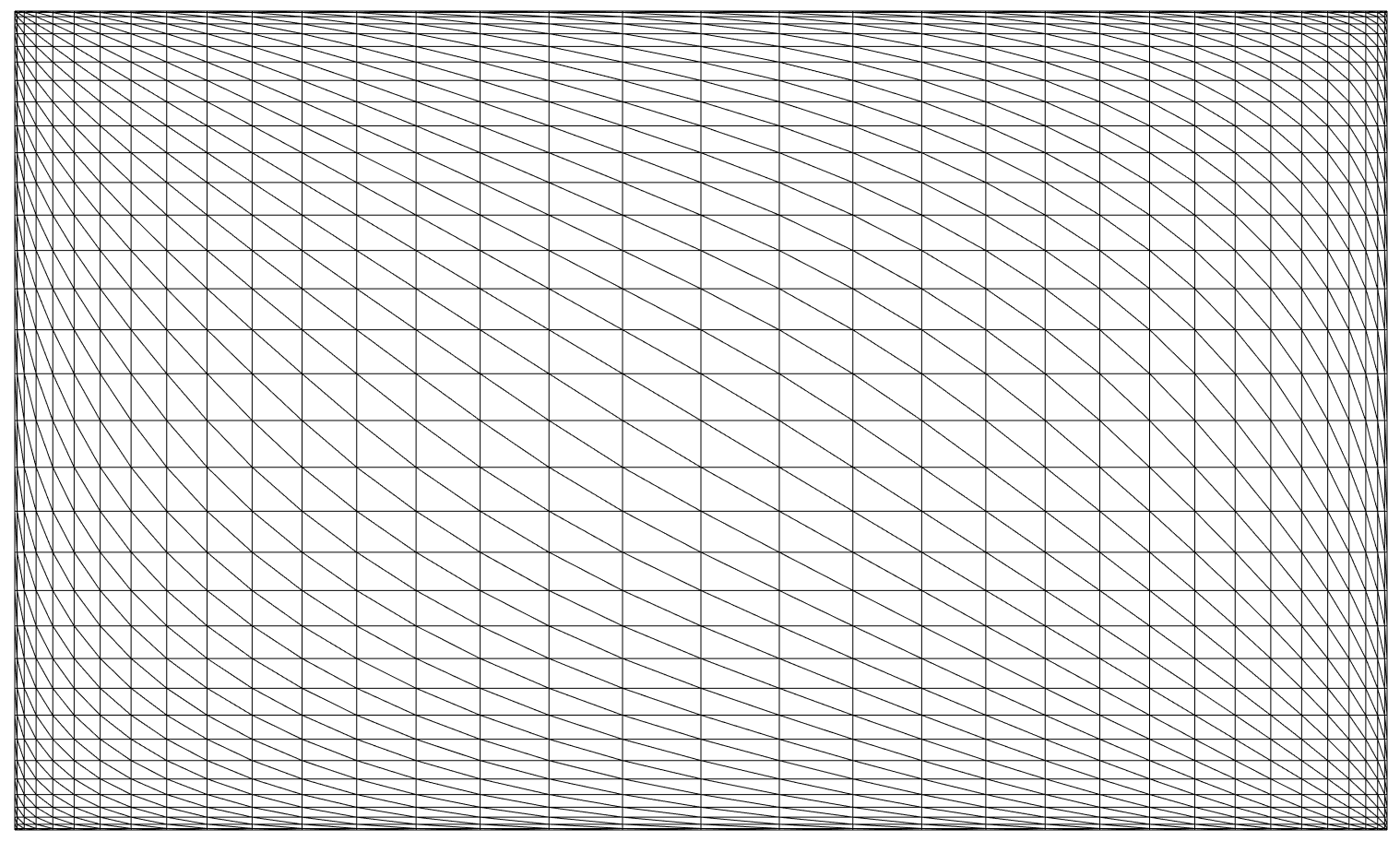}}
\qquad
{\includegraphics[height=3.5cm, width=3.5cm]{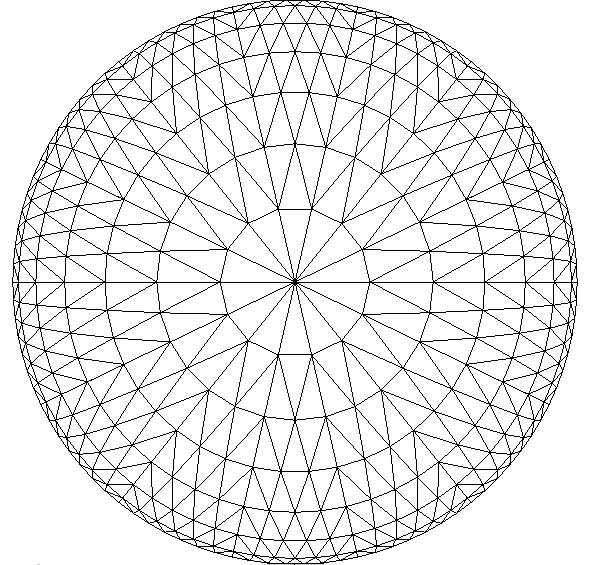}}}
\hskip 1.2cm
\subfloat[]{{\includegraphics[height=3.5cm, width=3.5cm]{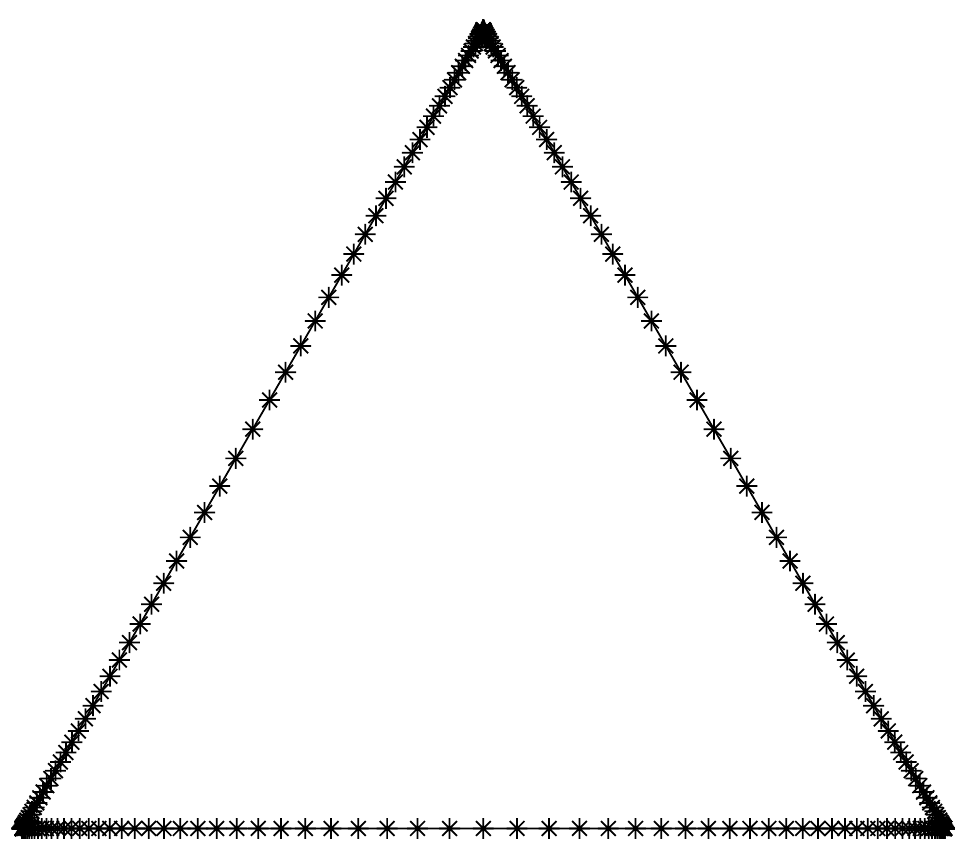}}
\qquad
{\includegraphics[width=3.5cm]{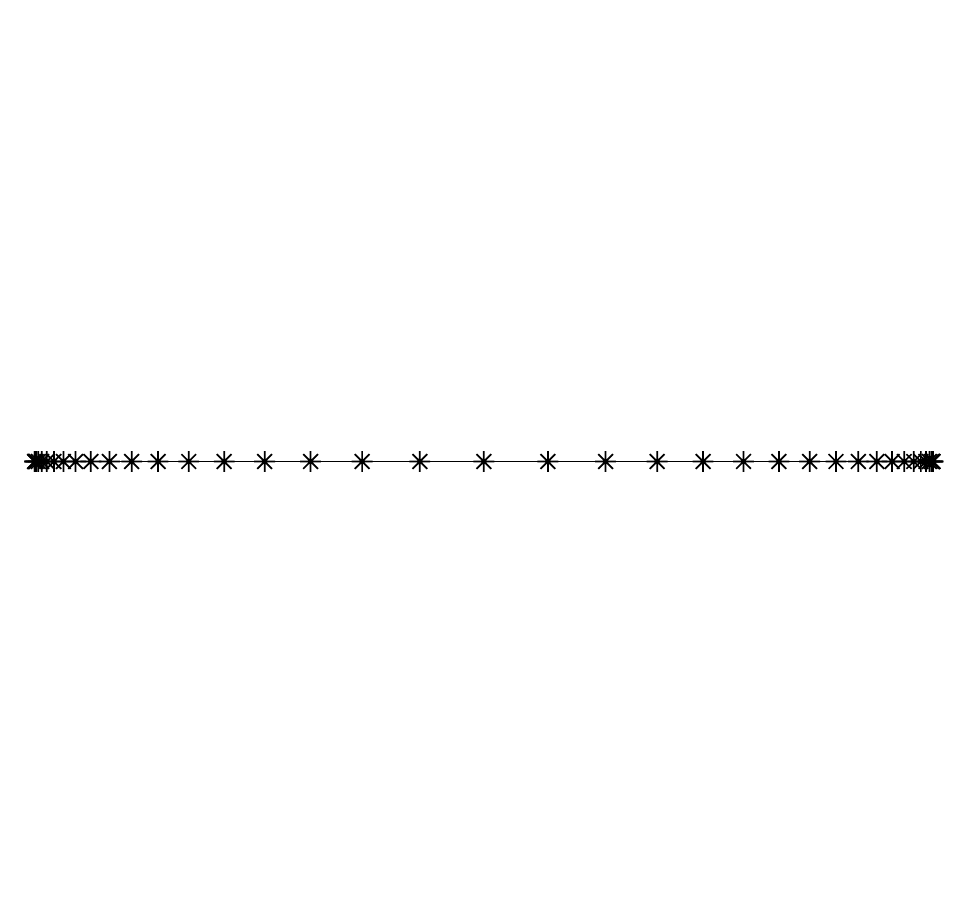}}}
\caption{$\tilde{\beta}$-graded meshes for the square and the circular screen with $\tilde{\beta}=2$ (a) and $\tilde{\beta}$-graded meshes for 1D obstacles with $\tilde{\beta}=3$ (b).}\label{gradmesh}
\end{figure}

In this article we consider the approximation on quasiuniform and $\tilde{\beta}$-graded meshes, for a constant $\tilde{\beta}\geq 0$. To define $\tilde{\beta}$-graded meshes on the interval $[-1,1]$, by symmetry it suffices to specify the nodes in $[-1,0]$. There we let 
\begin{equation}\label{gradedmesh}x_k=-1+\left(\frac{k}{N_l}\right)^{\tilde{\beta}}\end{equation} for $k=1,\ldots,N_l$. {We denote by $h$ the size of the longest interval and by $h_1 = x_1-x_0$ the size of the smallest interval.} For the square $[-1,1]^2$, the nodes of the $\tilde{\beta}$-graded mesh are tuples of such points, $(x_k,x_l),~k,l=1,\ldots, N_l$. For $\tilde{\beta}=1$ we recover a uniform mesh. \\

For general polyhedral geometries graded meshes can be locally modeled on these examples. In particular, on the circular screen of radius $1$, for $\beta =1$ we take a uniform mesh with nodes on concentric circles of radius $r_k=1-\frac{k}{N_l}$ for $k=0,\ldots,N_l-1$. For the $\tilde{\beta}$-graded mesh, the radii are moved to $r_k=1-(\frac{k}{N_l})^{\tilde{\beta}}$ for $k=0,\ldots,N_l-1$. While the triangles become increasingly flat near the boundary, their total number remains proportional to $N_l^2$.  \\

The global mesh size $h$ of a graded mesh is defined to be the diameter of the largest element. The diameter of the smallest element is of order $h^{\tilde{\beta}}$. \\

Examples of the resulting $2$-graded meshes on the square and the circular screens are depicted in Figure \ref{gradmesh}(a). \\

We also consider geometrically graded meshes on $\Gamma$. To define them on the reference interval $[-1,1]$ and with a refinement parameter $\sigma\in(0,1/2]$, in $[-1,0]$ we let $x_0=-1$, 
\begin{equation}\label{geometric points}
x_{k}=\sigma^{N_l+1-k}-1
\end{equation}
for $k=1,\ldots,N_l$, and we specify corresponding nodes in $[0,1]$ by symmetry. For the $hp$ version the polynomial degree $p$ increases linearly from $\partial \Gamma$: $p=\mu k$ in $[x_k, x_{k+1}]$ for a given $\mu >0$.

\section{Approximation results for Dirichlet and Neumann traces}\label{approxsection}

{This section splits into three subsections. In Subsection \ref{sec51} we} consider  the time-dependent elastodynamic problem in an exterior Lipschitz domain 
$\Omega \subset \mathbb{R}^n\setminus \overline{\Omega'}$, where $\Omega'$ has a piecewise 
smooth boundary with curved, non-intersecting edges, respectively cone points. 
Using the results from Section \ref{regularitysection}, we see that the solution admits an explicit singular expansion with the same singular behavior in the spatial variables as the time independent Lam\'{e} equation. This behavior is then used to analyze the error of piecewise polynomial approximations on a graded mesh {in Subsection \ref{sec52}, respectively $hp$ approximations on a quasi-uniform mesh in Subsection \ref{sec53}.}

\subsection{{Statement of regularity results}}\label{sec51}

We first consider a circular wedge {(Figure \ref{wedgemeshplot})}, leading to the regularity result in Proposition \ref{regpropositionwedge}. The case of a circular cone {(Figure \ref{conemeshplot})} is then discussed, leading to Proposition \ref{regpropositioncone}. 

\begin{figure}[b]
\centering
\includegraphics[scale=0.5]{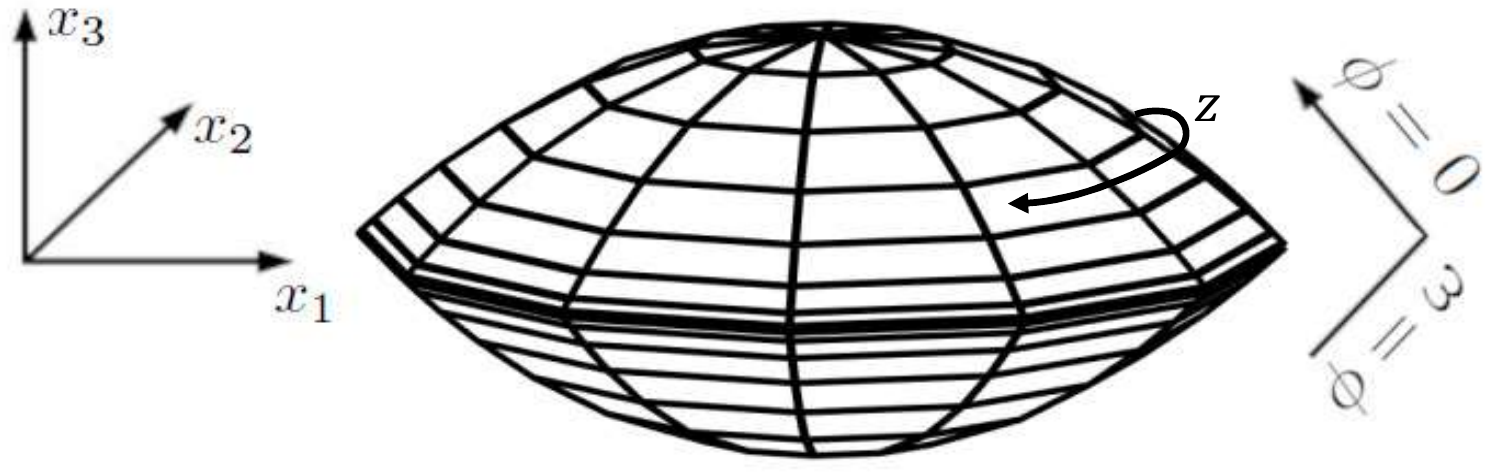}
\caption{Geometry and graded mesh on the wedge.}
\label{wedgemeshplot}
\end{figure}

For the \emph{exterior of a circular wedge} with opening angle $\omega$ and edge $\{(x_1,x_2,0) \in \mathbb{R}^3 :  {x_1^2+x_2^2}= 1\}$, in a neighborhood of the edge we use local cylindrical coordinates $(r,\phi,z)$ as in Subsection \ref{regularitywedge}: the distance to the edge is given by $r=|1-\sqrt{x_1^2+x_2^2}|$, $\phi$ is the polar angle, while the edge variable $z$ is the azimuthal angle in the $x_1 - x_2$-plane, {along the equator,} $\tan(z)= \frac{x_2}{x_1}$. For $\omega \to 2\pi^-$, the wedge degenerates into the circular screen $\{(x_1,x_2,0) \in \mathbb{R}^3 :  {x_1^2+x_2^2}\leq 1\}$. The geometry of the wedge and its discretization by a graded mesh are illustrated in Figure \ref{wedgemeshplot}. As in \cite{petersdorff3}, an analogous expansion to {Theorem} \ref{wedgelemma} for the {solution of the elastodynamic} equation \eqref{Navier equation} also holds for curved edges, with the same leading singular term $r^{\nu^\ast}$.

For the Dirichlet problem ($B=D$), respectively the Neumann problem ($B=N$), assume that the spectrum $\sigma(\mathcal{A}_B)$ of the pencil $\mathcal{A}_B$ {(from \eqref{pencildef} and its special case \eqref{pencilspec})} is constant on the edge and that there exists 
$\beta \in \mathbb{R}$ such that $\{\lambda \in \mathbb{C} : \mathrm{Im} \ \lambda = \beta-1\} \cap \sigma(\mathcal{A}_B) = \emptyset$. 

{Using Section 3 and Appendix B we can show} the following regularity result for the boundary traces of the solution: 

\begin{proposition}\label{regpropositionwedge} 
{a) 
Let $\gamma>0$, $q \in \mathbb{N}_0$ and $\nu^\ast$ the leading singular exponent, {which is} the minimum between $\frac{\pi}{\omega}$ and the minimal root of \eqref{eq1.11}. Let $(\bold{f},\bold{g}) \in \mathcal{R}V_{\beta,q}(\mathcal{Q},\gamma)$ and assume that the orthogonality condition 
\eqref{orthorelation} holds.
Then the Neumann trace of the solution $\textbf{u}$ of the Dirichlet problem \eqref{Navier equation}, \eqref{dirichlet condition} with right hand side $\bold{f}$, Dirichlet data $\bold{g}$ and   initial conditions \eqref{initial vanishing condition} satisfies
\begin{equation}
p_i(\textbf{u})(r,\phi,z,t)|_\Gamma = b_i(\phi,z,t) r^{\nu^\ast-1} + \phi_{i,0}(r,\phi,z,t)\ . \label{decompositionEdge}
\end{equation}
Here, $b_i$ is smooth for smooth data and $\phi_{i,0}$ {is} a less singular remainder.\\
b) Let $\gamma>0$, $q \in \mathbb{N}_0$ and $\nu^\ast$ the leading singular exponent, {which is} the minimum of $\frac{\pi}{\omega}$ and the minimal root of \eqref{eq1.12}. Assume that $i\lambda_1 = \nu^\ast$ is the only eigenvalue in the strip $\beta -1 \leq \mathrm{Im}\ \lambda_1 \leq0$. Let $(\bold{f},\bold{h}) \in \mathcal{R}V_{\beta,q-1}(\mathcal{Q},\gamma)$ and assume that the orthogonality condition 
\eqref{orthorelation} holds. 
Then the Dirichlet trace of the solution $\textbf{u}$ of the Neumann problem \eqref{components with hooke tensor}, \eqref{neumann condition} with right hand side $\bold{f}$, Neumann data $\bold{h}$ and initial conditions \eqref{initial vanishing condition} satisfies
\begin{equation}
u_i(r,\phi,z,t)|_\Gamma = a_i(\phi,z,t) r^{\nu^\ast} + u_{i,0}(r,\phi,z,t)\  . \label{decompostionEdget}
\end{equation}
Here, $a_i$ is smooth for smooth data and $u_{i,0}$ {is} a remainder which is less singular in the variable $r$.}
\end{proposition} 
\begin{proof}
a) First we note that for the Dirichlet problem with $\textbf{u}|_{\Gamma} = 0$ we locally have the regularity estimate
in Proposition \ref{9.7} by use of a partition of unity (see Proposition 9.3, (160) in \cite{matyu}). The corresponding 
estimate for the solution of the inhomogeneous problem is estimate (159) in Proposition 9.3, \cite{matyu}. 
Here, for curved edges, one introduces local charts in a neighborhood of the edge, to obtain a problem 
with variable coefficients in a wedge $\mathbb{D} = \mathbb{D}_j$ in the $j$-th coordinate chart. First one uses a function $(y,z) \mapsto \zeta^{(j)}(y,z) \in C^\infty(\mathbb{D}_j)$
which is independent of $z$ and, for sufficiently small $\delta>0$,  $\zeta^{(j)}=1$ for $|y|<\delta$ and $\zeta^{(j)} = 0$ for $|y|>2\delta$.
Set $\zeta^{(j)}_\tau(y,z) = \zeta^{(j)}(|\tau|y,z)$. Then one glues the functions $\zeta^{(j)}_\tau$ together with a 
partition of unity. In the proof of \eqref{9.7estimate} one replaces $\chi_\tau$ by the map $(y,z) \mapsto \eta(z)\zeta^{(j)}_\tau(y,z)$ 
supported in a small neighborhood of $z=0$, and $\eta=1$ near $z=0$. Compared to Proposition \ref{9.7} some additional
terms arise from the differentiation of the cut-off functions in $z$. This differentiation does not increase the order of growth in $|\tau|$.
Therefore, with a sufficiently large constant $\gamma_0>0$ and $\gamma>\gamma_0$ in Proposition \ref{9.7}, we can 
remove these additional terms from the estimate. The expansion 
\eqref{asymptoticexpansion} in Theorem \ref{conetheorem} is thereby also obtained for curved edges, and expression \eqref{decompositionEdge} follows by taking traces. 

Smoother data $\bold{f}$, $\bold{g}$ lead to a smoother remainder term in the expansion \eqref{asymptoticexpansion}. 

\noindent b) The proof for the Neumann problem is analogous. The relevant regularity estimates may be found in Proposition 9.4 in \cite{matyu}. 
\end{proof}

%
%


\begin{figure}[ht]
\centering
   \subfloat[]{
\includegraphics[scale=0.33]{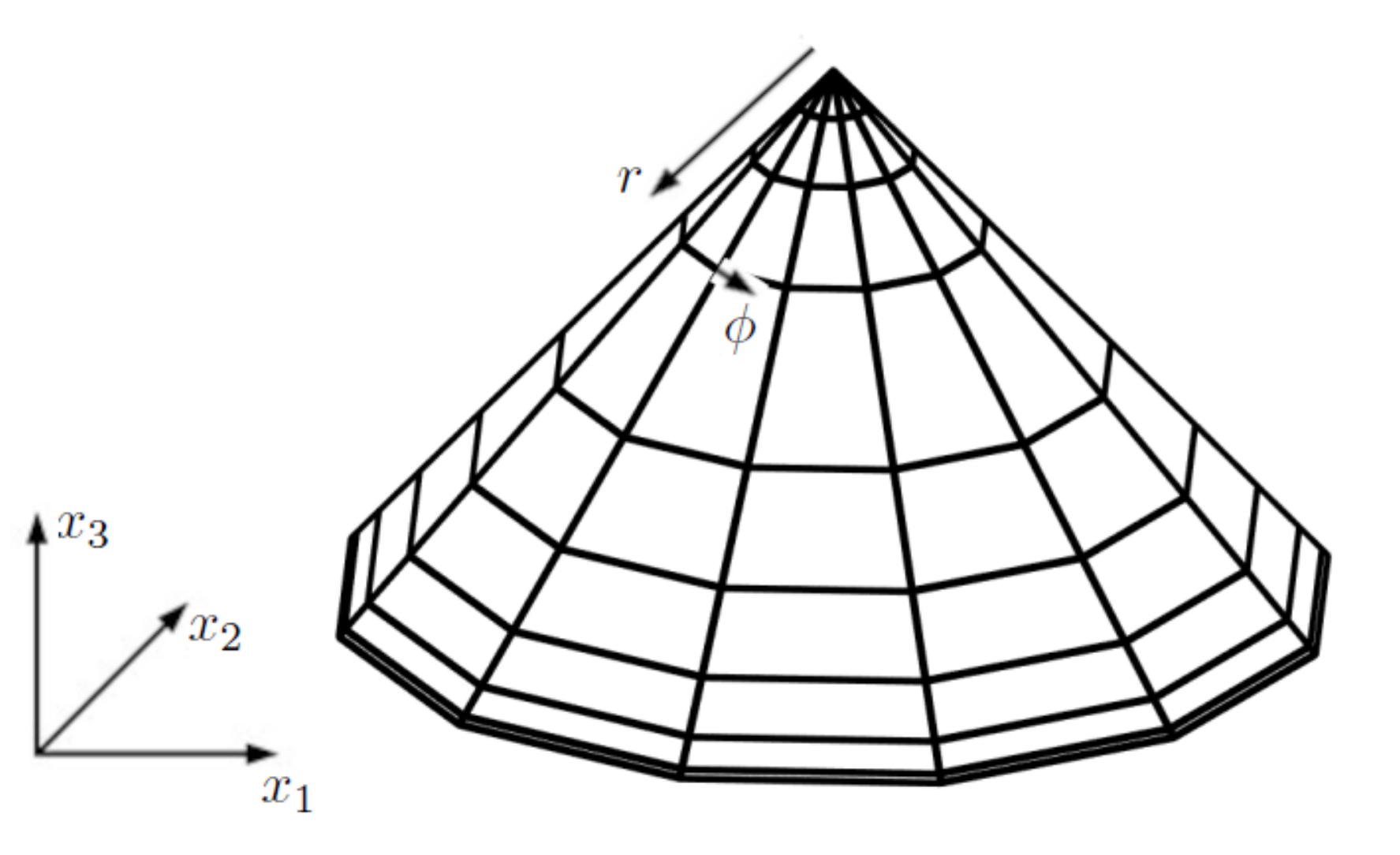}}
   \subfloat[]{ 
\includegraphics[scale=0.33]{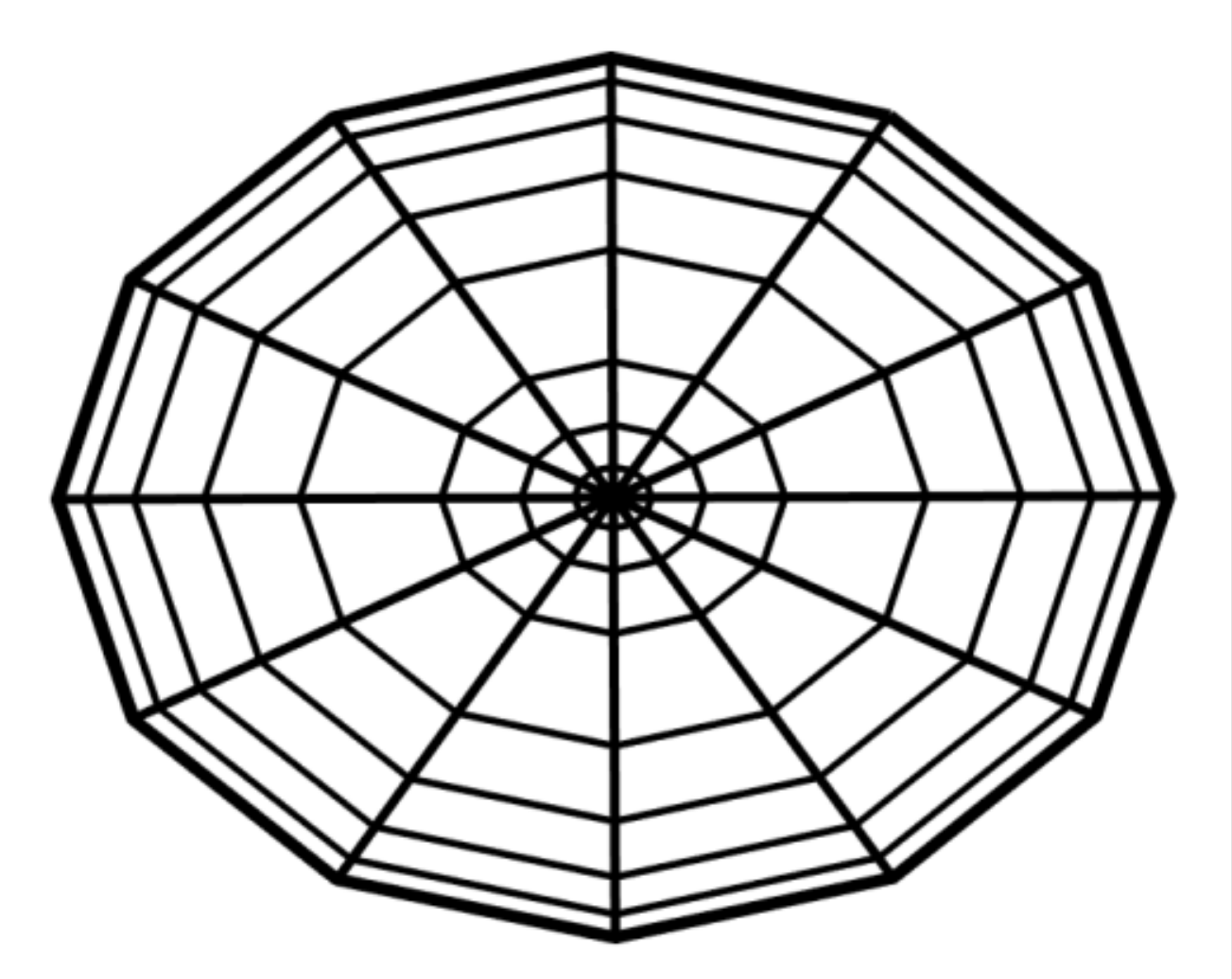}}
\caption{Geometry and graded mesh on a circular cone: viewed from the side (a) and from above (b).}
\label{conemeshplot}
\end{figure}

{We now consider the elastodynamic equations in the \emph{exterior of a cone} $\mathbb{K}$ with vertex at $r=0$, as illustrated in Figure \ref{conemeshplot}.} 

For the Dirichlet problem ($B=D$), respectively the Neumann problem ($B=N$), assume that the spectrum $\sigma(\mathcal{A}_B)$ of the pencil $\mathcal{A}_B$ {(from \eqref{pencildef} and its special case \eqref{pencilspec})} is constant on the edge and that there exists 
$\beta \in \mathbb{R}$ such that $\{\lambda \in \mathbb{C} : \mathrm{Im} \ \lambda = \beta-\frac{1}{2}\} \cap \sigma(\mathcal{A}_B) = \emptyset$. 

{Using Subsection \ref{regularitycone} and Appendix B we can show the following result} near the vertex of the cone for the boundary traces of the solution in spherical coordinates:


\begin{proposition}\label{regpropositioncone} 
{a) 
Let $\gamma>0$, $q \in \mathbb{N}_0$. Assume that $i\lambda_1 = \alpha$ is the only eigenvalue of the pencil $\mathcal{A}_D$ in the strip $\beta -\frac{1}{2} \leq \mathrm{Im}\ \lambda_1 \leq0$. 
Let $(\bold{f},\bold{g}) \in \mathcal{R}V_{\beta,q}(\mathcal{Q},\gamma)$ and assume that the orthogonality condition 
\eqref{orthorelation} holds.
Then the Neumann trace of the solution $\textbf{u}$ of the Dirichlet problem \eqref{Navier equation}, \eqref{dirichlet condition} with right hand side $\bold{f}$, Dirichlet data $\bold{g}$ and   initial conditions \eqref{initial vanishing condition} satisfies
\begin{equation}
p_i(\textbf{u})(r,\phi,\theta,t)|_\Gamma =\chi(r)r^{\alpha-1} b_i(\phi,\theta,t) +  \phi_{i,0}(r, \phi,\theta,t) \ . \label{decomposition}
\end{equation}
Here, 
$b_i$ is smooth for smooth data and $\phi_{i,0}$ a less singular remainder.\\
b) Let $\gamma>0$, $q \in \mathbb{N}_0$. Assume that $i\lambda_1 = \alpha$ is the only eigenvalue of the pencil $\mathcal{A}_N$ in the strip $\beta -\frac{1}{2} \leq \mathrm{Im}\ \lambda_1 \leq0$. Let $(\bold{f},\bold{h}) \in \mathcal{R}V_{\beta,q-1}(\mathcal{Q},\gamma)$ and assume that the orthogonality condition 
\eqref{orthorelation} holds. 
Then the Dirichlet trace of the solution $\bold{u}$ of the Neumann problem \eqref{components with hooke tensor}, \eqref{neumann condition} with right hand side $\bold{f}$, Neumann data $\bold{h}$ and initial conditions \eqref{initial vanishing condition} satisfies
\begin{equation}\label{decompositiont}
u_i(r,\phi,\theta,t)|_\Gamma  = \chi(r)r^{\alpha} a_i(\phi,\theta,t) + u_{i,0}(r, \phi,\theta,t) \ .\\
\end{equation}
Here, 
$a_i$ is smooth for smooth data and $u_{i,0}$ a less singular remainder.}
\end{proposition} 
\begin{proof}
{a)} First one notices that locally for the cone $\mathbb{K}$ the estimate \eqref{H2estimate} for the Dirichlet problem
holds, see also Proposition 9.1, (150) in \cite{matyu}. Taking traces of the resulting expansion \eqref{coneexpand} gives \eqref{decompositiont}. 
As in the case of a wedge {(Proposition \ref{regpropositionwedge} and Theorem \ref{conetheorem})}, {using the analogue of  \eqref{H2estimate} for smoother data $\bold{f}$, $\bold{g}$, we can derive expansion \eqref{decomposition} by taking the boundary traction $p_i(\textbf{u})$ of the decomposition \eqref{asymptoticexpansion} of the solution of the Dirichlet boundary value problem of the elastodynamic equations.}

{b)} For the Neumann problem, Proposition 9.2 in \cite{matyu} gives an estimate analogous to \eqref{H2estimate} for $\gamma>\gamma_0>0$ sufficiently large. {Again one derives an expansion for the solution like in Theorem \ref{conetheorem}, and takes the trace.}
\end{proof}

For both the wedge and the cone, we may assume, after possibly expanding $u_{i,0}$ and $\phi_{i,0}$ further in   \eqref{decompostionEdget}, \eqref{decompositionEdge}, respectively  \eqref{decompositiont}, \eqref{decomposition}, that the regular part $u_{i,0}$ belongs to $H^{3}$ in space and $\phi_{i,0}$ belongs to $H^{1}$ in space. Corresponding expansions then also hold for the solutions $\pmb{\Psi}$ and $\pmb{\Phi}$ to the integral equations \eqref{explicit BIE}, respectively \eqref{explicit BIE hypersingular}.

To simplify notation, for a domain in $\mathbb{R}^3$ with {both} wedge and cone singularities we define \begin{equation}\label{alphadef}\tilde{\alpha} = \min\left\{\mathrm{Re}\ \nu^\ast, \mathrm{Re}\ \alpha + \frac{1}{2}\right\},\end{equation} {where we recall that $\nu^\ast$ denotes the leading singular exponent at the edge (the minimum of $\frac{\pi}{\omega}$ and the minimal root of \eqref{eq1.12}), while $\alpha$ is the leading singular exponent at the cone (the leading eigenvalue of the pencil $\mathcal{A}_{D/N}$).} For a polygonal domain in $\mathbb{R}^2$, we set $\tilde{\alpha} = \mathrm{Re}\ \nu^\ast$. Note that $\nu^* = \frac{1}{2}$ for a screen in $\mathbb{R}^3$, respectively a crack in $\mathbb{R}^2$.


\subsection{Approximation on graded meshes} \label{sec52}

{We use the regularity results from the beginning of this section to deduce approximation properties on graded meshes:}

\begin{theorem}\label{approxtheorem2} Let $r\geq 0$ and $\varepsilon>0$. 
{a) Let $\bold{ u}$ be a strong solution to the homogeneous elastodynamic equation  \eqref{Navier equation} with inhomogeneous Dirichlet boundary conditions $\bold{ u}|_\Gamma = \bold{ g}$, with $\bold{ g}$ smooth. Further, let $\pmb{ \Phi}_{h,\Delta t}^{\tilde{\beta}}\in \left({V_{\Delta t,q}}\otimes X^{-1}_{h,0}\right)^n$ be the best approximation to $\bold{ p}(\bold{ u})$ in the norm of ${H}^{r}_\sigma(\R^+, \widetilde{H}^{-\frac{1}{2}}(\Gamma))^n$ on a ${\tilde{\beta}}$-graded spatial mesh  with $\Delta t \lesssim h_1$. Then $\|\bold{ p}(\bold{ u})|_\Gamma-\pmb{ \Phi}_{h, \Delta t}^{\tilde{\beta}}\|_{r,-\frac{1}{2}, \Gamma, \ast} \leq C_{{\tilde{\beta}},\varepsilon} h^{\min\{{\tilde{\beta}} \tilde{\alpha}-\varepsilon, \frac{3}{2}\}}$.  }

{b) Let $\bold{ u}$ be a strong solution to the homogeneous elastodynamic equation  \eqref{Navier equation} with inhomogeneous Neumann boundary conditions $\bold{ p}(\bold{u})|_\Gamma =  \bold{ h}$, with $\bold{ h}$ smooth. Further, let $\pmb{ \Psi}_{h,\Delta t}^{\tilde{\beta}} \in \left({V_{\Delta t,q}}\otimes X^{0}_{h,1}\right)^n$ be the best approximation to $\bold{ u}|_\Gamma$ in the norm of ${H}^{r}_\sigma(\R^+, \widetilde{H}^{\frac{1}{2}-s}(\Gamma))^n$ on a ${\tilde{\beta}}$-graded spatial mesh  with $\Delta t \lesssim h_1$. Then $\|\bold{ u}|_\Gamma-\pmb{ \Psi}_{h, \Delta t}^{\tilde{\beta}}\|_{r,\frac{1}{2}-s, \Gamma, \ast} \leq C_{{\tilde{\beta}},\varepsilon} h^{\min\{{\tilde{\beta}}(\tilde{\alpha}+s)-\varepsilon, \frac{3}{2}+s\}}$, where $s \in [0,\frac{1}{2}]$. }
\end{theorem}

Recall that $\| \cdot \|_{r,\pm\frac{1}{2}, \Gamma, \ast}$ denotes the norm on ${H}^{r}_\sigma(\R^+, \widetilde{H}^{\pm\frac{1}{2}}(\Gamma))^n$, as in Appendix A, and that $h$ is the diameter of the largest element in the graded mesh. Theorem \ref{approxtheorem2} implies a corresponding result for the solutions of the single layer and hypersingular integral equations on the surface:
\begin{corollary}\label{approxcor2} Let $r\geq 0$ and $\varepsilon>0$. 
{a) Let $\pmb{ \Phi}$ be the solution to the single layer integral equation  \eqref{explicit BIE} and  $\pmb{ \Phi}_{h,\Delta t}^{\tilde{\beta}}  \in \left({V_{\Delta t,q}}\otimes X^{-1}_{h,0}\right)^n$ the best approximation to $\pmb{\Phi}$ in the norm of ${H}^{r}_\sigma(\R^+, \widetilde{H}^{-\frac{1}{2}}(\Gamma))^n$ on a ${\tilde{\beta}}$-graded spatial mesh  with $\Delta t \lesssim h_1$. Then $\|\pmb{ \Phi}-\pmb{ \Phi}_{h, \Delta t}^{\tilde{\beta}}\|_{r,-\frac{1}{2}, \Gamma, \ast} \leq C_{{\tilde{\beta}},\varepsilon} h^{\min\{{\tilde{\beta}} \tilde{\alpha}-\varepsilon, \frac{3}{2}\}}$.  }

{b) Let $\pmb{ \Psi}$ be the solution to the hypersingular integral equation \eqref{hypersingeq} and  $\pmb{ \Psi}_{h,\Delta t}^{\tilde{\beta}} \in \left({V_{\Delta t,q}}\otimes X^{0}_{h,1}\right)^n$ the best approximation to $\pmb{ \Psi}$ in the norm of ${H}^{r}_\sigma(\R^+, \widetilde{H}^{\frac{1}{2}-s}(\Gamma))^n$ on a ${\tilde{\beta}}$-graded spatial mesh  with $\Delta t \lesssim h_1$. Then $\|\pmb{ \Psi}-\pmb{ \Psi}_{h, \Delta t}^{\tilde{\beta}}\|_{r,\frac{1}{2}-s, \Gamma, \ast} \leq C_{{\tilde{\beta}},\varepsilon} h^{\min\{{\tilde{\beta}}(\tilde{\alpha}+s)-\varepsilon, \frac{3}{2}+s\}}$, where $s \in [0,\frac{1}{2}]$. }
\end{corollary}
{Indeed, the solutions to the integral equations are given by $\pmb{\Psi} = \textbf{u}|_\Gamma$ {in terms of the solution $\textbf{u}$ which satisfies traction conditions ${\bf p}({\textbf u})|_\Gamma = \textbf{g}$, respectively} $\pmb{\Phi} = \textbf{p}(\textbf{u})|_\Gamma$ {in terms of the solution $\textbf{u}$ which satisfies Dirichlet conditions $\textbf{u}|_\Gamma=\bold{f}$}.}\\

{The proof extends the arguments for the wave equation in \cite{graded}, where ${\nu^\ast} = \frac{1}{2}$. It uses the  decompositions from Section \ref{approxsection}. In the approximation a cone is locally mapped by affine transformations onto a square, as in Figure \ref{fig:A1}. Further, the following  approximation properties in $1d$ are crucial. They may be found in \cite[Satz 3.7, Satz 3.10]{disspetersdorff}.}
\begin{lemma}\label{keylemmagrad}{Let $\varepsilon>0$, $a \in \mathbb{C}$ with $\mathrm{Re}\ a>0$ and $s \in [-1, -\mathrm{Re}\ a+\frac{1}{2})$.} Then there holds with the piecewise constant interpolant $ \Pi_{r}^{0} r^{-a}$ of $ r^{-a}$ {on a $\tilde{\beta}$-graded mesh}
$$\|r^{-a} - \Pi_{r}^{0} r^{-a}\|_{\widetilde{H}^{s}([0,1])} \lesssim  h^{\min\{{\tilde{\beta}}(-\mathrm{Re}\ a-s+\frac{1}{2})-\varepsilon, 1-s\}} . $$\ 
\end{lemma}

\begin{lemma}\label{keylemmagrad2}{Let $\varepsilon>0$, $a \in \mathbb{C}$ with $\mathrm{Re}\ a>0$ and $s \in [0, \mathrm{Re}\ a+\frac{1}{2})$.} Then there holds with the linear interpolant $\Pi_{r}^{1} r^{a}$ of $ r^{a}$ {on a $\tilde{\beta}$-graded mesh}
$$\|r^{a} - \Pi_{r}^{1} r^a\|_{\widetilde{H}^{s}([0,1])} \lesssim  h^{\min\{{\tilde{\beta}}(\mathrm{Re}\ a-s+\frac{1}{2})-\varepsilon, 2-s\}} . $$
\end{lemma}

\begin{proof}[Proof of Theorem \ref{approxtheorem2}]
{(a)}, wedge singularity: Approximating ${\bf p}({\bf u})$ on a rectangular mesh $\overline{\Gamma}_h = \bigcup \overline{\Gamma}_j$, we obtain with the triangle inequality and the approximation properties in the time variable: 
 \begin{align*}&\|{\bf p}({\textbf u}) - \Pi_x \Pi_t {\bf p}({\textbf u})\|_{r, -\frac{1}{2}, \Gamma, \ast} \\
& \leq \sum_k \|{\bf p}({\textbf u}) - \Pi_t {\bf p}({\textbf u})\|_{r, -\frac{1}{2}, (t_k, t_{k+1}]\times \Gamma, \ast}+\sum_{k,j} \|\Pi_t {\bf p}({\textbf u}) - \Pi_x \Pi_t {\bf p}({\textbf u})\|_{r, -\frac{1}{2}, (t_k, t_{k+1}]\times \Gamma_j, \ast}\\
& \leq \sum_k  (\Delta t)^{a}\|{\bf p}({\textbf u}) \|_{r+a, -\frac{1}{2}, (t_k, t_{k+1}]\times \Gamma}+\sum_{k,j} \|\Pi_t {\bf p}({\textbf u}) - \Pi_x \Pi_t {\bf p}({\textbf u})\|_{r, -\frac{1}{2}, (t_k, t_{k+1}]\times \Gamma_j, \ast}\ .
\end{align*}
Now, we use the decomposition \eqref{decompositionEdge} for ${\bf p}({\textbf u})$ and consider the singular and regular parts separately.
For the second sum, we use the singular expansion of each component,
\begin{align*}
\|\Pi_t p_i({\textbf u}) - \Pi_x \Pi_t p_i({\textbf u})\|_{r, -\frac{1}{2}, (t_k, t_{k+1}]\times \Gamma_j, \ast} &\leq \| \Pi_t b_i(\phi,z,t) r^{{\nu^\ast}-1} - \Pi_t \Pi_x b_i(\phi,z,t) r^{{\nu^\ast}-1}\|_{r, -\frac{1}{2}, (t_k, t_{k+1}]\times \Gamma_j, \ast} \\ & \qquad + \|\Pi_t {\phi}_{i,0} - \Pi_x\Pi_t {\phi}_{i,0}\|_{r, -\frac{1}{2}, (t_k, t_{k+1}]\times \Gamma_j, \ast}\ .
\end{align*}
For the first term we deduce from Lemma \ref{lemma3.3}
\begin{align*}
&\| \Pi_t b_i(\phi,z,t) r^{{\nu^\ast}-1} - \Pi_t \Pi_x b_i(\phi,z,t) r^{{\nu^\ast}-1}\|_{r, -\frac{1}{2}, (t_k, t_{k+1}]\times \Gamma_j, \ast}\\
&\leq \|\Pi_t b_i(\phi,z,t) - \Pi_t \Pi_z b_i(\phi,z,t)\|_{r, \varepsilon - \frac{1}{2} } \|r^{{\nu^\ast}-1}\|_{-\varepsilon }  +\| \Pi_t \Pi_z b_i(\phi,z,t)\|_{r,0} \|r^{{\nu^\ast}-1} - \Pi_rr^{{\nu^\ast}-1}\|_{ -\frac{1}{2}} \ .
\end{align*}

From Lemma \ref{keylemmagrad} we have $\|r^{{\nu^\ast}-1} - \Pi_rr^{{\nu^\ast}-1}\|_{ -\frac{1}{2}} \lesssim h^{{\tilde{\beta}}\mathrm{Re}\ {\nu^\ast} -\varepsilon}$ and $\|\Pi_t b_i(\phi,z,t) - \Pi_t \Pi_z b_i(\phi,z,t)\|_{r, \varepsilon - \frac{1}{2} } \lesssim h^{3/2-\varepsilon} \|\Pi_t b_i\|_{r, H^1}$, by the approximation properties in $z$. 

Finally, with Lemma \ref{lemma3.4} and Lemma \ref{lemma3.2}, in the anisotropic rectangle $R$ with sidelengths $h_1$, $h_2$ in the $x_1$, respectively $x_2$ directions: 
\begin{align*}&\|\Pi_t {\phi}_{0,i} - \Pi_x\Pi_t {\phi}_{0,i}\|_{r, -\frac{1}{2}, (t_k, t_{k+1}]\times R,\ast} \\
&\lesssim   (\Delta t)^{\rho-r}\|\partial_t^\rho {\phi}_{0,i}\|_{L^2([t_k,t_{k+1}]\times R)}  +  \max\{h_1,h_2, \Delta t\}^{\frac{1}{2}}\left(h_1 \| {\phi}_{0,i,x_1}\|_{L^2([t_k,t_{k+1}]\times R)}  + h_2 \| {\phi}_{0,i,x_2}\|_{L^2([t_k,t_{k+1}]\times R)} \right)\ .
\end{align*}
Note that the approximation error for the smooth term is of higher order. By summing over all rectangles $\Gamma_j$ of the mesh of the screen and all components, we conclude that $\|{\bf p}({\textbf u}) - \Pi_x \Pi_t {\bf p}({\textbf u})\|_{r, -\frac{1}{2}, \Gamma, \ast} \lesssim  h^{{\tilde{\beta}} \mathrm{Re}\ {\nu^\ast}-\varepsilon}$ if $\Delta t \leq \min\{h_1,h_2\}$.\\

\noindent {(a)}, cone singularity: To discuss the approximation of ${\bf p}({\bf u})$ in the cone geometry, for simplicity, we let $\Gamma$ be the square $\tilde{R}=[0,1]^{2}$. Figure \ref{fig:A1} shows how to reduce the mesh on the cone to this case by an affine map, with the exception of a small number of triangular elements.

For the rectangular elements, the approximation of the singular function follows closely the proof in \cite{graded}, and we present it below for the convenience of the reader.

For the additional triangular elements in Figure \ref{fig:A1}(b) with linear basis functions, the crucial observation is that their angles are independent of $h$, leading to a shape-regular mesh. In particular, the quotient $\rho$ of the radii of the smallest circumscribed to the largest inscribed circle remains bounded and the expected interpolation inequalities hold: For the linear interpolant $p$ of a function $f$ determined by the vertices of a triangle $T$ of circumscribed radius $\leq h$, one has  
$$\|f-p\|_{H^s(T)} \leq C_0 h^{2-s} \|f\|_{H^2(T)}\ .$$
Here, $s\in [0,1]$ and the constant $C_0$ only depends on $\rho$ and $s$. The respective proofs for the regular part $\pmb{\phi}_{0}$ and the singular function $r^{\lambda-1} b_i$ in this way directly apply to the arising triangles.


\begin{figure}[ht]
\centering
   \subfloat[]{
\includegraphics[scale=0.5]{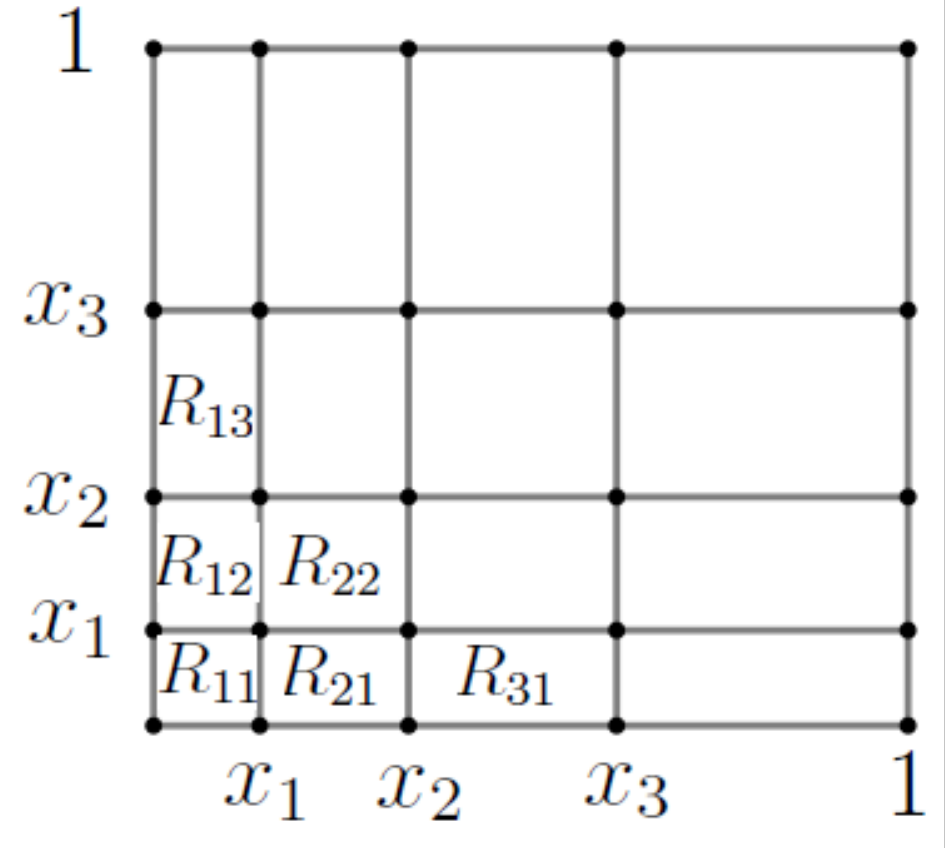}}
   \subfloat[]{ 
\includegraphics[scale=0.48]{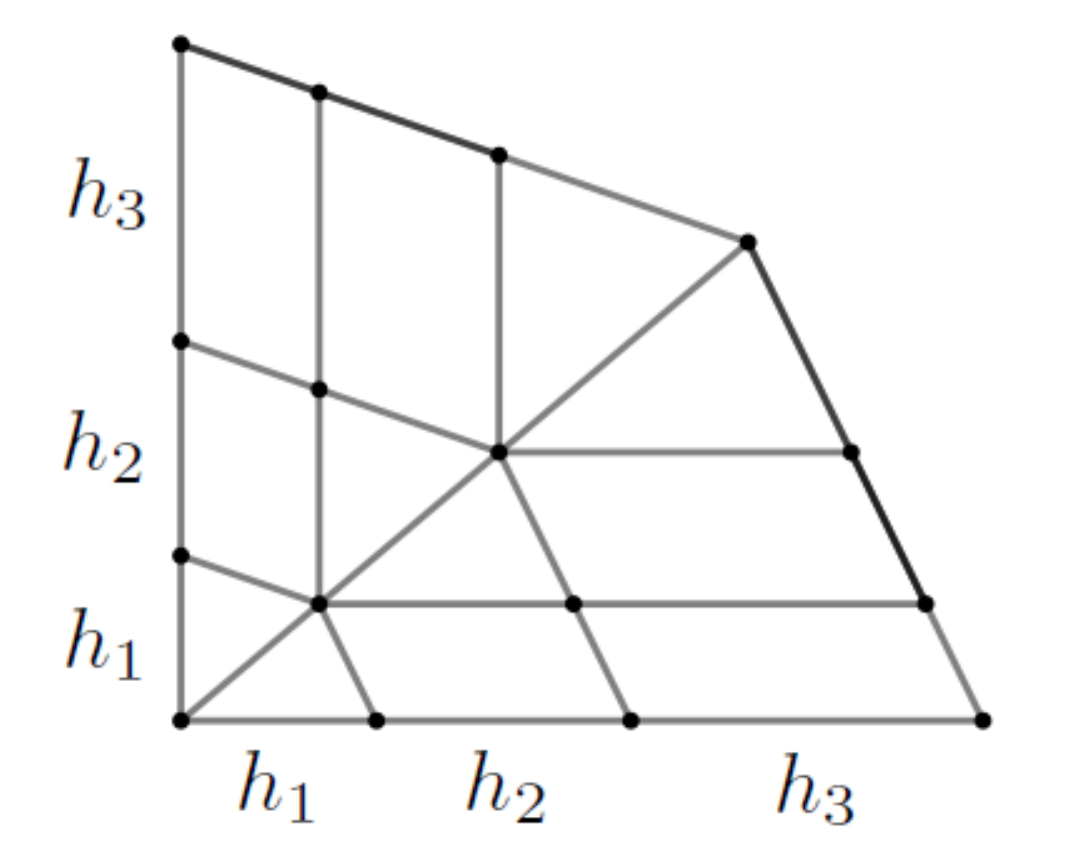}}
\caption{Affine map between meshes on (a) square and (b) cone. The parallelograms in (b) correspond to rectangles in (a), and two adjacent triangles in (b) are mapped to the diagonal squares $R_{ii}$ in (a).}
\label{fig:A1}
\end{figure}



As the approximation of the regular part $\pmb{\phi}_0$ {in the expansion \eqref{decomposition}} has already been considered in the proof for the circular wedge, it remains to analyze the approximation of the vertex singularity in \eqref{decomposition}. 

In the following we approximate the corner singularity:

In every space-time element we estimate
\begin{align*}
\|r^{\alpha-1} b_i(\phi,\theta,t) - \Pi_t \Pi_{r,\phi} r^{\alpha-1} b_i(\phi,\theta,t)\| &\leq \|r^{\alpha-1} b_i(\phi,\theta,t) - \Pi_t  r^{\alpha-1} b_i(\phi,\theta,t)\|\\&\qquad  + \|r^{\alpha-1} \Pi_t b_i(\phi,\theta,t) -  \Pi_{r,\phi} r^{\alpha-1} \Pi_tb_i(\phi,\theta,t)\|\ .
\end{align*}
As $b_i$ is smooth in time, the first term $\|r^{\alpha-1} b_i(\phi,\theta,t) - \Pi_t  r^{\alpha-1} b_i(\phi,\theta,t)\|$ can be estimated  using standard approximation properties in time and is neglible for small $\Delta t$. $\Pi_t b_i(\phi,\theta,t)$ is of the same form as the function $b_i$ in the elliptic case \cite{gwinsteph}. One may therefore adapt the elliptic approximation results to $\|(1-\Pi_{r,\phi}) r^{\alpha-1} \Pi_t b_i(\phi,\theta,t)\|$. This is then summed over all elements. We consider
\begin{equation*}
 \lVert r^{\alpha-1} \Pi_t b_i - \Pi_{\phi,r} r^{\alpha-1} \Pi_t b_i \rVert = \lVert (1-\Pi_{\phi,r}) r^{\alpha-1} \Pi_{t} b_i(\phi,\theta,t) \rVert
\end{equation*}
Let $ \Pi_t b_i(\phi,\theta,t) = \sum\limits_{j} t^{j} b_{i,j}(\phi,\theta)$ and $ f_{j}(r,\phi)=r^{\alpha-1} b_{i,j}(\phi,\theta)$ on $[t_{k},t_{k+1})$. 

With $ p_{j}|_{R_{kl}} = \sum\limits_{j} \frac{t^{j}}{h_{k}h_{l}} \int\limits_{R_{kl}} f_{j}(x,y) dy dx$ one obtains from \eqref{3.22} 
\begin{align*}
 \lVert r^{\alpha-1} \Pi_t b_i - \Pi_{r,\phi} r^{\alpha-1} \Pi_{t} b_i \rVert_{r,-\frac{1}{2},\tilde{R},*}^{2} 
 &\lesssim \sum \sum \max \lbrace \Delta t , h_{k} , h_{l} \rbrace \\ &( h_{k}^{2} \lVert \partial_{1} (r^{\alpha-1} \Pi_t b_i) \rVert_{r,0,[t_{j},t_{j+1}) \times R_{kl}}^{2} + h_{l}^{2} \lVert \partial_{2} (r^{\alpha-1} \Pi_t b_i) \rVert_{r,0,[t_{j},t_{j+1}) \times R_{kl}}^{2} ) \\
 & + \lVert r^{\alpha-1} \Pi_t b_i - \Pi_{r,\phi} r^{\alpha-1} \Pi_t b_i \rVert_{r,-\frac{1}{2}, R_{11}}^{2}
\end{align*}
The individual summands are estimated for different ranges of $l, k$:

Estimate for $k\geqslant 2,\; l\geqslant 2$: Note for $k\geqslant 2,\; x\in [x_{k-1},x_k]$ there holds $\vert h_k\vert \leq {\tilde{\beta}} 2^{{\tilde{\beta}}\tilde{\gamma}}h\, x^{\tilde{\gamma}}$ with $\tilde{\gamma} =1-\frac{1}{{\tilde{\beta}}}>0$. Therefore, if $ \Delta t  \leq h_{k} $ for all $k$
\begin{equation*}
	\max\{h_k, h_l, \Delta t\}h_k^2\| \partial_x (r^{\alpha-1} \Pi_t b_i) \|_{r,0,[t_{j},t_{j+1}) \times R_{kl}}^{2} \lesssim h^{3} \| \partial_x (r^{\alpha-1} \Pi_t b_i) \max \{ x^{\tilde{\gamma}} ,y^{\tilde{\gamma}} \}^{1/2} x^{\tilde{\gamma}} \|_{r,0,[t_{j},t_{j+1}) \times R_{kl}}^{2}
\end{equation*}
and
\begin{align}
 \lVert r^{\alpha-1} \Pi_t b_i - \Pi_{x,y} r^{\alpha-1} \Pi_{t} b_i\|_{r,-\frac{1}{2}, \bigcup_{k\geq2, l \geq 2} R_{kl},*}^{2}
 \lesssim h^{3} \lVert \partial_x (r^{\alpha-1} \Pi_t b_i) \max \{ x^{\tilde{\gamma}} , y^{\tilde{\gamma}} \} x^{2 \tilde{\gamma}} \rVert_{r,0,\tilde{R}} \label{3.J} \\
 + h^{3} \lVert \partial_y (r^{\alpha-1} \Pi_t b_i) \max \{ x^{\tilde{\gamma}} , y^{\tilde{\gamma}} \} y^{2 \tilde{\gamma}} \rVert_{r,0,\tilde{R}} \ .\nonumber
\end{align}
As $ |\partial_1(r^{\alpha-1} \Pi_t b_i ) | \lesssim r^{\alpha-2} \tilde{b}_i(\phi, \theta, t)$ for some 
$\tilde{b}_i$ square-integrable in space, and $ \max \{ x^{\tilde{\gamma}} , y^{\tilde{\gamma}} \} \leq r^{\tilde{\gamma}}$, the right hand side of \eqref{3.J} is finite if 
\begin{equation}\label{betaestimation}
{\tilde{\beta}}>\frac{3}{2({\alpha}+1/2)}.
\end{equation}
Therefore
\begin{align*}
 \lVert r^{\alpha-1} \Pi_t b_i - \Pi_{r,\phi} r^{\alpha-1} \Pi_{t} b_i \rVert_{r,-\frac{1}{2}, \bigcup_{k\geq2, l \geq 2} R_{kl},*}^{2}
 \lesssim h^{3},
\end{align*}
provided $ \Delta t \leq h_{k}$ for all $k$.\\

Estimate for $k=1,\, l>1$ (analogously $k>1,\, l=1$): With $ f(x,y) = r^{\alpha-1} b_i(\phi,\theta)$
\begin{align*}
 & \sum\limits_{j} \sum\limits_{l=2}^{N} \lVert (1-\Pi_{r,\phi}) \Pi_t f \rVert_{r,-\frac{1}{2},[t_{j},t_{j+1}) \times R_{kl},*}^{2}
 \\ &\leq \sum\limits_{j}\sum\limits_{l=2}^N\max\{\Delta t,h_k,h_l \}\left(h_1^2\Vert \partial_1 (r^{\alpha-1} \Pi_t b_i )\Vert^2_{r,0,[t_{j},t_{j+1}) \times R_{kl},*} + h_l^2 \Vert \partial_2 (r^{\alpha-1} \Pi_t b_i ) \Vert^2_{r,0,[t_{j},t_{j+1}) \times R_{kl},*} \right)
\end{align*}
Proceed as in \eqref{3.J} to see that also this term is bounded for ${\tilde{\beta}} > \frac{3}{2 (\alpha+\frac{1}{2})}$.\\

\noindent Estimate for $k=1,\, l=1$: $r^{\alpha-1}\in L^2(R_{11})$ because $\alpha >0$. Now $\lVert (1-\Pi_{r,\phi}) r^{\alpha-1} \rVert_{L^2(R_{11})} \lesssim \lVert r^{\alpha-1}\rVert_{L^2(R_{11})}$, by the $L^2$-stability of $\Pi_{r,\phi}$, and
\begin{align*}
 & \|r^{\alpha-1} \Pi_t b_i(\phi,\theta,t) - \Pi_{r,\phi} r^{\alpha-1} \Pi_t b_i(\phi,\theta,t)\|_{r,-\frac{1}{2},R_{11},*}^{2} \\
 & \lesssim \lVert (1-\Pi_{r,\phi}) r^{\alpha-1} \Pi_{t} b_i(\phi,\theta,t) \rVert_{r,-1,R_{11},*} \lVert (1-\Pi_{r,\phi}) r^{\alpha-1} \Pi_{t} b_i(\phi,\theta,t) \rVert_{r,0,R_{11},*}
\end{align*}
The second term is $ \leq h^{\alpha}$. For the first 
\begin{align*}
 \lVert (1-\Pi_{r,\phi}) r^{\alpha-1} \Pi_{t} b_i(\phi,\theta,t) \rVert_{r,-1,R_{11},*} \equiv \sup\limits_{g \in H^{-r}(\mathbb{R^{+}},\tilde{H}^{1}(R_{11}))}
   \frac{\langle (1-\Pi_{r,\phi}) r^{\alpha-1} \Pi_{t} b_i(\phi,\theta,t), g \rangle}{\lVert g \rVert_{-r,1,R_{11}}}
\end{align*}
Replacing $g$ by $g-p$, $p$ the $ H^{-r}(\mathbb{R}^{+},H^{0}(R_{11}))$-projection of $g$, 
we obtain for $\Delta t \leq h_{1} $:
\begin{align*}
 \lVert (1-\Pi_{r,\phi}) r^{\alpha-1} \Pi_{t} b_i(\phi,\theta,t) \rVert_{r,-1,R_{11},*} \leq \lVert (1-\Pi_{r,\phi}) r^{\alpha-1} \Pi_{t} b_i(\phi,\theta,t) \rVert_{r,0,R_{11}} \sup\limits_{g} \frac{\lVert g-p \rVert_{-r,0,R_{11}}}{\lVert g \rVert_{-r,1,R_{11}}}\ .\end{align*}
The first factor is bounded by  $h_{1}^{\alpha}$, while the second factor is bounded by $h_{1}$.
We conclude
\begin{equation*}
 \|r^{\alpha-1} \Pi_t b_i(\phi,\theta,t) - \Pi_t \Pi_{r,\phi} r^{\alpha-1} \Pi_t b_i(\phi,\theta,t)\|_{r,-\frac{1}{2},R_{11},*}^{2}
 \lesssim h_{1}^{2 \alpha +1}.
\end{equation*}
The assertion follows by noting that $h_1 = h^{\tilde{\beta}}$.\\

\noindent {(b)}, wedge singularity: For the approximation of ${\bf u}$ a key ingredient is Lemma \ref{TobiasTRACE}. Proceeding as above, {using the expansion \eqref{decompostionEdget}} one estimates the $i$-th component on every rectangle $R$ of the mesh:
\begin{align*}
\|\Pi_t  {u}_i - \Pi_x \Pi_t  {u}_i\|_{r, \frac{1}{2}, (t_k, t_{k+1}]\times R,\ast} &\leq \| \Pi_t a_i(\phi,z,t) r^{{\nu^\ast}} - \Pi_t \Pi_x a_i(\phi,z,t) r^{{\nu^\ast}}\|_{r, \frac{1}{2}, (t_k, t_{k+1}]\times R,\ast} \\ & \qquad + \|\Pi_t {{u}_{i,0}} - \Pi_x\Pi_t {u}_{i,0}\|_{r, \frac{1}{2}, (t_k, t_{k+1}]\times R,\ast}\ .
\end{align*}
For the first term we note with Lemma \ref{lemma3.3a}, with $\Pi_{x} = \Pi_{r,z}$,
\begin{align*}
&\| \Pi_t a_i(\phi,z,t) r^{{\nu^\ast}} - \Pi_t \Pi_x a_i(\phi,z,t) r^{{\nu^\ast}}\|_{r, \frac{1}{2}, (t_k, t_{k+1}]\times R,\ast}\\
&\leq \|\Pi_t a_i(\phi,z,t) - \Pi_t \Pi_z a_i(\phi,z,t)\|_{r,  \frac{1}{2},(t_k, t_{k+1}]\times R,\ast} \|r^{{\nu^\ast}}\|_{ \frac{1}{2} }  +\| \Pi_t \Pi_z a_i(\phi,z,t)\|_{r, \frac{1}{2}, (t_k, t_{k+1}]\times R,\ast} \|r^{{\nu^\ast}} - \Pi_rr^{{\nu^\ast}}\|_{\frac{1}{2}, R,\ast} \ .
\end{align*}
From the approximation properties in space note that 
$$\|\Pi_t a_i(\phi,z,t) - \Pi_t \Pi_z a_i(\phi,z,t)\|_{r,  \frac{1}{2} } \leq C \|\Pi_t a_i(\phi,z,t)\|_{r, H^2} h^{\frac{3}{2}}$$
and, from Lemma \ref{keylemmagrad2}, 
$$\|r^{{\nu^\ast}} - \Pi_rr^{{\nu^\ast}}\|_{\frac{1}{2}} \lesssim h^{\min\{{\tilde{\beta}} \mathrm{Re}\ {\nu^\ast}-\varepsilon, \frac{3}{2}\}}\ .$$

Each component of the regular part $\textbf{u}_0$ {in the expansion \eqref{decompositionEdge}} may be approximated as in \cite[Theorem 18]{graded}: We let ${\bf q}\in S_h^{\tilde{\beta}}$ denote the interpolant of $\textbf{u}_0$ in space and time.
On $\tilde{R}:=[0,1] \times [0,1]$, decomposed into rectangles $R_{jk}:=[x_{j-1},x_{j}] \times [y_{k-1},y_{k}]$ with side lengths $ h_{j},h_{k}$,
\begin{align*}
	\| {\bf u}_{0} -{\bf q}\|^2_{r,0,\tilde{R}} 
	\lesssim \max\{h, \Delta t\}^4\| \textbf{u}_0\|^2_{r, 3, \tilde{R}}
\end{align*}
and
\begin{align*}
	\| \textbf{u}_0 -{\bf q}\|^2_{r, 1, \tilde{R}}
	\lesssim\max\{h, \Delta t\}^2\| \textbf{u}_0\|^2_{r, 3,\tilde{R}} \ .
\end{align*}
Here we have used $h_{k} \leq {\tilde{\beta}} \ h$ {and recall that we do not indicate the time interval in the norm when it is $\mathbb{R}_+$}.
Interpolation yields $\| \textbf{u}_0 -{\bf q}\|_{r,\frac{1}{2},\tilde{R}}  \lesssim \max\{h, \Delta t\}^{3/2}\| \textbf{u}_0\|_{r, 3,\tilde{R}}$.\\

To approximate each component of the singular part, 
we set $ f_{1}(z,t):= a_i(\phi,z,t)$, $f_{2}(r) := r^{{\nu^\ast}}$ and $q(x,t) := q_1(z,t)q_2(r)$ with piecewise linear interpolants $q_j$ of $f_j$. Hence for $0\leq s<1$
\begin{align}\label{keyineq}
	\| f-q\|_{r,s,\tilde{R}}\leq &\| q_1\|_{r,0,I}\| f_2-q_2\|_{H^s(I)} + \| q_1\|_{r,s,I}\| f_2-q_2\|_{L^2(I)} \\
	&+\| f_1-q_1\|_{r,0,I}\| f_2\|_{H^s(I)} + \| f_1-q_1\|_{r,s,I}\| f_2\|_{L^2(I)}\ . \nonumber
\end{align}
Using the approximation results for $r^{\nu^\ast}$ in Lemma \ref{keylemmagrad2}, we conclude 
\begin{equation*}
	\| f-q\|_{r,\frac{1}{2},\tilde{R}}\leq c\ h^{{\tilde{\beta}} \mathrm{Re}\ {\nu^\ast}-\epsilon}.
\end{equation*} 

The approximation of the singular function on the cone {closely follows the proof for the traction $\textbf{p}(\textbf{u})$ in part a) above. For the wave equation the details are presented in \cite{graded}.}
\end{proof}
The approximation argument extends from rectangular to triangular elements as in \cite{disspetersdorff}.\\


The results for the approximation of the edge singularity in $n=3$ translate into corresponding results for a linear crack in $n=2$. In particular, in Figure \ref{fig:energy_segment} we observe the predicted rates for ${\tilde{\beta}}=1,2,3$, when ${\nu^\ast}=\frac{1}{2}$, and in Figure \ref{fig:energy_eq_triangle}  for ${\nu^\ast} = 0.5372$.


\subsection{Approximation by $hp$ methods}\label{sec53}

{We use the regularity results from the beginning of this section to deduce approximation properties of the $hp$ version on quasi-uniform meshes:}

To state the main result for the approximation error of $hp$-methods, recall from \eqref{alphadef} that $\tilde{\alpha} = \min\left\{\mathrm{Re}\ \nu^\ast, \mathrm{Re}\ \alpha + \frac{1}{2}\right\}$.

\begin{theorem}\label{approxtheorem1} Let $r\geq 0$ and $\varepsilon>0$. 
{a) Let $\bold{u}$ be a strong solution to the homogeneous elastodynamic equation \eqref{Navier equation} with inhomogeneous Dirichlet boundary conditions $\bold{u}|_\Gamma = \bold{g}$, with $\bold{g}$ smooth. Further, let $\pmb{\Phi}_{h,\Delta t} \in \left({V_{\Delta t,p}}\otimes X^{-1}_{h,p}\right)^n$ be the best approximation  in the norm of ${H}^{r}_\sigma(\R^+, \widetilde{H}^{-\frac{1}{2}}(\Gamma))^n$ to the traction $\bold{p}(\bold{u})|_\Gamma$ on a quasiuniform spatial mesh  with $\Delta t \lesssim h$. Then for $p=0,1,2,\dots$ $$\|\bold{p}(\bold{u})|_\Gamma - \pmb{\Phi}_{h,\Delta t}\|_{r, -\frac{1}{2}, \Gamma, \ast} \lesssim \left(\frac{h}{(p+1)^2}\right)^{\tilde{\alpha}{-\varepsilon}} +  \left(\frac{\Delta t}{p+1}\right)^{p+1-r}+ \left(\frac{h}{p+1}\right)^{\frac{1}{2}+\eta}\ ,$$
 where {$r \in [0,p+1)$}  and $\pmb{\phi}_0 \in {H}^{p+1}_\sigma(\R^+, \widetilde{H}^{\eta}(\Gamma))^n$ is the regular part of the singular expansion of $\bold{p}(\bold{u})$.}

{b) Let $\bold{u}$ be a strong solution to the homogeneous elastodynamic equation \eqref{Navier equation} with inhomogeneous Neumann boundary conditions $\bold{p}(\bold{u})|_\Gamma = \bold{h}$, with $\bold{h}$ smooth. 
Further, let $\pmb{\Psi}_{h,\Delta t}\in \left({V_{\Delta t,p}}\otimes X^{0}_{h,p}\right)^n$ be the best approximation in the norm of ${H}^{r}_\sigma(\R^+, \widetilde{H}^{\frac{1}{2}-s}(\Gamma))^n$ to the Dirichlet trace $\bold{u}|_\Gamma$  on a quasiuniform spatial mesh with $\Delta t \lesssim h$. Then for $p=1,2,3,\dots$ $$\|\bold{u}|_\Gamma-\pmb{\Psi}_{h, \Delta t}\|_{r,\frac{1}{2}-s, \Gamma, \ast} \lesssim \left(\frac{h}{p^2}\right)^{\tilde{\alpha}+s{-\varepsilon}} +  \left(\frac{\Delta t}{p}\right)^{p-r}+ \left(\frac{h}{p}\right)^{-\frac{1}{2}+s+\eta}\ ,$$
 where {$r \in [0,p)$}  and $\bold{u}_0 \in {H}^{p}_\sigma(\R^+, \widetilde{H}^{\eta}(\Gamma))$ is the regular part of the singular expansion of $\bold{u}$.}
\end{theorem}

Theorem \ref{approxtheorem1} implies a corresponding result for the solutions of the single layer and hypersingular integral equations on the surface:
\begin{corollary}\label{approxcor1} Let $r\geq 0$ and $\varepsilon>0$. 
{a) Let $\pmb{\Phi}$ be the solution to the single layer integral equation \eqref{explicit BIE} and  $\pmb{\Phi}_{h,\Delta t} \in \left({V_{\Delta t,p}}\otimes X^{-1}_{h,p}\right)^n$ the best approximation  in the norm of ${H}^{r}_\sigma(\R^+, \widetilde{H}^{-\frac{1}{2}}(\Gamma))^n$ to $\pmb{\Phi}$  on a quasiuniform spatial mesh  with $\Delta t \lesssim h$. Then for $p=0,1,2,\dots$ $$\|\pmb{\Phi} - \pmb{\Phi}_{h,\Delta t}\|_{r, -\frac{1}{2}, \Gamma, \ast} \lesssim \left(\frac{h}{(p+1)^2}\right)^{\tilde{\alpha}{-\varepsilon}} +  \left(\frac{\Delta t}{p+1}\right)^{p+1-r}+ \left(\frac{h}{p+1}\right)^{\frac{1}{2}+\eta}\ ,$$
 where {$r \in [0,p+1)$}  and $\pmb{\phi}_0 \in {H}^{p+1}_\sigma(\R^+, \widetilde{H}^{\eta}(\Gamma))^n$ is the regular part of the singular expansion of $\pmb{\Phi}$.}

{b) Let $\pmb{\Psi}$ be the solution to the hypersingular integral equation \eqref{hypersingeq} and  $\pmb{\Psi}_{h,\Delta t} \in \left({V_{\Delta t,p}}\otimes X^{0}_{h,p}\right)^n$ the best approximation in the norm of ${H}^{r}_\sigma(\R^+, \widetilde{H}^{\frac{1}{2}-s}(\Gamma))^n$ to $\pmb{\Psi}$  on a quasiuniform spatial mesh  with $\Delta t \lesssim h$. Then for $p=1,2,3,\dots$ $$\|\pmb{\Psi}-\pmb{\Psi}_{h, \Delta t}\|_{r,\frac{1}{2}-s, \Gamma, \ast} \lesssim \left(\frac{h}{p^2}\right)^{\tilde{\alpha}+s{-\varepsilon}} +  \left(\frac{\Delta t}{p}\right)^{p+1-r}+ \left(\frac{h}{p}\right)^{-\frac{1}{2}+s+\eta}\ ,$$
 where $r \in [0,p)$, $s\in [0,\frac{1}{2}]$  and $\bold{u}_0 \in {H}^{p+1}_\sigma(\R^+, \widetilde{H}^{\eta}(\Gamma))^n$ is the regular part of the singular expansion of $\pmb{\Psi} = \bf{u}|_\Gamma$.}
\end{corollary}

For the proof, we recall \cite[Theorem 3.1]{heuer01}: 
\begin{lemma}\label{keylemmagradhp}For $\varepsilon > 0$, ${\mathrm{Re}}\ a<1$ and $s \in [-1, \min\{-{\mathrm{Re}}\ a+\frac{1}{2},0\})$ there holds with the piecewise polynomial interpolant of degree $p$, $ \Pi_{r}^{p} r^{-a}$, of $ r^{-a}$ 
$$\|r^{-a} - \Pi_{r}^{p} r^{-a}\|_{s,[0,1],\ast} \lesssim  \left(\frac{h}{(p+1)^2}\right)^{-\mathrm{Re}\ a-s+\frac{1}{2}-\varepsilon} . $$\ 
\end{lemma}
Similarly, for positive powers of $r$ we recall \cite[Theorem 3.1]{heuer2}:
\begin{lemma}\label{keylemmagrad2hp}For $\varepsilon > 0$, $0<{\mathrm{Re}}\ a$ and $s \in [0,{\mathrm{Re}}\ a+\frac{1}{2})$ there holds with the piecewise polynomial interpolant of degree $p+1$, $ \Pi_{r}^{p+1} r^{a}$, of $ r^{a}$
$$\|r^{a} - \Pi_{y}^{p+1} r^a\|_{s,[0,1],\ast} \lesssim \left(\frac{h}{p^2}\right)^{\min\{\mathrm{Re}\ a-s+\frac{1}{2}, 2-s\}{-\varepsilon}} . $$
\end{lemma}

\begin{proof}[Proof of Theorem \ref{approxtheorem1}]{
For the proof of part {a)}, we focus on the case of the wedge singularity, as the approximation of the singular function on the cone closely follows the proof in \cite{graded}.\\

We choose $\pmb{\Phi}_{h, \Delta t} = \Pi_x^{p} \Pi_t^{{p}}\textbf{p}(\textbf{u})$. Using the decomposition \eqref{decompositionEdge} for $\textbf{p}(\textbf{u})$, we can separate the singular and regular parts on the rectangular mesh:
\begin{align*}&\|p_i(\textbf{u})  - \Pi_x^{p} \Pi_t^{{p}}p_i(\textbf{u}) \|_{r, -\frac{1}{2}, \Gamma, \ast} \leq \|b_i(\phi,z,t) r^{{\nu^\ast}-1} - \Pi_t^{{p}} \Pi_x^{p} b_i(\phi,z,t) r^{{\nu^\ast}-1}\|_{r, -\frac{1}{2}, \Gamma,\ast}  + \|{\phi}_{i,0} - \Pi_t^{{p}} \Pi_x^{p}{{\phi}_{i,0}}\|_{r, -\frac{1}{2}, \Gamma,\ast}\\ & \leq \|b_i(\phi,z,t) r^{{\nu^\ast}-1} -\Pi_t^{{p}} b_i(\phi,z,t) r^{{\nu^\ast}-1}\|_{r, -\frac{1}{2}, \Gamma,\ast}+\|\Pi_t^{{p}} b_i(\phi,z,t) r^{{\nu^\ast}-1}- \Pi_t^{{p}} \Pi_x^{p} b_i(\phi,z,t) r^{{\nu^\ast}-1}\|_{r, -\frac{1}{2}, \Gamma,\ast} \\ & \qquad + \|{\phi}_{i,0} - \Pi_t^{{p}} \Pi_x^{p}{{\phi}_{i,0}}\|_{r, -\frac{1}{2}, \Gamma,\ast}\\ &
\leq \|b_i(\phi,z,t) -\Pi_t^{{p}} b_i(\phi,z,t)\|_{r, \epsilon-\frac{1}{2}} \|r^{{\nu^\ast}-1}\|_{-\varepsilon, I, \ast}  +  \|\Pi_t^{{p}} b_i(\phi,z,t) r^{{\nu^\ast}-1}- \Pi_t^{{p}} \Pi_z^{p} b_i(\phi,z,t) r^{{\nu^\ast}-1}\|_{r, -\frac{1}{2}, \Gamma,\ast} \\ & \qquad + \| \Pi_t^{{p}} \Pi_z^{p} b_i(\phi,z,t) r^{{\nu^\ast}-1}-\Pi_t^{{p}} \Pi_z^{p} b_i(\phi,z,t) \Pi_y^{p} r^{{\nu^\ast}-1}\|_{r, -\frac{1}{2}, \Gamma,\ast}+
\|{\phi}_{i,0} - \Pi_t^{{p}} \Pi_x^{p}{{\phi}_{i,0}}\|_{r, -\frac{1}{2}, \Gamma,\ast}\ .
\end{align*}
In the second term we used $\Pi_x^{p} =  \Pi_z^{p} \Pi_r^{p}$. The first term was estimated using Lemma \ref{lemma3.3}, and the result is now bounded by $$\|b_i(\phi,z,t) -\Pi_t^{{p}} b_i(\phi,z,t)\|_{r, \epsilon-\frac{1}{2}} \lesssim \left(\frac{ \Delta t}{p+1}\right)^{p+1-r}   \|b_i(\phi,z,t)\|_{p+1, \epsilon-\frac{1}{2}}\ .$$
Using Lemma \ref{lemma3.3}, we obtain for the second and third terms:
\begin{align*}&
 \|\Pi_t^{{p}} b_i(\phi,z,t) r^{{\nu^\ast}-1}- \Pi_t^{{p}} \Pi_z^{p} b_i(\phi,z,t) r^{{\nu^\ast}-1}\|_{r, -\frac{1}{2}, \Gamma,\ast} + \| \Pi_t^{{p}} \Pi_z^{p} b_i(\phi,z,t) r^{{\nu^\ast}-1}-\Pi_t^{{p}} \Pi_z^{p} b_i(\phi,z,t) \Pi_y^{p} r^{{\nu^\ast}-1}\|_{r, -\frac{1}{2}, \Gamma,\ast}\\ & \lesssim \|\Pi_t^{{p}} b_i(\phi,z,t) - \Pi_t^{{p}} \Pi_z^{p} b_i(\phi,z,t)\|_{r, \varepsilon - \frac{1}{2} } \|r^{{\nu^\ast}-1}\|_{-\varepsilon, I, \ast}  +\| \Pi_t^{{p}} \Pi_z^{p} b_i(\phi,z,t)\|_{r,0} \|r^{{\nu^\ast}-1} - \Pi_r^{p}r^{{\nu^\ast}-1}\|_{-\frac{1}{2}, I, \ast } \ .
\end{align*}}
We finally note that $$\|r^{{\nu^\ast}-1} - \Pi_{r}^{p} r^{{\nu^\ast}-1}\|_{\frac{1}{2}, I, \ast} \lesssim  \left(\frac{h}{(p+1)^2}\right)^{\mathrm{Re}\ {\nu^\ast} -\varepsilon}$$ from Lemma \ref{keylemmagradhp}, as well as $$\|\Pi_t^{{p}} b_i(\phi,z,t) - \Pi_t^{{p}} \Pi_z^{p} b_i(\phi,z,t)\|_{r, \varepsilon - \frac{1}{2} } \lesssim \left(\frac{h}{p+1}\right)^{\frac{1}{2}+k- \epsilon}\|b_i(\phi,z,t)\|_{r,k} \ . $$ 
When the regular part $\pmb{\phi}_{0}$ in \eqref{decompositionEdge} is $H^\eta$ in space, we obtain from the approximation properties \cite{hp}: 
\begin{align*}\|{\phi}_{i,0} - \Pi_t^{{p}} \Pi_x^{p}{{\phi}_{i,0}}\|_{r, -\frac{1}{2}, \Gamma,\ast} &\lesssim_\sigma   \Big(\left(\frac{\Delta t}{p+1}\right)^{p+1-r}+ \left(\frac{h}{p+1}\right)^{1/2+\eta}\Big)\|{\phi}_{i,0}\|_{p+1, \eta, \Gamma}\ .
\end{align*}
Combining these estimates, the asserted estimate follows for $\Delta t \lesssim h$ \\
$$\|\textbf{p}(\textbf{u}) - \Pi^p_x \Pi^p_t \textbf{p}(\textbf{u})\|_{r, -\frac{1}{2}, \Gamma, \ast} \lesssim \left(\frac{h}{(p+1)^2}\right)^{\mathrm{Re}\ {\nu^\ast}{-\varepsilon}} +  \Big(\left(\frac{\Delta t}{p+1}\right)^{p+1-r}+ \left(\frac{h}{p+1}\right)^{1/2+\eta}\Big)\|{\phi}_{i,0}\|_{p+1, \eta, \Gamma}\ .$$
The approximation of the Dirichlet trace $\bold{u}|_\Gamma$ to prove part {b)} follows the above arguments.
\end{proof}
The approximation argument extends from rectangular to triangular elements as in \cite{disspetersdorff}.\\

A similar estimate is obtained for $V_{\Delta t,1}$, with $(\Delta t)^p$ replaced by $\Delta t$.

\section{Algorithmic details}\label{sec:algo}

The numerical experiments in Section \ref{sec:numer} consider the two-dimensional case, therefore in the following we keep the dimension $n=2$ fixed. We introduce the set $\left\lbrace w^{(p)}_m(\textbf{x})\right\rbrace_{m=1}^{M^{(p)}_{h}}$, containing the basis functions of the space $X^{-1}_{h,p}$, which are piecewise polynomials depending on the Lagrangian polynomials on each element $e_i$. Similarly, the set $\left\lbrace w^{(q)}_m(\textbf{x})\right\rbrace_{m=1}^{M^{(q)}_{h}}$ corresponds to a basis of the functional space $X^{0}_{h,q}$. For the time discretization we choose piecewise constant basis functions for the approximation of $\pmb{\Phi}$ ($s_p=0$),
$$v^{(0)}_n(t)=H[t-t_n]-H[t-t_{n+1}], \quad n=0,...,N_{\Delta t}-1,$$
and linear basis functions for the approximation of $\pmb{\Psi}$ ($s_q=1$),
$$v^{(1)}_n(t)=\frac{t-t_n}{\Delta t} H[t-t_n]-\frac{t-t_{n+1}}{\Delta t}H[t-t_{n+1}], \quad n=0,...,N_{\Delta t}-1.$$
Hence, the components of the discrete functions $\pmb{\Phi}_{h,\Delta t}$ and $\pmb{\Psi}_{h,\Delta t}$ can be expressed in space and time as
$$\pmb{\Phi}_{i, h,\Delta t}(\textbf{x},t)=\sum_{n=0}^{N_{\Delta t}-1}\sum_{m=1}^{M^{(p)}_{h}}\alpha^{i}_{nm}w^{(p)}_m(\textbf{x})v^{(0)}_n(t),\quad i=1,2,$$
and 
$$\pmb{\Psi}_{i, h, \Delta t}(\textbf{x},t)=\sum_{n=0}^{N_{\Delta t}-1}\sum_{m=1}^{M^{(q)}_{h}}\beta^{i}_{nm}w^{(q)}_m(\textbf{x})v^{(1)}_n(t),\quad i=1,2,$$
The space-time Galerkin equation \eqref{energetic weak formulationh} leads to the linear system 
\begin{equation}\label{linear system}\left(
\begin{array}{ccccc}
{E}_{\mathcal{V}}^{(0)} & 0 & 0 & \cdots & 0\\
{E}_{\mathcal{V}}^{(1)} & {E}_{\mathcal{V}}^{(0)} & 0 & \cdots & 0\\
{E}_{\mathcal{V}}^{(2)} & {E}_{\mathcal{V}}^{(1)} & {E}_{\mathcal{V}}^{(0)} & \cdots & 0\\
\vdots & \vdots & \vdots & \ddots & \vdots\\
{E}_{\mathcal{V}}^{(N_{\Delta t}-1)} & {E}_{\mathcal{V}}^{(N_{\Delta t}-2)} & {E}_{\mathcal{V}}^{(N_{\Delta t}-3)} & \cdots & {E}_{\mathcal{V}}^{(0)}
\end{array}
\right)
\left(
\begin{array}{c}
\pmb{\alpha}_{(0)}\\
\pmb{\alpha}_{(1)}\\
\pmb{\alpha}_{(2)}\\
\vdots \\
\pmb{\alpha}_{(N_{\Delta t}-1)}
\end{array}
\right)
=
\left(
\begin{array}{c}
\textbf{g}_{(0)}\\
\textbf{g}_{(1)}\\
\textbf{g}_{(2)}\\
\vdots \\
\textbf{g}_{(N_{\Delta t}-1)}
\end{array}
\right),
\end{equation}
where for all $l=0,..., N_{\Delta t}-1$ the $l$-th block, the $l$-th entry of the solution vector and the $l$-th entry of the right hand side are organized as 
$${E}_{\mathcal{V}}^{(l)}=
\left(
\begin{array}{cc}
{E}_{\mathcal{V},11}^{(l)} & {E}_{\mathcal{V},12}^{(l)}\\
{E}_{\mathcal{V},21}^{(l)} & {E}_{\mathcal{V},22}^{(l)}
\end{array}
\right),\quad
\begin{array}{ll}
\pmb{\alpha}_{(l)}=
\left(
\begin{array}{cccccc}
\alpha_{l1}^{1} &
\cdots &
\alpha_{lM^{(p)}_{h}}^{1} &
\alpha_{l1}^{2} &
\cdots &
\alpha_{lM^{(p)}_{h}}^{2}
\end{array}
\right)^\top
\\[4pt]
\textbf{g}_{(l)}=
\left(
\begin{array}{cccccc}
\textbf{g}_{l1}^1 &
\cdots &
\textbf{g}_{lM^{(p)}_{h}}^1 &
\textbf{g}_{l1}^2 &
\cdots &
\textbf{g}_{lM^{(p)}_{h}}^2
\end{array}
\right)^\top
\end{array}
.$$
Solving \eqref{linear system} by backsubstitution leads to a marching-on-in-time time stepping scheme (MOT).
To obtain the generic matrix entry of the sub-block $E^{(l)}_{\mathcal{V}}$, where $l=n-\widetilde{n}$ is the nonnegative difference between two time indexes, we can perform an analytical integration in the time variables $t$, obtaining 
\begin{align}
\left(\mathbb{E}_{\mathcal{V},ij}^{(l)}\right)_{\widetilde{m},m}&=\langle V_{ij}w^{(p)}_m \partial_t v^{(0)}_n,w^{(p)}_{\widetilde{m}}v^{(0)}_{\widetilde{n}}\rangle_{L^2(\Sigma)}=-\langle V_{ij}w^{(p)}_m v^{(0)}_n,w^{(p)}_{\widetilde{m}}\partial_t v^{(0)}_{\widetilde{n}}\rangle_{L^2(\Sigma)}\nonumber\\
=&-\sum_{\xi,\varsigma=0}^1(-1)^{\xi+\varsigma}\int_{\Gamma}w^{(p)}_{\widetilde{m}}(\textbf{x})\int_0^{t_{\widetilde{n}+\xi}}\int_{\Gamma}G_{ij}(\textbf{x},\pmb{\xi};t_{\widetilde{n}+\xi},\tau)w^{(p)}_m(\pmb{\xi})H[\tau-t_{n+\varsigma}]d\Gamma_{\pmb{\xi}}d\tau d\Gamma_{\textbf{x}}.\label{integration in t}
\end{align}
Further, it is also possible to compute exactly the integration in $\tau$ of \eqref{integration in t}, leading to the matrix entry
\begin{equation}\label{general matrix elements}
\left({E}_{\mathcal{V},ij}^{(l)}\right)_{\widetilde{m},m}=- \sum_{\xi,\varsigma=0}^1\frac{(-1)^{\xi+\varsigma}}{2\pi \rho}\int_{\Gamma}\int_{\Gamma}w^{(p)}_{\widetilde{m}}(\textbf{x})w^{(p)}_m(\pmb{\xi})\nu^{\mathcal{V}}_{ij}(r;\Delta_{\widetilde{n}+\xi,n+\varsigma})d\Gamma_{\textbf{x}}d\Gamma_{\pmb{\xi}},
\end{equation}
for all $i,j=1,2$; $m,\widetilde{m}=1,...,M^{(p)}_{h}$ and $n,\widetilde{n}=0,...,N_{\Delta t}-1$.
Here, the positive time difference $t_{\widetilde{n}+\xi}-t_{n+\varsigma}=\Delta_{\widetilde{n}+\xi,n+\varsigma}$ and the integration kernel $\nu^{\mathcal{V}}_{ij}$
{\begin{align}
\nu^{\mathcal{V}}_{ij}(r;\Delta_{\widetilde{n}+\xi,n+\varsigma}):=\quad\quad & \nonumber\\
\Delta_{\widetilde{n}+\xi,n+\varsigma}\left(\frac{r_ir_j}{r^4}-\frac{\delta_{ij}}{2r^2}\right)& \left[ \frac{H[c_{\mathtt{P}}\Delta_{\widetilde{n}+\xi,n+\varsigma}-r]}{c_{\mathtt{P}}}\varphi_{\mathtt{P}}(r;\Delta_{\widetilde{n}+\xi,n+\varsigma})- \frac{H[c_{\mathtt{S}}\Delta_{\widetilde{n}+\xi,n+\varsigma}-r]}{c_{\mathtt{S}}}\varphi_{\mathtt{S}}(r;\Delta_{\widetilde{n}+\xi,n+\varsigma})\right]\nonumber\\
+\frac{\delta_{ij}}{2} & \left[ \frac{H[c_{\mathtt{P}}\Delta_{\widetilde{n}+\xi,n+\varsigma}-r]}{c_{\mathtt{P}}^2}\widehat{\varphi}_{\mathtt{P}}(r;\Delta_{\widetilde{n}+\xi,n+\varsigma})+ \frac{H[c_{\mathtt{S}}\Delta_{\widetilde{n}+\xi,n+\varsigma}-r]}{c_{\mathtt{S}}^2}\widehat{\varphi}_{\mathtt{S}}(r;\Delta_{\widetilde{n}+\xi,n+\varsigma})\right].\label{elastodynamics kernel}
\end{align}}
For each $\gamma={\mathtt{P}},{\mathtt{S}}$ the specific kernel functions are given by
\begin{align}
& {\varphi_{\gamma}(r;\Delta_{\widetilde{n}+\xi,n+\varsigma}):=\sqrt{c^2_{\gamma}\Delta_{\widetilde{n}+\xi,n+\varsigma}^2-r^2}},\label{varphi_gamma}\\
& \widehat{\varphi}_{\gamma}(r;\Delta_{\widetilde{n}+\xi,n+\varsigma}):=\log \left( \sqrt{c^2_{\gamma}\Delta_{\widetilde{n}+\xi,n+\varsigma}^2-r^2}+c_{\gamma}\Delta_{\widetilde{n}+\xi,n+\varsigma} \right)-\log (r).\label{widehat_varphi_gamma}
\end{align}
If $0\leqslant r\leqslant c_{\mathtt{S}}\Delta_{\widetilde{n}+\xi,n+\varsigma}<c_{\mathtt{P}}\Delta_{\widetilde{n}+\xi,n+\varsigma}$ the kernel $\nu_{ij}$ has a reduced form:
\begin{align}
&\nu^{\mathcal{V}}_{ij}(r;\Delta_{\widetilde{n}+\xi,n+\varsigma} )=\nonumber\\
& \frac{c_{\mathtt{P}}^2-c_{\mathtt{S}}^2}{c_{\mathtt{P}} c_{\mathtt{S}}}\left( \frac{r_i r_j}{r^2}-\frac{\delta_{ij}}{2}\right)\frac{\Delta_{\widetilde{n}+\xi,n+\varsigma}}{c_{\mathtt{P}}\sqrt{c_{\mathtt{S}}^2\Delta_{\widetilde{n}+\xi,n+\varsigma}^2-r^2}+c_{\mathtt{S}}\sqrt{c_{\mathtt{P}}^2\Delta_{\widetilde{n}+\xi,n+\varsigma}^2-r^2}} - \frac{c_{\mathtt{P}}^2+c_{\mathtt{S}}^2}{c_{\mathtt{P}}^2 c_{\mathtt{S}}^2}\frac{\delta_{ij}}{2}\log(r)\nonumber\\
& + \frac{\delta_{ij}}{2}\left[\frac{1}{c_{\mathtt{P}}^2}\log\left( c_{\mathtt{P}}\Delta_{\widetilde{n}+\xi,n+\varsigma}+\sqrt{c_{\mathtt{P}}^2\Delta_{\widetilde{n}+\xi,n+\varsigma}^2-r^2}\right)+\frac{1}{c_{\mathtt{S}}^2}\log\left( c_{\mathtt{S}}\Delta_{\widetilde{n}+\xi,n+\varsigma}+\sqrt{c_{\mathtt{S}}^2\Delta_{\widetilde{n}+\xi,n+\varsigma}^2-r^2}\right)\right],\label{kernel reduced form}
\end{align}
with space singularity of kind $\mathcal{O}\left(\log (r) \right)$ for $r\to 0$. This behavior is well-studied for boundary integral operators related to 2D elliptic problems.\\
The discrete function $\pmb{\psi}_{i, h, \Delta t}$ in the weak formulation \eqref{hypersingeqh} produces the linear system ${E}_{\mathcal{W}}\pmb{\beta}=\textbf{h}$,
similar to the one obtained by the discretization of the single layer operator $\mathcal{V}$. In particular, the same Toeplitz structure is obtained in time, and the matrix entries are computed with analytical integrations in time variables, similar to those adopted in \eqref{integration in t}, leading to
\begin{equation}\label{general matrix elements hypersingular}
\left({E}_{\mathcal{W},ij}^{(l)}\right)_{\widetilde{m},m}=- \sum_{\xi,\varsigma=0}^1\frac{(-1)^{\xi+\varsigma}}{2\pi \rho {\Delta t}^2}\int_{\Gamma}\int_{\Gamma}w^{(q)}_{\widetilde{m}}(\textbf{x})w^{(q)}_m(\pmb{\xi})\nu^{\mathcal{W}}_{ij}(r;\Delta_{\widetilde{n}+\xi,n+\varsigma})d\Gamma_{\textbf{x}}d\Gamma_{\pmb{\xi}},
\end{equation}
for all $i,j=1,2$; $m,\widetilde{m}=1,...,M^{(q)}_{h}$ and $n,\widetilde{n}=0,...,N_{\Delta t}-1$.
Here,  $t_{\widetilde{n}+\xi}-t_{n+\varsigma}=\Delta_{\widetilde{n}+\xi,n+\varsigma}$ and the integration kernel $\nu^{\mathcal{W}}_{ij}$ is
{\begin{align}
&\nu^{\mathcal{W}}_{ij}(r;\Delta_{\widetilde{n}+\xi,n+\varsigma} )=\nonumber\\
&\frac{H[c_{\mathtt{P}}\Delta_{\widetilde{n}+\xi,n+\varsigma}-r]}{c_{\mathtt{P}}^3}\left[ \left( D_{\varphi,c_{\mathtt{P}}}^{ij}+D_{c_{\mathtt{P}}}^{ij}\frac{\Delta_{\widetilde{n}+\xi,n+\varsigma}^2c_{\mathtt{P}}^2}{r^2}\right)\frac{\Delta_{\widetilde{n}+\xi,n+\varsigma}\:\varphi_{\mathtt{P}}(r;\Delta_{\widetilde{n}+\xi,n+\varsigma})}{r^2}+D_{\widehat{\varphi},c_{\mathtt{P}}}^{ij}\frac{\widehat{\varphi}_{\mathtt{P}}(r;\Delta_{\widetilde{n}+\xi,n+\varsigma})}{c_{\mathtt{P}}}\right]\nonumber\\
-&\frac{H[c_{\mathtt{S}}\Delta_{\widetilde{n}+\xi,n+\varsigma}-r]}{c_{\mathtt{S}}^3}\left[ \left( D_{\varphi,c_{\mathtt{S}}}^{ij}+D_{c_{\mathtt{S}}}^{ij}\frac{\Delta_{\widetilde{n}+\xi,n+\varsigma}^2c_{\mathtt{S}}^2}{r^2}\right)\frac{\Delta_{\widetilde{n}+\xi,n+\varsigma}\:\varphi_{\mathtt{S}}(r;\Delta_{\widetilde{n}+\xi,n+\varsigma})}{r^2}+D_{\widehat{\varphi},c_{\mathtt{S}}}^{ij}\frac{\widehat{\varphi}_{\mathtt{S}}(r;\Delta_{\widetilde{n}+\xi,n+\varsigma})}{c_{\mathtt{S}}}\right],\label{kernel reduced form hypersingular}
\end{align}}
where the coefficients {$D_{\varphi,c_{\gamma}}^{ij}$, $D_{c_{\gamma}}^{ij}$ and $D_{\widehat{\varphi},c_{\gamma}}^{ij}$} are defined in the {Appendix A.3 of \cite{Dicredico2022}}. If $0\leqslant r\leqslant c_{\mathtt{S}}\Delta_{\widetilde{n}+\xi,n+\varsigma}<c_{\mathtt{P}}\Delta_{\widetilde{n}+\xi,n+\varsigma}$ the kernel $\nu^{\mathcal{W}}_{ij}$ has a reduced form with singularity $\mathcal{O}(r^{-2})$ for $r\to 0$.\\
We also have to take into account that both kernels $\nu^{\mathcal{V}}_{ij}$ and $\nu^{\mathcal{W}}_{ij}$ depend on the difference $c_{\gamma}^2\Delta_{\widetilde{n}+\xi,n+\varsigma}^2-r^2$ through the Heaviside functions, which lead to a jump at the points where the argument vanishes. 
{To overcome this issue related to the possible presence of one or two wave fronts which can reduce the integration domain in local space variables, we apply to the latter a suitable decomposition. This splitting procedure drastically reduces the number of quadrature nodes required to achieve single precision accuracy \cite{Aimi20}.}\\
{Moreover, to numerically evaluate \eqref{general matrix elements} and \eqref{general matrix elements hypersingular}, we employ specific quadrature rules to treat the singularities of the kernels $\nu^{\mathcal{V}}_{ij}$ and $\nu^{\mathcal{W}}_{ij}$ defined in \eqref{kernel reduced form} and \eqref{kernel reduced form hypersingular}. The interested reader is refered to \cite{Aimi20} for a detailed description of the applied quadrature schemes in case of the integration of the weakly singular kernel. For the numerical evaluation of the hypersingular integrals we refer to \cite{Dicredico2022}.}

\section{Numerical results}\label{sec:numer}

The numerical experiments in this section consider $h$, $p$ and $hp$ discretizations for the soft scattering problem \eqref{energetic weak formulation} (Sections \ref{flat_ob_single_layer}-\ref{polygonal_ob_single_layer}) and the hard scattering problem \eqref{hypersingeq} (Section \ref{flat_ob_hypersingular}). They illustrate the singular behavior of the solution near the crack tip and the theoretically expected convergence rates.\\
Unless  stated otherwise, for the $h$ version on uniform or graded meshes piecewise constant basis functions in space and time are chosen to approximate the solution of the Dirichlet problem \eqref{energetic weak formulationh}. Piecewise linear functions are used for the Neumann problem \eqref{hypersingeqh}. The $p$ and $hp$ versions are implemented with higher polynomial degrees in space, up to $p=7$. The Lam\'e parameters and the mass density, where it is not otherwise specified, are set to be $\lambda=2$, $\mu=1$ and $\varrho=1$ for all the results presented in this section.\\
{All the numerical results for the Dirchlet problem are computed for a prescribed right hand side $\widetilde{\textbf{g}} = (\mathcal{K}+1/2) \textbf{g}$ in \eqref{energetic weak formulation}. While the analysis in Sections \ref{regularitysection} and \ref{approxsection} relies on this form of $\widetilde{\textbf{g}}$, as typical in the BEM literature, for numerical convenience we directly prescribe $\widetilde{\textbf{g}}$. Analogously, for the Neumann problem we prescribe $\widetilde{\textbf{h}} = (\mathcal{K}'-1/2) \textbf{h}$. Also, we set the weight $\sigma=0$.}\\ 

\subsection{Soft scattering problems on flat obstacle}\label{flat_ob_single_layer}
\noindent \textbf{Example 1.} Here we consider the discrete weakly singular integral equation \eqref{energetic weak formulationh} on a flat obstacle $\Gamma=\left\lbrace (x,0)\in\mathbb{R}\:\vert\: x\in[-0.5,0.5]\right\rbrace$ up to time $T=1$. The Dirichlet datum corresponds to $\widetilde{g}_i(x,t)=\widetilde{g}(x,t)=H[t]f(t)x^4$, $i=1,2$,
where the function 
\begin{equation}\label{function f}
f(t)=
\left\lbrace
\begin{array}{ll}
\sin^2(4\pi t), & t\in[0,1/8]\\
1, & t> 1/8
\end{array}
\right.
\end{equation}
is a temporal profile that leads to an exact solution $\pmb{\Phi}$ which becomes static in time. 
\begin{figure}[h!]
\centering
\includegraphics[height=5cm, width=5cm]{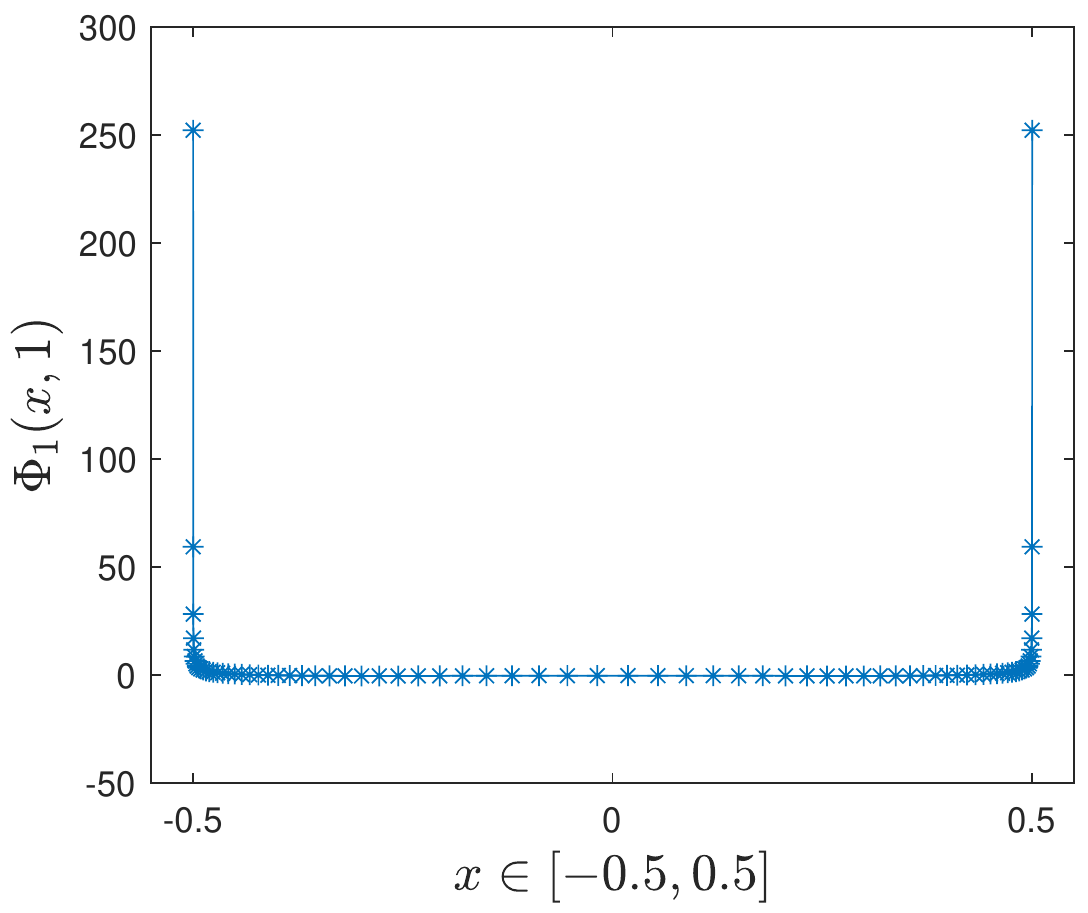}
\caption{Horizontal component of $\pmb{\Phi}$, calculated on the obstacle $\Gamma$ at the time instant $T=1$. This plot is obtained from the $h$ version on a $3$-graded mesh with $81$ nodes and $\Delta t=0.00625$.}
\label{fig:horizontal_component_end}
\end{figure}
In Figure \ref{fig:horizontal_component_end}, the horizontal component of the discrete solution $\pmb{\Phi}$ of \eqref{energetic weak formulationh} is represented on the obstacle $\Gamma$ at a fixed time instant: as we can observe from the plot, the behaviour of the solution is singular near the crack tips. Tables \ref{tab: 1}, \ref{tab: 2} and \ref{tab: 3} contain the values $\pmb{\alpha}^\top {E}_{\mathcal{V}} \pmb{\alpha}$,
{namely the squared energy norm of the Galerkin solution, as the {number of} spatial degrees of freedom (DOF) is increased (see Section \ref{sec:algo} for details about the construction of the vector $\pmb{\alpha}$ and the matrix ${E}_{\mathcal{V}}$). This number, in particular, corresponds to the $L^2(\Sigma)$ product at the left hand side of 
\eqref{energetic weak formulationh} with the discrete solution {$\pmb{\Phi}_{h,\Delta t}$} replacing the test function. {For simplicity}, in the following tables the number of DOF is indicated only for one component of the vector{-valued} solution}. The values reported in \ref{tab: 1} are obtained by applying a \textit{p} version in space: the boundary is uniformly discretized with segments of length $h=0.1$, while the degree $p$ of the space basis function is increased. For $p=1$ we set the time step $\Delta t=0.025$ and we halve it whenever $p$ increases.

\begin{table}[ht!]
\caption{Energy norm squared of the approximate solution for $T=1$ ($p$ version)}
\centering
\begin{tabular}{|c||ccccccc|}
\hline
\small{degree $p$} & \scriptsize{$1$} & \scriptsize{$2$} & \scriptsize{$3$} & \scriptsize{$4$} & \scriptsize{$5$} & \scriptsize{$6$} & \scriptsize{$7$} \\ \hline
\small{DOF} & \scriptsize{$11$} & \scriptsize{$21$} & \scriptsize{$31$} & \scriptsize{$41$} & \scriptsize{$51$} & \scriptsize{$61$} & \scriptsize{$71$}\\ \hline
$\Delta t$ & \scriptsize{$2.50000\cdot 10^{-2}$} & \scriptsize{$1.25000\cdot 10^{-2}$} & \scriptsize{$6.25000\cdot 10^{-3}$} & \scriptsize{$3.12500\cdot 10^{-3}$} & \scriptsize{$1.56250\cdot 10^{-3}$} & \scriptsize{$7.81250\cdot 10^{-4}$} & \scriptsize{$3.90625\cdot 10^{-4}$} \\ \hline
\small{$\pmb{\alpha}^\top {E}_{\mathcal{V}} \pmb{\alpha}$} & \scriptsize{$3.4108\cdot 10^{-2}$} & \scriptsize{$3.6257\cdot 10^{-2}$} & \scriptsize{$3.7012\cdot 10^{-2}$} & \scriptsize{$3.7338\cdot 10^{-2}$} & \scriptsize{$3.7511\cdot 10^{-2}$} & \scriptsize{$3.7615\cdot 10^{-2}$} & \scriptsize{$3.7684\cdot 10^{-2}$}\\ 
\hline
\end{tabular}
\label{tab: 1}
\end{table}
The energy values reported in Table \ref{tab: 2} refer to the solution of the problem with the \textit{h} version: we fix an algebraically graded mesh on the arc as in \eqref{gradedmesh}, for given grading parameter ${\tilde{\beta}} = 1,\ 2,\ 3$ and number of mesh points $2N+1$.
In Table \ref{tab: 3} the discretization method used is the \textit{hp} version. We set on $\Gamma$ the mesh points geometrically graded, as indicated in the rule
\begin{equation}\label{geometric points}
\left\lbrace
\begin{array}{lc}
x_{0,L}=-\frac{1}{2},\: x_{L,j}=\frac{1}{2}\left(\sigma^{N+1-j}-1\right)& j=1,\ldots,N+1\\[4pt]
x_{N+1}=\frac{1}{2},\:x_{R,j}=\frac{1}{2}\left(1-\sigma^{j}\right), & j=1,\ldots,N
\end{array}
\right. ,
\end{equation}
with $\sigma=0.2,\ 0.5$ and, for ease of programming, at each refinement of the mesh the degree $p$ increases uniformly on all the space elements. The parameter $L_{\sigma}$ in the table represents the length of the smallest segment of the mesh.

\begin{table}[ht!]
\caption{Energy norm squared of the approximate solution for $T=1$ (\textit{h} version with algebraically graded mesh)}
\centering
\begin{tabular}{|c|c|c|c|c|}
\hline
$\Delta t$ & DOF & $\pmb{\alpha}^\top {E}_{\mathcal{V}} \pmb{\alpha}$, ${\tilde{\beta}}=1$  & $\pmb{\alpha}^\top {E}_{\mathcal{V}} \pmb{\alpha}$, ${\tilde{\beta}}=2$ & $\pmb{\alpha}^\top {E}_{\mathcal{V}} \pmb{\alpha}$, ${\tilde{\beta}}=3$ \\ \hline\hline
$1.2500\cdot 10^{-2}$ & $10$ & $3.0143\cdot 10^{-2}$ & $3.6212\cdot 10^{-2}$ & $3.7315\cdot 10^{-2}$ \\
$6.2500\cdot 10^{-3}$ & $20$ & $3.3906\cdot 10^{-2}$ & $3.7501\cdot 10^{-2}$ & $3.7835\cdot 10^{-2}$ \\
$3.1250\cdot 10^{-3}$ & $40$ & $3.5933\cdot 10^{-2}$ & $3.7813\cdot 10^{-2}$ & $3.7903\cdot 10^{-2}$ \\
$1.5625\cdot 10^{-3}$ & $80$ & $3.6943\cdot 10^{-2}$ & $3.7890\cdot 10^{-2}$ & $3.7914\cdot 10^{-2}$ \\
\hline
\end{tabular}
\label{tab: 2}
\end{table}

\begin{table}[ht!]
\caption{Energy norm squared of the approximate solution for $T=1$ (\textit{hp} version geometrically graded)}
\centering
\begin{tabular}{|c|c|c|c|c|c|c|}
\hline
$p$  & $L_{0.5}$ & $L_{0.2}$  & DOF & $\Delta t$ & $\pmb{\alpha}^\top {E}_{\mathcal{V}} \pmb{\alpha}$, $\sigma=0.5$ & $\pmb{\alpha}^\top {E}_{\mathcal{V}} \pmb{\alpha}$, $\sigma=0.2$\\ \hline\hline
$0$ & $5.00000\cdot 10^{-1}$ & $5.0 \cdot 10^{-1}$ & $2$  & $2.500000\cdot 10^{-1}$ & $4.0730\cdot 10^{-3}$ & $4.0730\cdot 10^{-3}$\\
$1$ & $2.50000\cdot 10^{-1}$ & $1.0 \cdot 10^{-1}$ & $5$ & $1.250000\cdot10^{-1}$ & $2.7847\cdot 10^{-2}$ & $3.4521\cdot 10^{-2}$\\
$2$ & $1.25000\cdot10^{-1}$ & $2.0 \cdot 10^{-2}$ & $13$ & $6.250000\cdot10^{-2}$ & $3.2960\cdot 10^{-2}$ & $3.4581\cdot 10^{-2}$\\
$3$ & $6.25000\cdot10^{-2}$ & $4.0 \cdot 10^{-3}$ & $25$ & $3.125000\cdot10^{-2}$ & $3.6745\cdot 10^{-2}$ & $3.7256\cdot 10^{-2}$\\
$4$ & $3.12500\cdot10^{-2}$ & $8.0 \cdot 10^{-4}$ & $41$ & $1.562500\cdot10^{-2}$ & $3.7618\cdot 10^{-2}$ & $3.7783\cdot 10^{-2}$\\
$5$ & $1.56250\cdot10^{-2}$ & $1.6\cdot 10^{-4}$ & $61$ & $7.812500\cdot10^{-3}$ & $3.7828\cdot 10^{-2}$ & $3.7890\cdot 10^{-2}$\\
$6$ & $7.81250\cdot10^{-3}$ & $3.2\cdot 10^{-5}$ & $85$ & $3.906250\cdot10^{-3}$ & $3.7888\cdot 10^{-2}$ & $3.7911\cdot 10^{-2}$\\
$7$ & $3.90625\cdot10^{-3}$ & $6.4 \cdot 10^{-6}$ & $113$ & $1.953125\cdot10^{-3}$ & $3.7906\cdot 10^{-2}$ & $3.7915\cdot 10^{-2}$\\
\hline
\end{tabular}
\label{tab: 3}
\end{table}
The energy is increasing towards a common benchmark value for the tested discretization methods: to illustrate the related convergence rate, in Figure \ref{fig:energy_segment} the squared error in energy norm is plotted with respect to the spatial DOF. We observe that the decay of the error follows a straight line in the logarithmic plots for both the \textit{p} version and the \textit{h} version with ${\tilde{\beta}}=1$, corresponding to algebraic convergence with rate $2$ (\textit{p}), respectively $1$ (\textit{h}) in terms of DOF. This means that the error tends to $0$ like $p^{-1}$, respectively $h^{1/2}$. This convergence rate is expected from Corollary \ref{approxcor1}. Indeed, by Proposition \ref{DPbounds} the energy $\pmb{\alpha}^\top {E}_{\mathcal{V}} \pmb{\alpha}$ is bounded by the Sobolev norm considered in Corollary  \ref{approxcor1}.   {Analogous results are obtained for the $h$ version with {polynomial degrees $p=1,2$} in space.} 

On algebraically graded meshes with ${\tilde{\beta}}=2$ and $3$ the error similarly decays along a straight line, but of slope $-{\tilde{\beta}}$ with increasing DOF. In particular, the BEM on the graded mesh \eqref{gradedmesh} with ${\tilde{\beta}}=3$ recovers the optimal convergence order $h^{3/2}$ expected in the energy norm for smooth solutions, as in Corollary \ref{approxcor2}. 

The fastest convergence in Figure \ref{fig:energy_segment} is obtained by the \textit{hp} version, for which the error decays faster than a straight line for both $\sigma=0.2,0.5$. The graph of the squared error indicates exponential decay. Convergence is fastest for $\sigma = 0.2$, which is close to the theoretically optimal $\sigma\simeq 0.17$. The nodes in this case are more densely clustered near the endpoints of $\Gamma$ than for $\sigma=0.5$.

\begin{figure}[h!]
\centering
\includegraphics[width=0.45\textwidth]{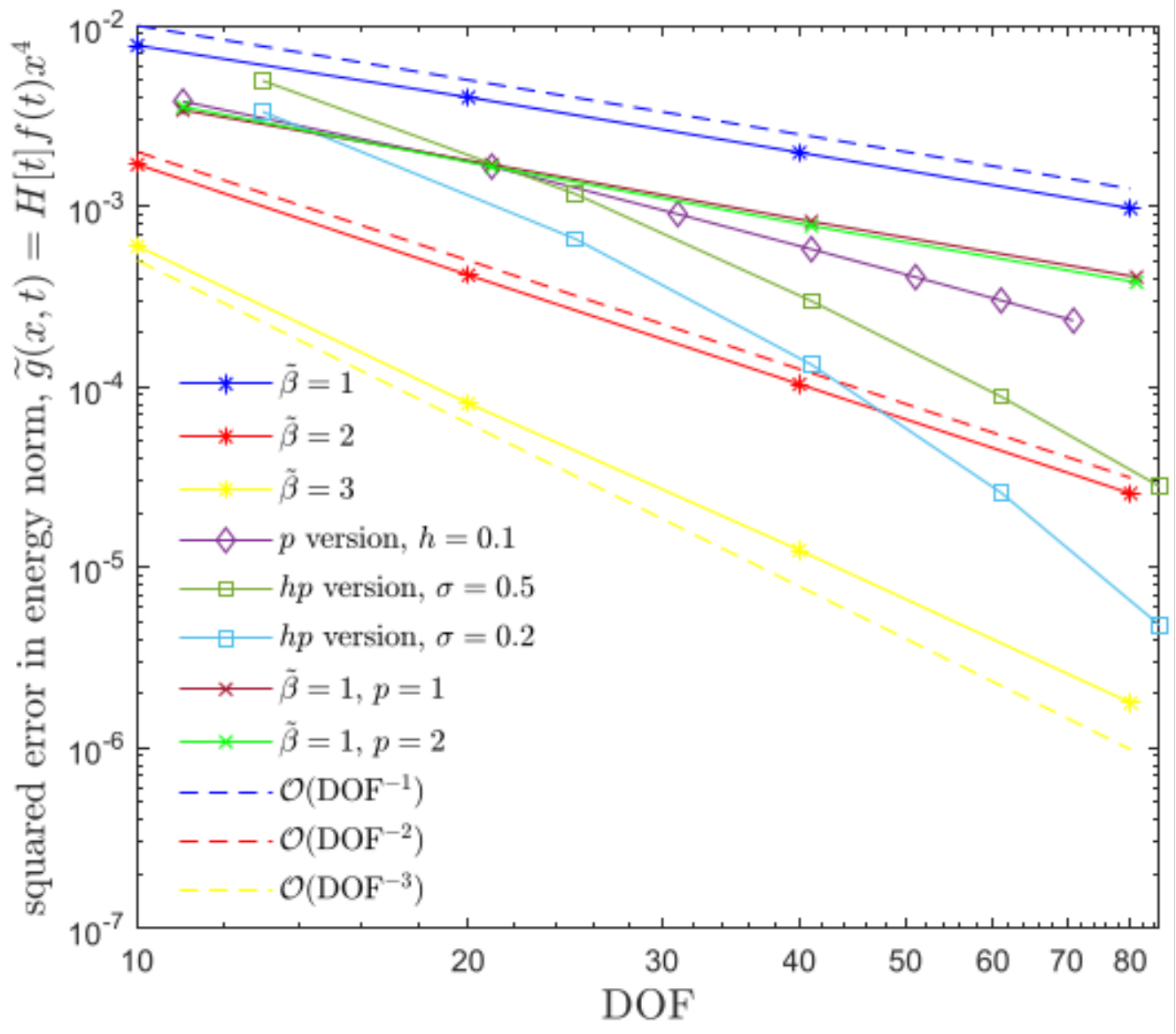}
\caption{Squared error of the energy norm 
for various discretization methods.}
\label{fig:energy_segment}
\end{figure}

To illustrate the singular behavior of the solution, Figure \ref{fig:h_v_components} plots the horizontal and the vertical components of the approximate $\pmb{\Phi}$ with respect to the distance $r$ towards the left end of the arc $(-0.5,0)^\top$  for various time instants: one observes that the singular behavior is independent of time, and the components increase as $\mathcal{O}(r^{-1/2})$ for $r\to 0$. This confirms the discussion in Section \ref{regularity2d}. The solution in this figure is obtained from the \textit{h} version on a $3$-graded mesh with $81$ nodes.
\begin{figure}[h!]
\centering
\subfloat{\includegraphics[width=0.35\textwidth]{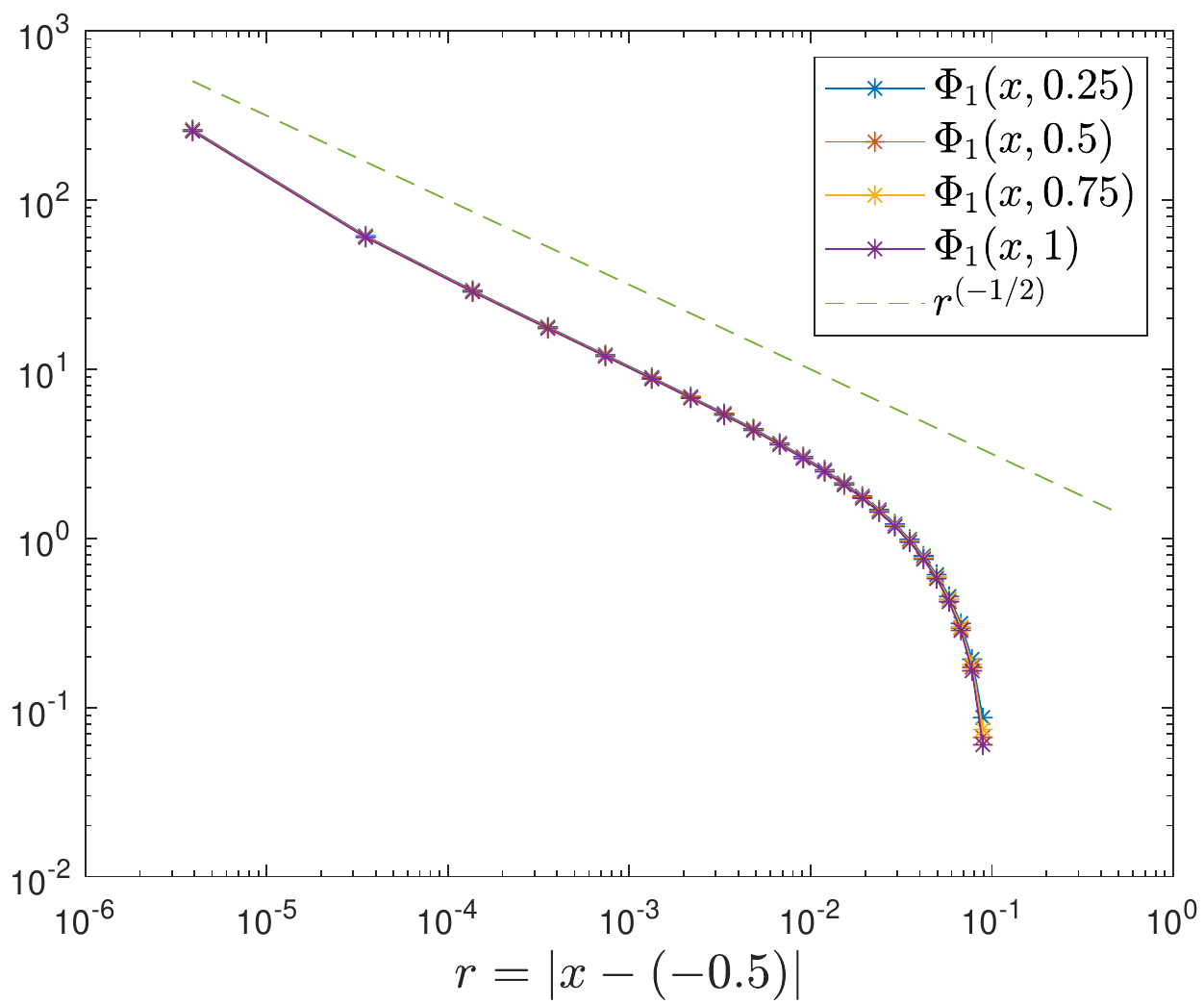}}
\qquad
\subfloat{\includegraphics[width=0.35\textwidth]{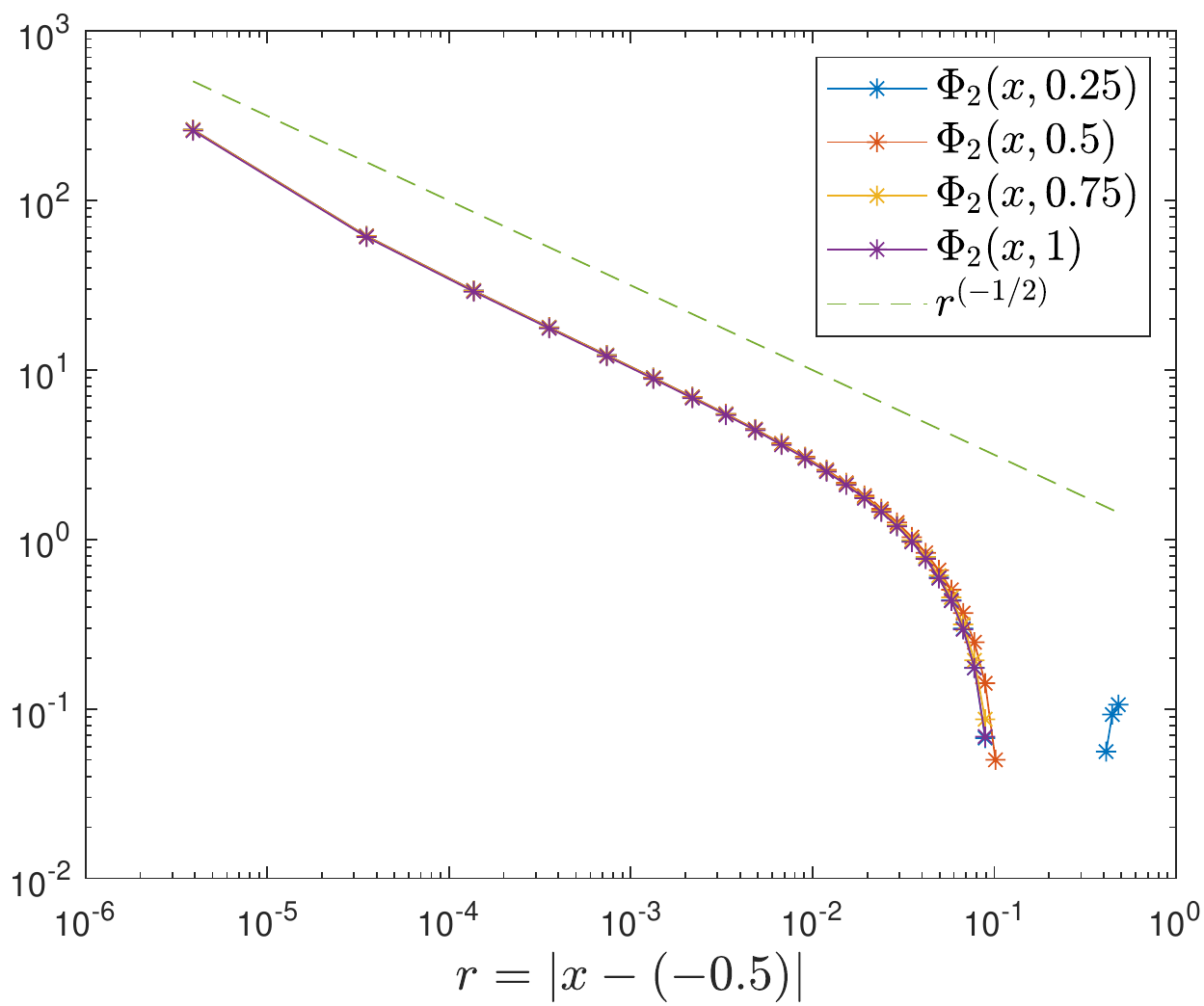}}
\caption{Asymptotic behaviour towards the left end of $\Gamma$}
\label{fig:h_v_components}
\end{figure}
\\

\noindent \textbf{Example 2.} Similar results as in Example 1 are obtained also for other boundary data on a flat obstacle $\Gamma=\left\lbrace (x,0)\in\mathbb{R}\:\vert\: x\in[-0.5,0.5]\right\rbrace$. We here set $\widetilde{g}_i(x,t)=\widetilde{g}(x,t)=H[t]f(t)x$, $i=1,2$, where the function $f(t)$ is the temporal profile defined in \eqref{function f}. The solution of the problem \eqref{energetic weak formulationh} is again singular at the end points of the arc and, as observed in the previous experiment, the components of $\pmb{\Phi}$ increase as $\mathcal{O}(r^{-1/2})$ when the distance $r$ tends to zero (see Figure \ref{fig:second_experiment}).
\begin{figure}[h!]
\centering
\includegraphics[width=0.7\textwidth]{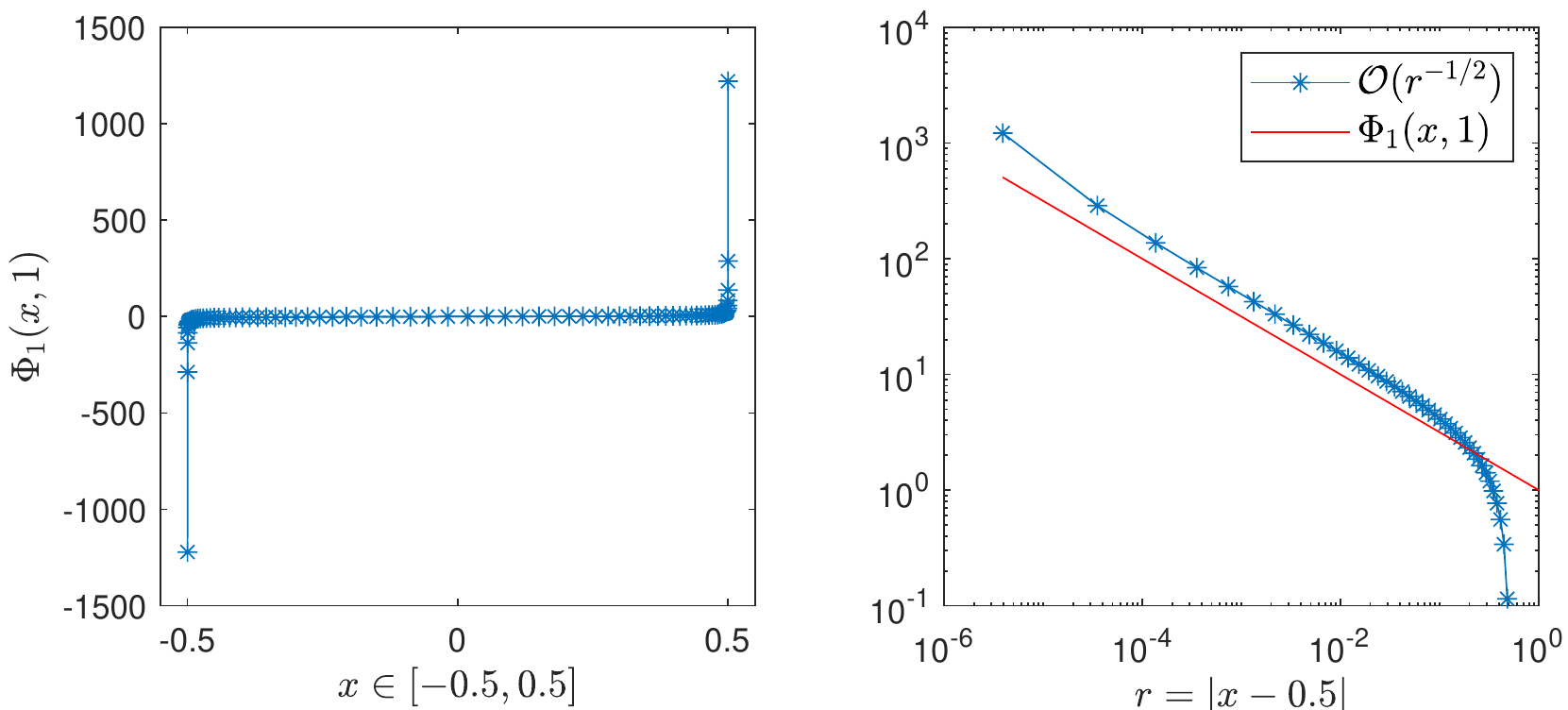}
\caption{Behaviour of the horizontal component $\Phi_1$ on $\Gamma$ and w.r.t the distance towards the right endpoint at $T=1$. Both plots are obtained imposing on $\Gamma$ an algebraically $3$-graded mesh of $80$ segments.}
\label{fig:second_experiment}
\end{figure}
\begin{figure}[h!]
\centering
\includegraphics[width=0.45\textwidth]{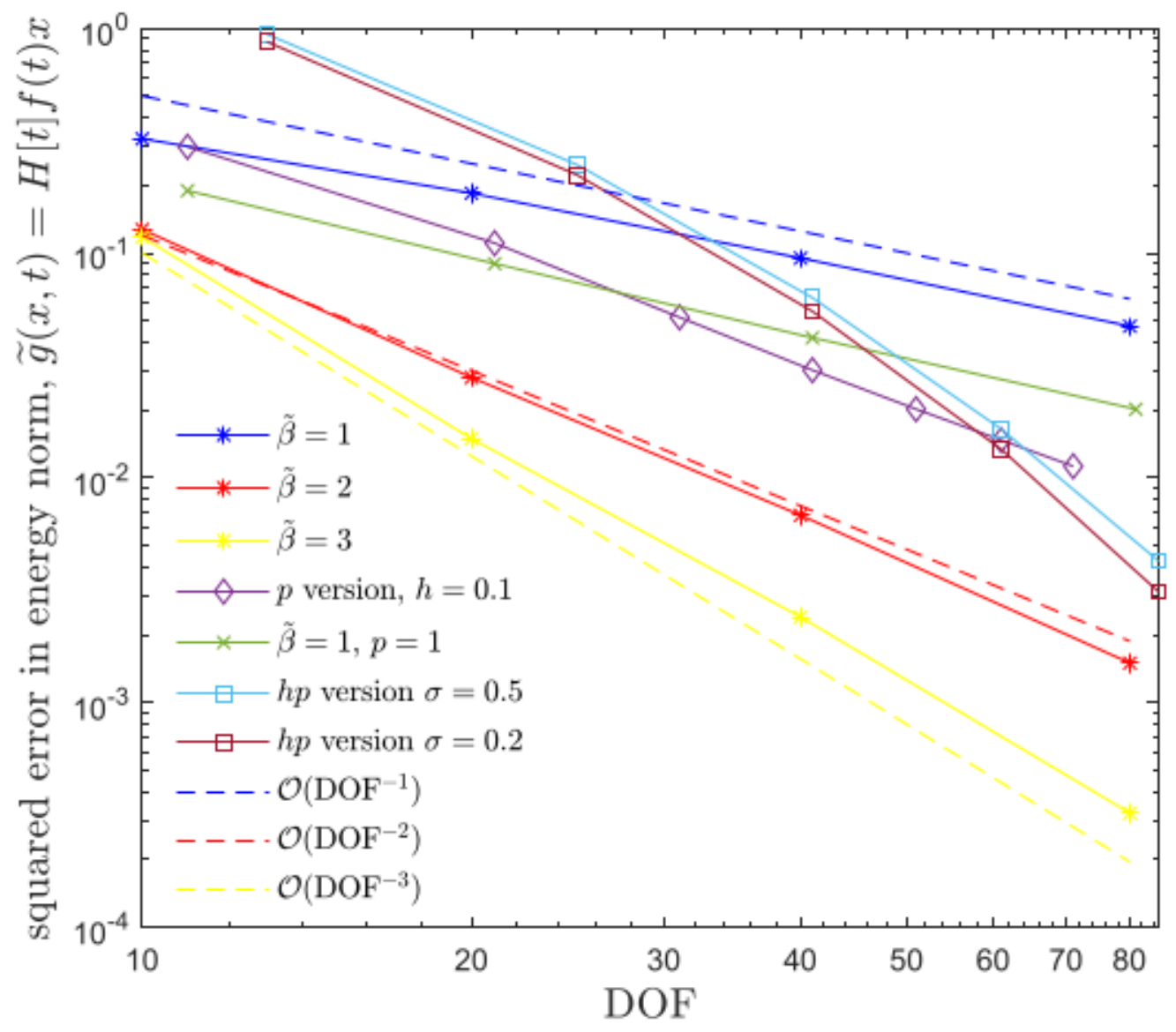}
\caption{Squared error of the energy norm for various discretization methods}
\label{fig:energy_segment_2nd_exp}
\end{figure}

We again study the decay of the error in energy norm for this new Dirichlet condition, leading to similar considerations for the rate of convergence of the different discretization methods. The spatial and temporal discretization parameters for the \textit{h}, \textit{p} and \textit{hp} version are chosen as in the previous experiment. The results are shown in Figure \ref{fig:energy_segment_2nd_exp}. The squared error for the \textit{h} version is $\mathcal{O}\left(h^{{\tilde{\beta}}}\right)$ on the algebraically ${\tilde{\beta}}$-graded mesh, as in Corollary \ref{approxcor2}. The corresponding result for the \textit{p} version is $\mathcal{O}\left(p^{-2}\right)$, in agreement with Corollary \ref{approxcor1}. Faster than algebraic convergence is achieved by the \textit{hp} version on a geometrically graded mesh.

\subsection{Soft scattering problems on polygonal obstacles}\label{polygonal_ob_single_layer}

In the following we consider the weakly singular integral equation \eqref{energetic weak formulation} on different {types of closed obstacles} $\Gamma$, {as shown in Figure \ref{fig:mesh_expected_exponent}(b), where the four considered convex polygonal geometries are collected.}

{Recalling the notation stated in Section \ref{sec 2}, a closed arc $\Gamma$ determines a partition of $\mathbb{R}^2$ made by the bounded interior domain $\Omega'$, with $\partial \Omega'=\Gamma$, and its {complement} $\Omega=\mathbb{R}^2\setminus \overline{\Omega'}$}. From Section \ref{regularity2d} we know that the solution $\textbf{u}$ {in the exterior set $\Omega$} and near a corner point of $\Gamma$ locally behaves like a power of the distance $r$ to the {vertex}: 
{$$u_i\approx C_{i,\omega_{ext}}(t)r^{\nu^*(\omega_{ext})}, \quad\: r\to 0,$$}
where $\omega_{ext}$ {is the considered exterior angle (with {complement} $\omega_{int}$)} and the exponent {$\nu^*(\omega_{ext})$} is the smallest solution of the equation \eqref{eq1.11}, namely
{\begin{equation}\label{equation for the exponents}
\sin^2(\omega_{ext}\: \nu^*)=\left( \frac{\nu^*}{k}\sin{\omega_{ext}}\right)^2,
\end{equation}}
with positive real part, where $k=3-2\lambda/(\lambda+\mu)$.
The prefactor {$C_{i,\omega_{ext}}(t)$} is a smooth function in $t$, independent of $r$, so the leading singular behaviour does not change with time. {The solution $\pmb{\Phi}=\textbf{p}(\textbf{u})\vert_{\Gamma}$} of the boundary integral equation \eqref{energetic weak formulation} represents the traction {at} the obstacle {and, from the discussion in Section \ref{regularitysection},} its asymptotic behaviour {a the vertex can be expressed as
$$\Phi_i\approx \widetilde{C}_{i,\omega_{ext}}(t)r^{\nu^*(\omega_{ext})-1},\: r\to 0.$$} 
For Lam\'e parameters $\lambda=2$, $\mu=1$ and mass density $\varrho=1$, Figure \ref{fig:mesh_expected_exponent}(a) shows the {exterior and interior} exponents, {$\nu^*(\omega_{ext}) = \nu^*(2\pi-\omega_{int})$ and $\nu^*(\omega_{int})$}, as a function of $\omega_{int}$. Red crosses indicate the exponents $\nu^*$ corresponding to the red corners of the polygons depicted on the right of Figure \ref{fig:mesh_expected_exponent}(b), for
interior angles $\frac{7\pi}{24}$ ($\nu^*=0.5372$), $\frac{\pi}{3}$ ($\nu^*=0.5451$), $\frac{3\pi}{8}$ ($\nu^*=0.542$) and $\frac{3\pi}{5}$ ($\nu^*=0.6306$).\\

\noindent \textbf{Example 3.} We consider the Galerkin solution of the weakly singular integral equation \eqref{energetic weak formulationh} on 
the polygons represented in Figure \ref{fig:mesh_expected_exponent}(b) up to time $T=1$. In all cases {the right hand side imposed} is $\widetilde{g}_1(\textbf{x},t)=0,\:\widetilde{g}_2(\textbf{x},t)=H[t]f(t)100\vert x\vert^{9.5}$. An example of the solution produced by the boundary condition is in Figure \ref{fig:mesh_expected_exponent}(c), where the vertical component of $\pmb{\Phi}$ is plotted at the base of the equilateral triangle $\Gamma_1$. The solution is characterized by a high gradient near the corners on the base. 
{The mesh on each side of polygons $\Gamma_i$, $i=1,\ldots,4$, is algebraically graded towards the corners following  \eqref{gradedmesh}, for given grading parameter ${\tilde{\beta}} = 1,\ 2,\ 3$. The polygons $\Gamma_1$ and $\Gamma_4$, which are both equilateral, are discretized with $80$ segments per side, while {for} $\Gamma_2$ and $\Gamma_3$ we use $80$ segments on the two  sides which are of equal length and $75$ and $87$ segments {on} the base, respectively. The time step is chosen as $\Delta t = 0.00625$ for all experiments.}

\begin{figure}[h!]
\centering
\subfloat[]{\includegraphics[width=0.32\textwidth]{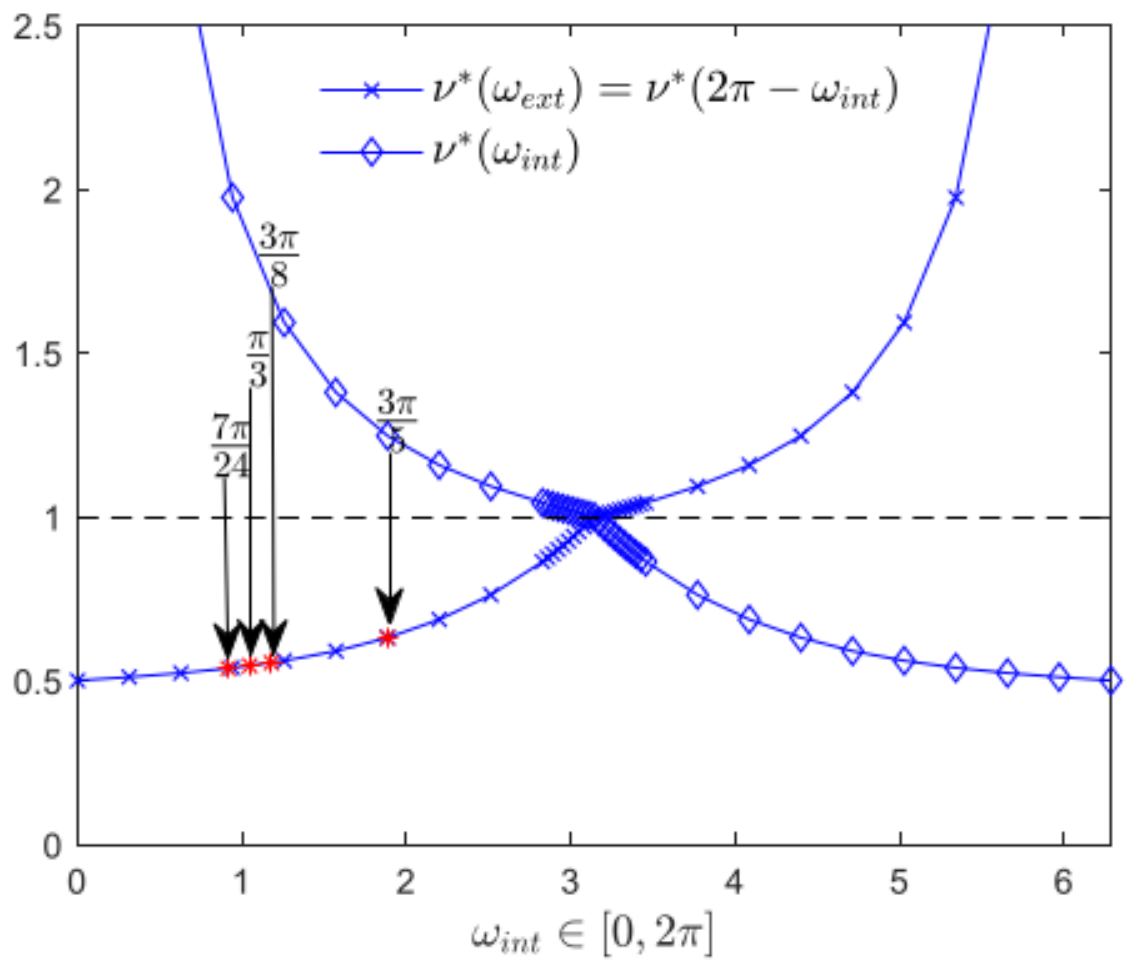}}
\:
\subfloat[]{\includegraphics[width=0.32\textwidth]{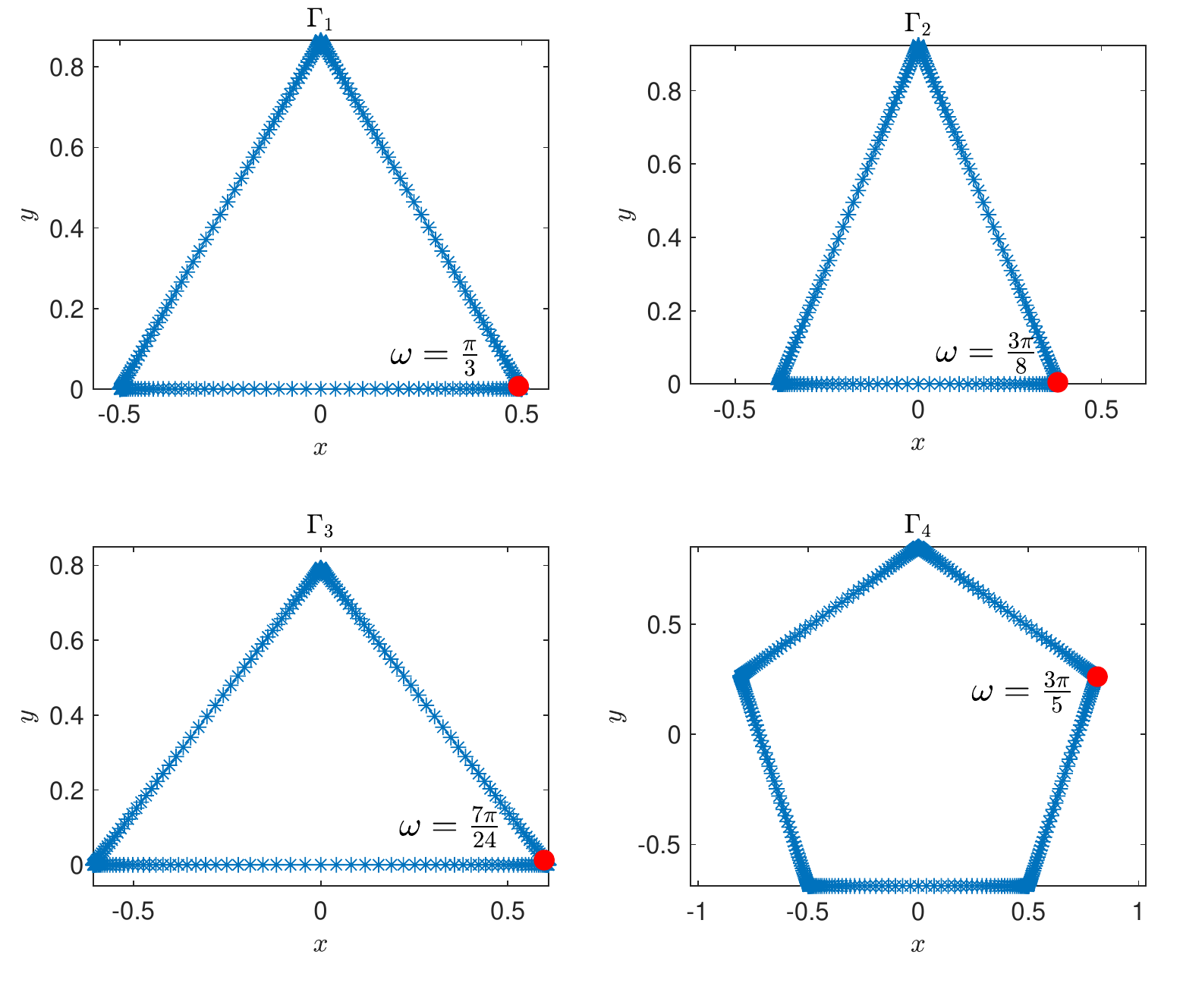}}
\:
\subfloat[]{\includegraphics[width=0.32\textwidth]{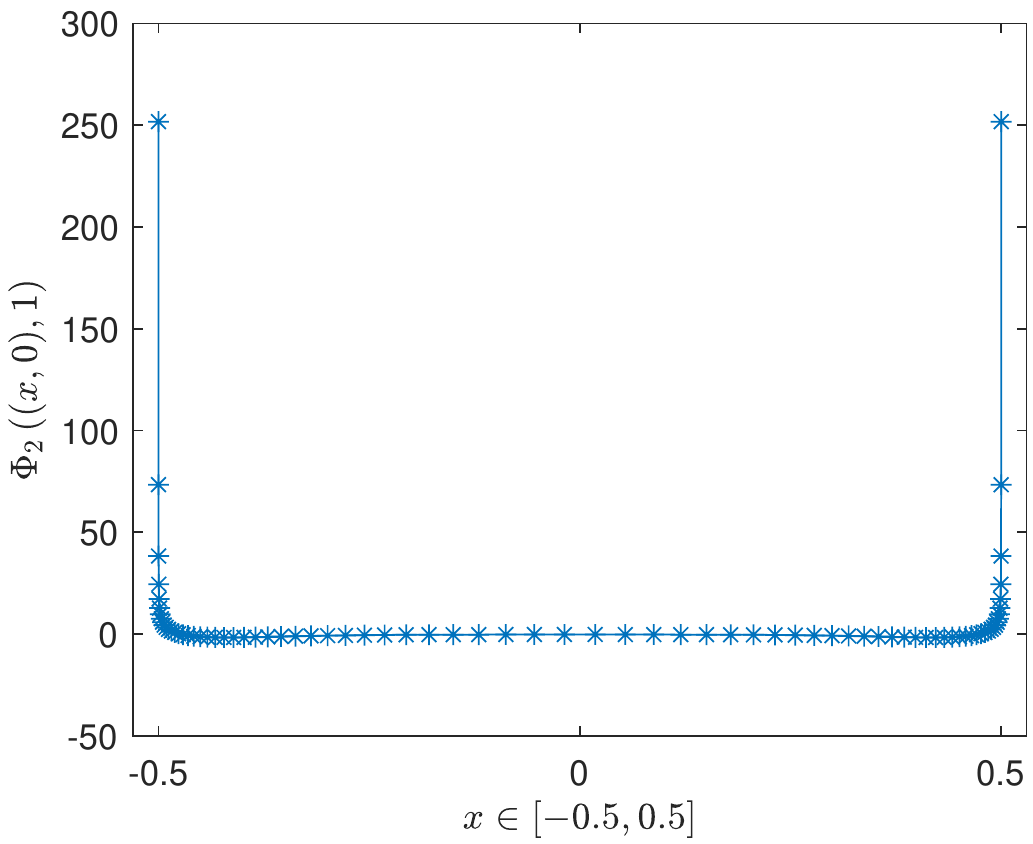}}
\caption{Expected exponent with dependence on $\omega_{int}$ and its complementary ($k=5/3$) and tested polygonal meshes (a and b); plot of the vertical component of $\pmb{\Phi}$ on the base of $\Gamma_1$ (c).}
\label{fig:mesh_expected_exponent}
\end{figure}

\begin{figure}[h!]
\centering
\subfloat{\includegraphics[width=0.35\textwidth]{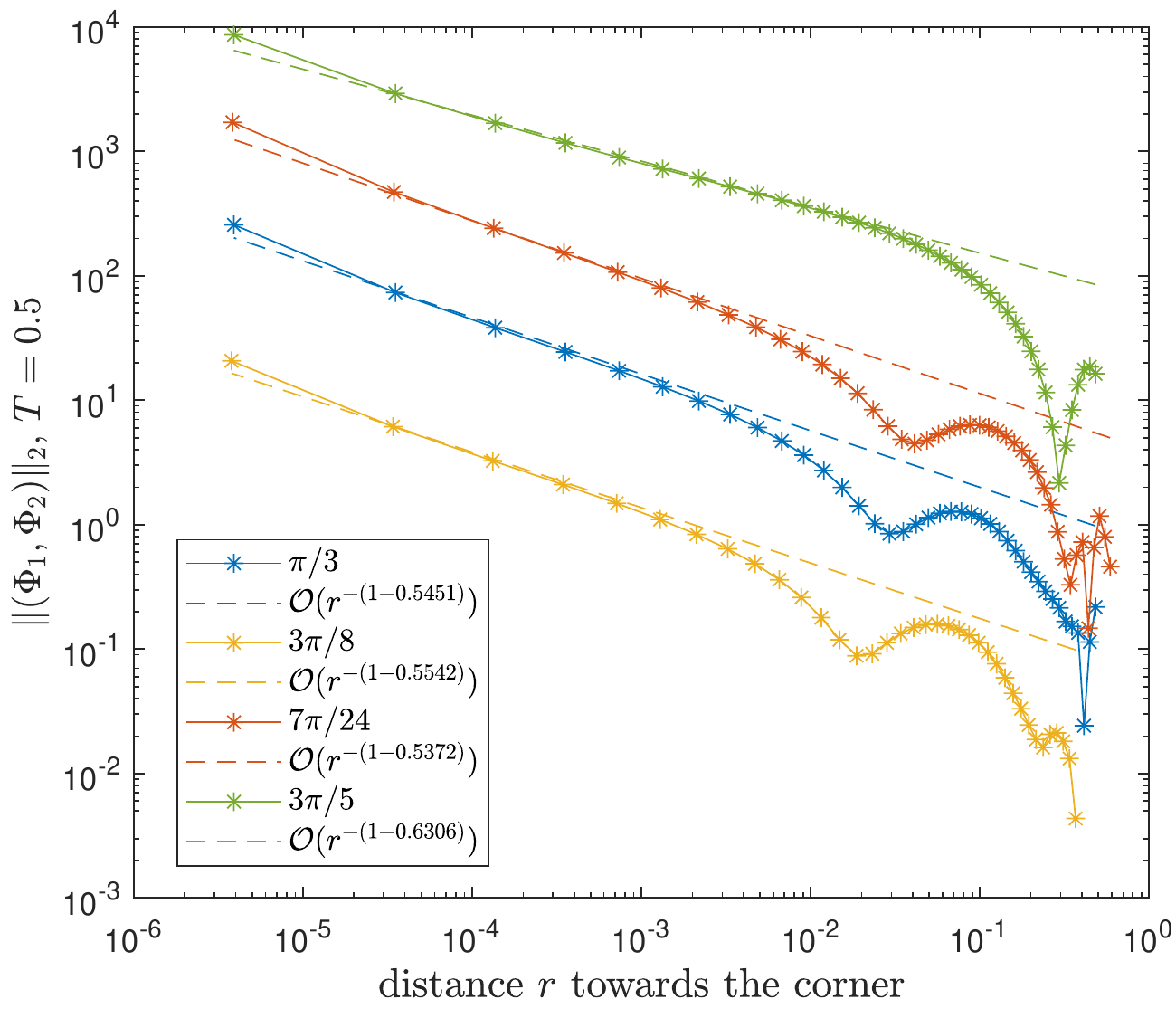}}
\qquad
\subfloat{\includegraphics[width=0.35\textwidth]{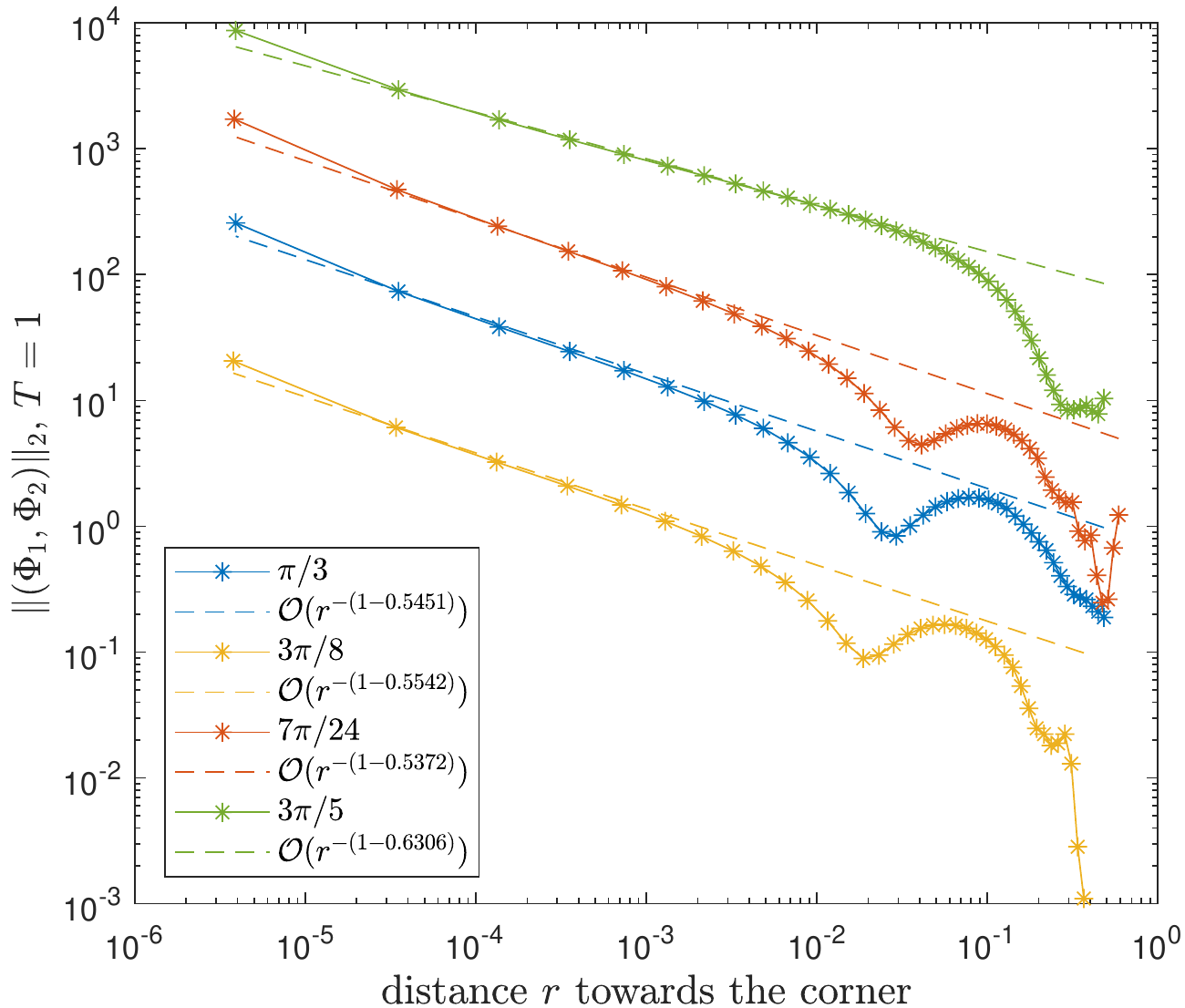}}
\caption{Asymptotic behavior towards the vertices}
\label{fig:T=05_and_T=1}
\end{figure}

In Figure \ref{fig:T=05_and_T=1}, for each geometry the Euclidean norm of $\pmb{\Phi}$ is plotted with respect to the distance $r$ towards the angle indicated in Figure \ref{fig:mesh_expected_exponent}. We observe that the solution follows the expected behavior $\mathcal{O}(r^{-(1-\nu^*)})$  for all the considered geometries. In particular, the asymptotic behavior for acute corners leads to stronger singularities ($1-\nu^*\approx 0.5$) than for the obtuse angle of the pentagon ($1-\nu^*\approx 0.37$). This confirms the theoretical discussion in Section \ref{regularity2d}. \\

We finally consider the convergence in energy on the polygonal obstacles. In particular, we examine the equilateral triangle $\Gamma_1$ and report in Table \ref{tab: 4} the value of the energy for each level of the space discretization. The energy tends to a benchmark value with increasing DOF (also in this case the number refers to one component of the vector solution), and the squared error in energy norm is shown in Figure \ref{fig:energy_eq_triangle}. The decay of the squared error in a log scale plot is linear, corresponding to $\mathcal{O}($DOF${}^{-2\nu^*{\tilde{\beta}}})$ in each experiment as in Corollary \ref{approxcor2}.\\

\begin{table}[ht!]
\caption{Energy norm squared of the approximate solution for $T=1$}
\centering
\begin{tabular}{|c|c|c|c|c|}
\hline
$\Delta t$ & DOF & $\pmb{\alpha}^\top {E}_{\mathcal{V}} \pmb{\alpha}$, ${\tilde{\beta}}=1$  & $\pmb{\alpha}^\top {E}_{\mathcal{V}} \pmb{\alpha}$, ${\tilde{\beta}}=2$ & $\pmb{\alpha}^\top {E}_{\mathcal{V}} \pmb{\alpha}$, ${\tilde{\beta}}=3$ \\ \hline\hline
$5.00\cdot 10^{-2}$ & $30$ & $5.7394\cdot 10^{-2}$ & $7.4875\cdot 10^{-2}$ & $7.6829\cdot 10^{-2}$ \\
$2.50\cdot 10^{-2}$ & $60$ & $6.8490\cdot 10^{-2}$ & $7.6828\cdot 10^{-2}$ & $7.7460\cdot 10^{-2}$ \\
$1.25\cdot 10^{-2}$ & $120$ & $7.3821\cdot 10^{-2}$ & $7.7448\cdot 10^{-2}$ & $7.7566\cdot 10^{-2}$ \\
$6.25\cdot 10^{-3}$ & $240$ & $7.5989\cdot 10^{-2}$ & $7.7558\cdot 10^{-2}$ & $7.7582\cdot 10^{-2}$ \\
\hline
\end{tabular}
\label{tab: 4}
\end{table}

\begin{figure}[h!]
\centering
\includegraphics[width=0.4\textwidth]{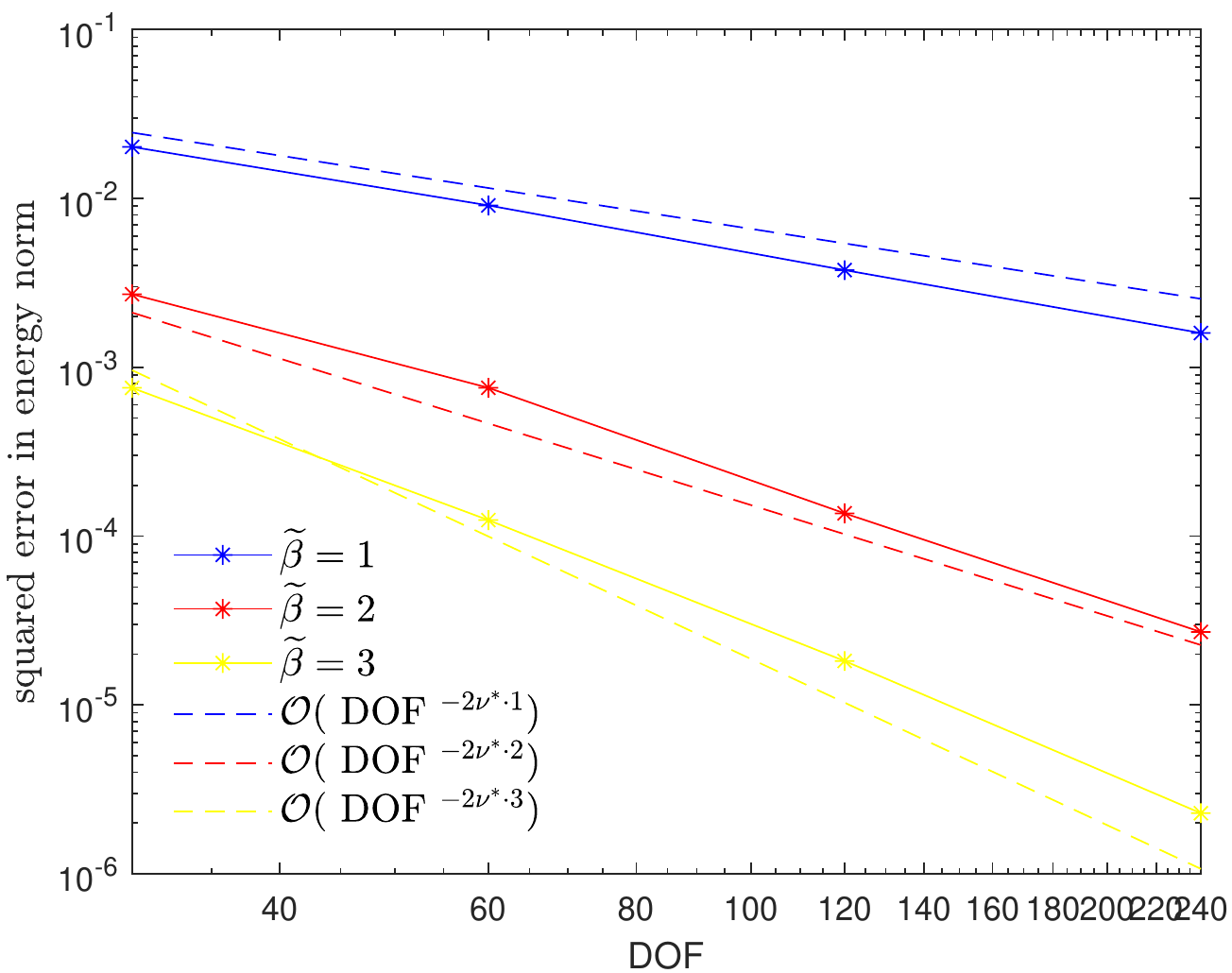}
\caption{Squared error of the energy norm with \textit{h} version on $\Gamma_1$, ${\tilde{\beta}}$-graded mesh}
\label{fig:energy_eq_triangle}
\end{figure}

\noindent \textbf{Example 4.} In this example we show numerically that the singular behavior at the corners and the decay of the energy error do not depend on the boundary data  imposed at the obstacle. We specifically consider the triangular obstacle $\Gamma_2$ in Figure \ref{fig:mesh_expected_exponent}(b). The solution $\pmb{\Phi}$ of \eqref{energetic weak formulationh} is calculated {for a right hand side} with trivial horizontal direction $\widetilde{g}_1(\textbf{x},t)=0$ and different vertical components $\widetilde{g}_2(\textbf{x},t)=H[t]f(t)$, $\widetilde{g}_2(\textbf{x},t)=H[t]f(t)x^4$ and $\widetilde{g}_2(\textbf{x},t)=100H[t]f(t)\vert x\vert^{9.5}$. In Figure \ref{fig:different_datum}(a), we consider the behavior of the Euclidean norm of $\pmb{\Phi}$ for these different boundary data, plotted as a function of the distance $r$ to the vertex which is highlighted in red (Figure \ref{fig:mesh_expected_exponent}(b), geometry $\Gamma_2$). The singular exponent is expected to be $\nu^*\simeq 0.542$ for a base angle of $3\pi/8$. Indeed, we find that, in log scale, the slope of the norm for $r \to 0$ is parallel to the dashed line corresponding to $r^{-(1-0.542)}$ for each of the tested boundary data. In Figure \ref{fig:different_datum}(b) the vertical component of $\pmb{\Phi}$ is shown on the base of $\Gamma_2$ at time $T=1$, highlighting the singular behavior at the corners. 

\begin{figure}[h!]
\centering
\subfloat[]{\includegraphics[width=0.35\textwidth]{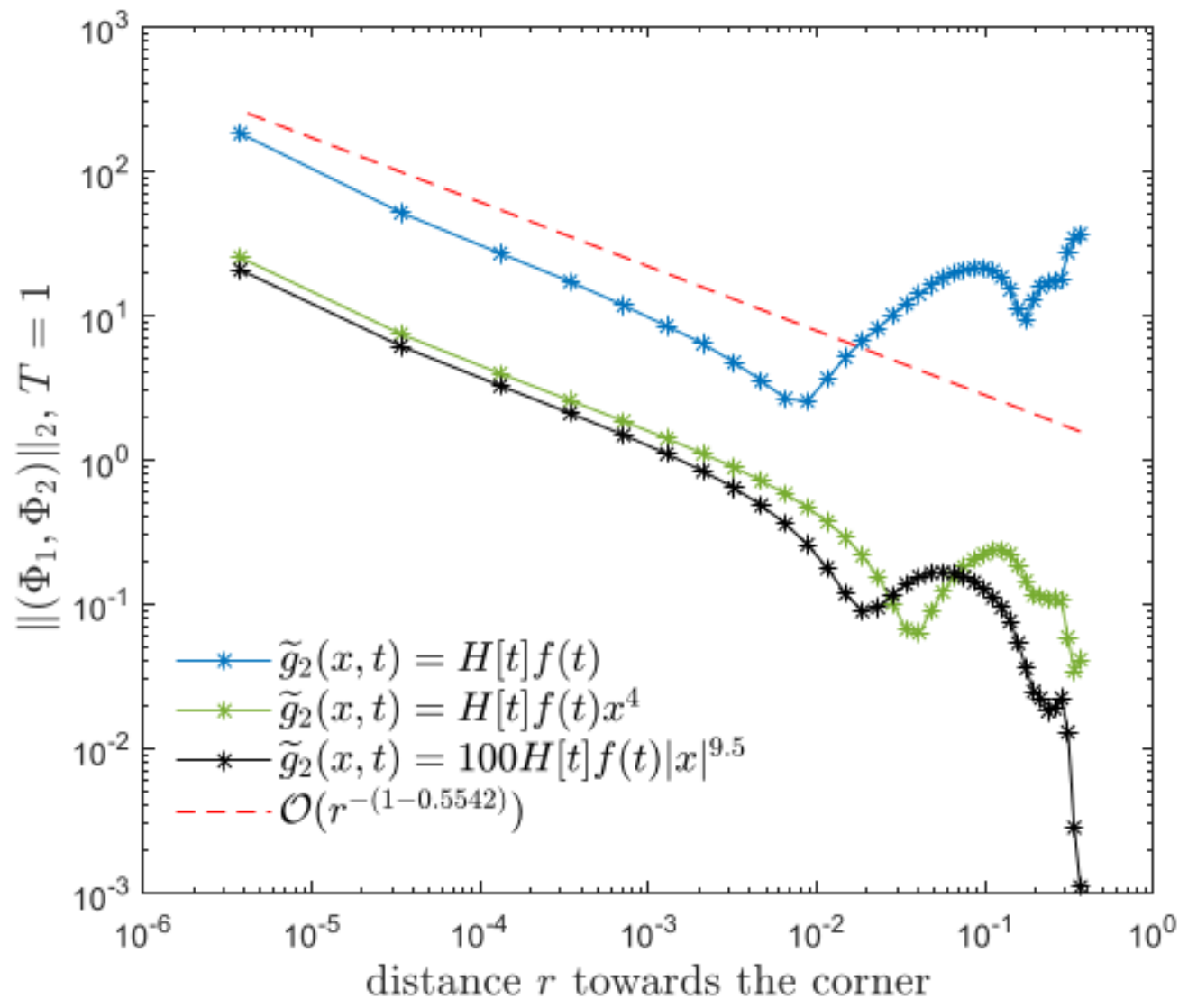}}
\quad
\subfloat[]{\includegraphics[width=0.35\textwidth]{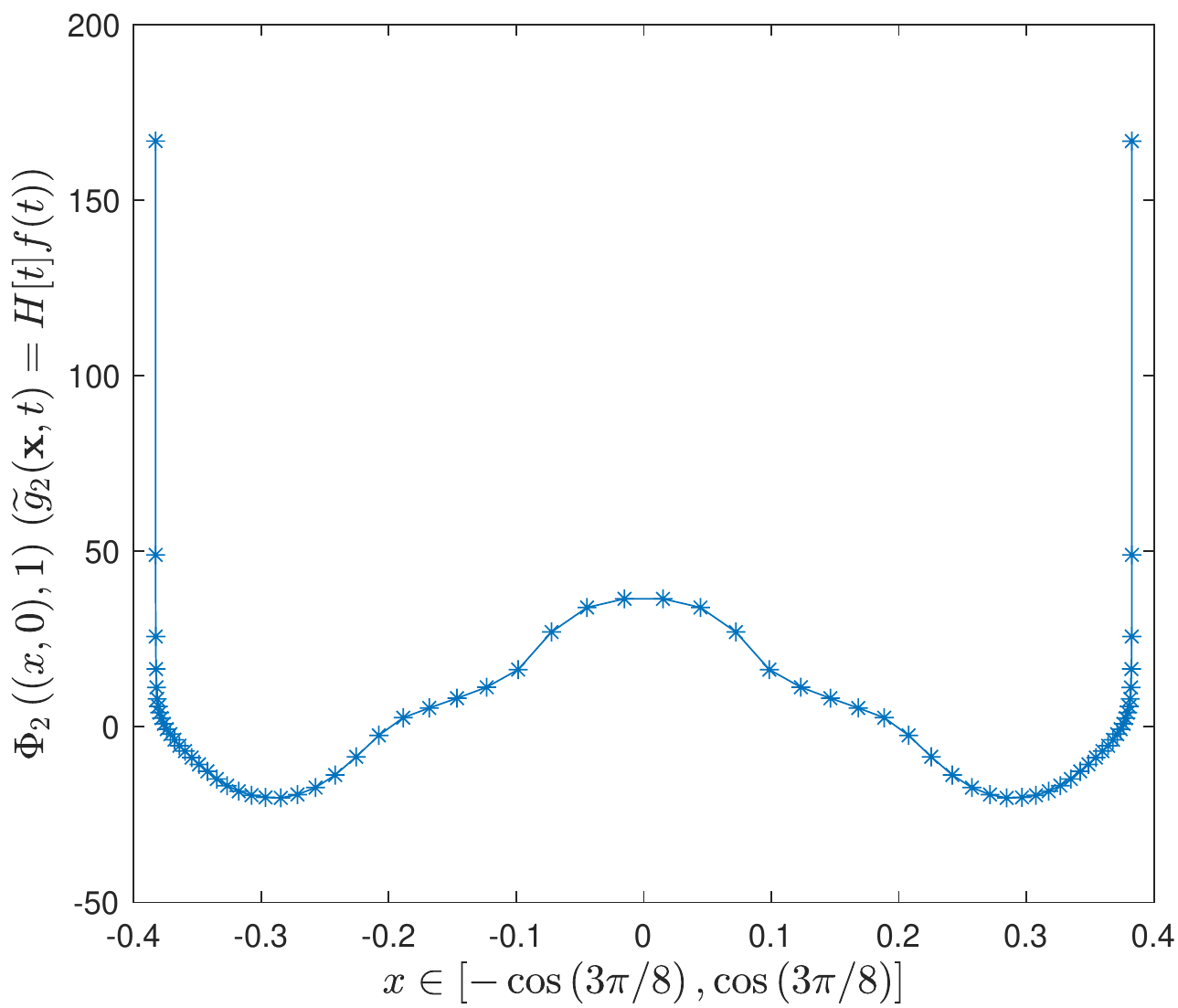}}
\caption{Asymptotic behavior towards the vertices in $\Gamma_2$ for different boundary conditions (a) and plot of the vertical component of $\pmb{\Phi}$ on the base of $\Gamma_2$ for the indicated boundary condition (b).}
\label{fig:different_datum}
\end{figure}

In Figure \ref{fig:energy_eq_triangle_2} we  consider the equilateral triangle $\Gamma_1$ of \ref{fig:mesh_expected_exponent}(b) and study the decay of the error for increasing degrees of freedom for the \textit{h} version. The number of segments and the time step are the same as in \ref{tab: 4}. {The right hand side is here given by} $\widetilde{g}_1(\textbf{x},t)=0$, $\widetilde{g}_2(\textbf{x},t)=H[t]f(t)x^4$. An algebraically ${\tilde{\beta}}$-graded mesh is used on each side, where ${\tilde{\beta}}=1,2$. The energy tends to a benchmark value as the number of degrees of freedom increases, and the squared error in energy norm in a log scale plot decays linearly as $\mathcal{O}($DOF${}^{-2\nu^*{\tilde{\beta}}})$, in agreement with Corollary \ref{approxcor2}.
\begin{figure}[h!]
\centering
\includegraphics[width=0.35\textwidth]{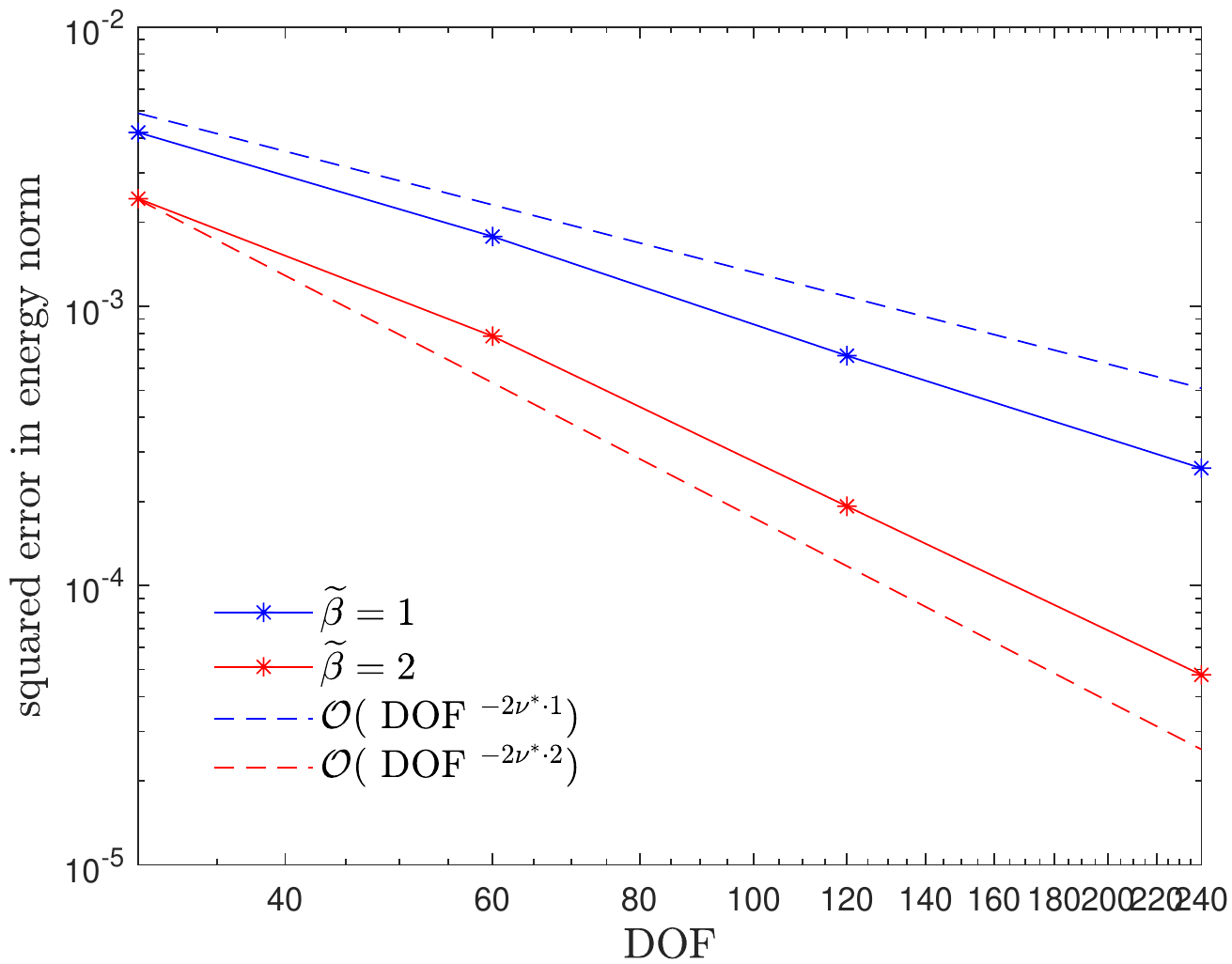}
\caption{Squared error of the energy norm with \textit{h} version on $\Gamma_1$, ${\tilde{\beta}}$-graded mesh, $\widetilde{g}_2(\textbf{x},t)=H[t]f(t)x^4$}
\label{fig:energy_eq_triangle_2}
\end{figure}

\subsection{Hard scattering problems on flat obstacle}\label{flat_ob_hypersingular}
In the following we consider the discrete hypersingular integral equation \eqref{hypersingeq} on the obstacle $\Gamma=\left\lbrace (x,0)\in\mathbb{R}\:\vert\: x\in[-0.5,0.5]\right\rbrace$ {for a time independent Neumann condition. We focus, in particular, on the solution of the discrete problem \eqref{hypersingeqh} using $h$, $p$ and $hp$ versions.}\\


\noindent \textbf{Example 5.} We consider Neumann data corresponding to $\widetilde{h}_i(\textbf{x},t)=\eta_i$, where $\eta_i\in\mathbb{R}$ is constant for $i=1,2$. The datum at the boundary is independent of time. Therefore, as time increases, the components $\Psi_i(\textbf{x},t)$ of the solution tend to the stationary functions
\begin{equation}\label{elastostatic}
\Psi_{i,\infty}(\textbf{x})=k_i\sqrt{1/4-x^2},\quad k_i=-\frac{ c_{\mathtt{P}}^2 }{\rho c_{\mathtt{S}}^2\left( c_{\mathtt{P}}^2-c_{\mathtt{S}}^2\right)}\eta_i,\quad i=1,2,
\end{equation}
representing the components of the solution for the reference \emph{elastostatic} Neumann problem with boundary datum $\textbf{h}_{\infty}(\textbf{x})=\eta_i$. We specifically set $\eta_i=1$ for $i=1,2$, so that both components of $\pmb{\Psi}$ converge to the same elastostatic function $\Psi_{1,\infty} = \Psi_{2,\infty}$. Two different sets of velocities are considered, $c_{\mathtt{S}}=1,\:c_{\mathtt{P}}=2$ and $c_{\mathtt{S}}=1,\:c_{\mathtt{P}}=3$. 

Figure \ref{fig:time_history_and_vertical_cp=2,3}(a) shows the time history of $\Psi_1$ and $\Psi_2$, calculated at the midpoint $(0,0)$ of $\Gamma$, for both sets of velocities on the time interval $[0, 7.5]$. We observe that after an initial transient phase the solution approaches the stationary value \eqref{elastostatic}. In Figure \ref{fig:time_history_and_vertical_cp=2,3}(b) the vertical component $\Psi_2$ is plotted on $\Gamma$ for speeds $c_{\mathtt{P}}=2,3$ at time $T=7.5$. This time is large enough so that for both problems the numerical solution closely matches the stationary reference solution in \eqref{elastostatic}. For the plots in Figure \ref{fig:time_history_and_vertical_cp=2,3} equation \eqref{hypersingeqh} is solved on a uniform space-time mesh with mesh size $h=0.025$ and $\Delta t=0.0125$, respectively.

\begin{figure}[h!]
\centering
\subfloat[]{\includegraphics[width=0.35\textwidth]{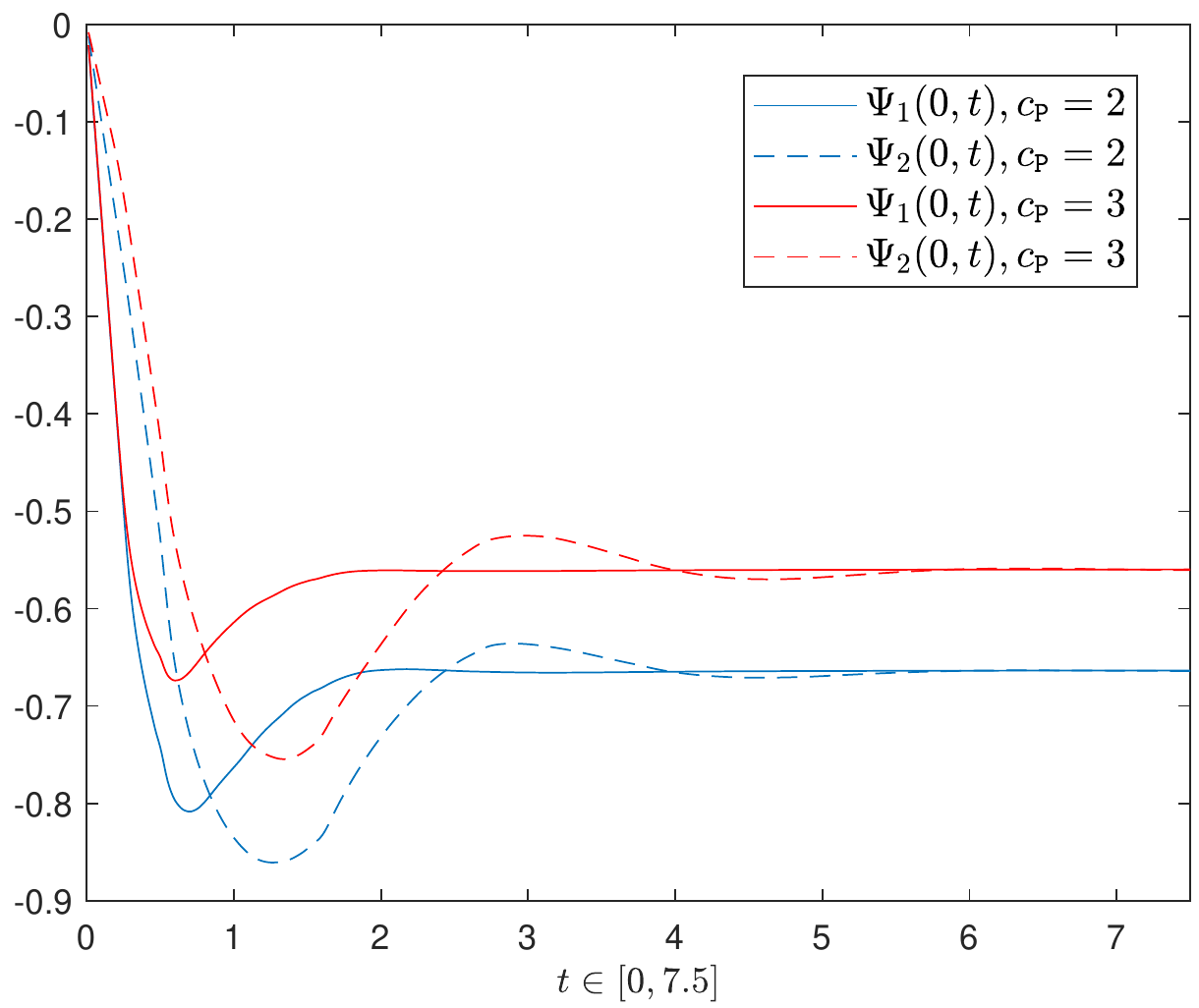}}
\quad
\subfloat[]{\includegraphics[width=0.35\textwidth]{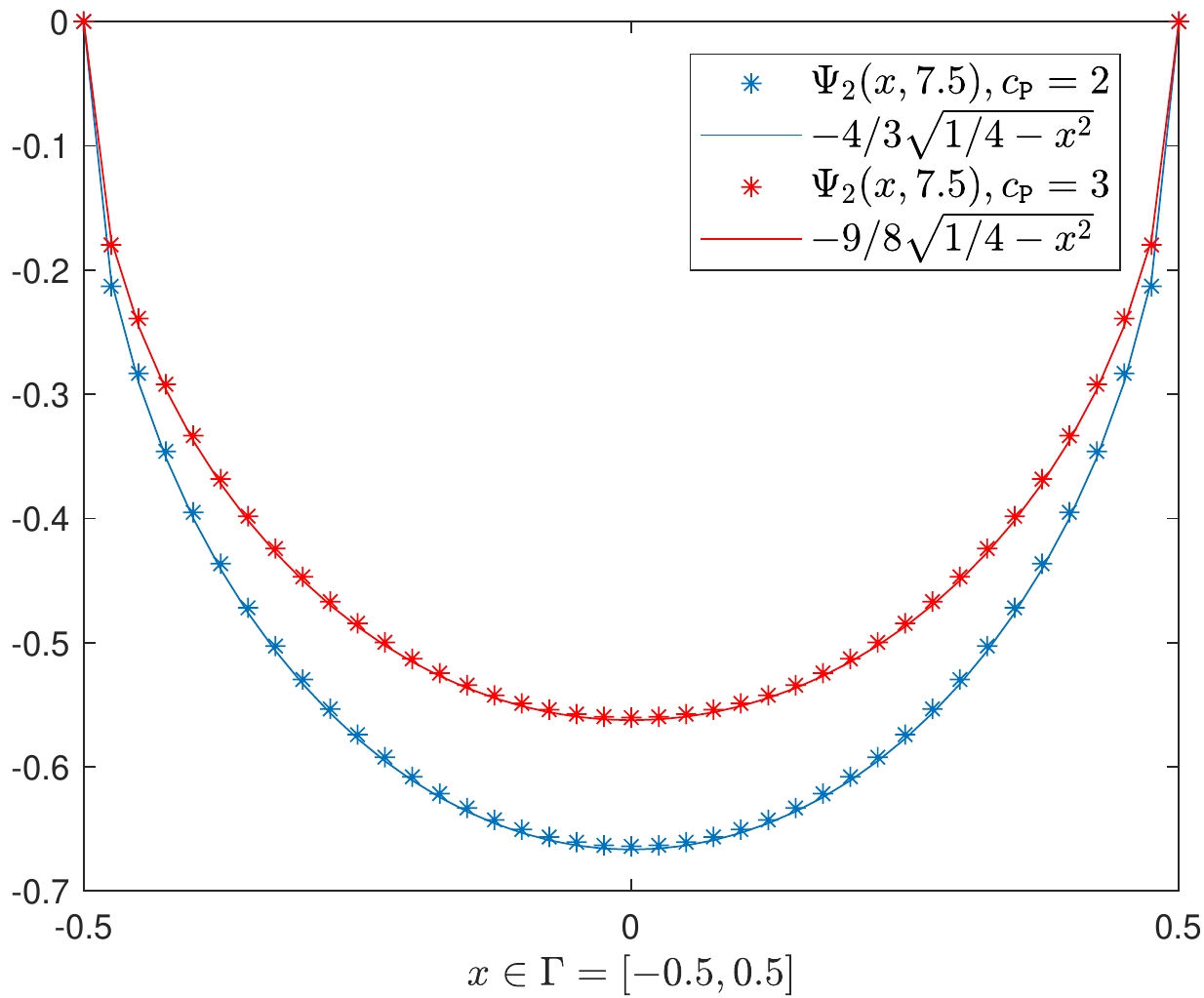}}
\caption{Time history of $\Psi_1$ and $\Psi_2$ calculated at the middle point of $\Gamma$ for the couples of velocities $c_{\mathtt{S}}=1,\:c_{\mathtt{P}}=2$ and $c_{\mathtt{S}}=1,\:c_{\mathtt{P}}=3$ (a). Vertical component $\Psi_2$ calculated at the final time instant $T=7.5$ and the related elastostatic solution $\Psi_{2,\infty}(\textbf{x},t)=k_2\sqrt{1/4-x^2}$ ($k_2=-4/3$ for $c_{\mathtt{P}}=2$ and $k_1=-9/8$ for $c_{\mathtt{P}}=3$) (b).}
\label{fig:time_history_and_vertical_cp=2,3}
\end{figure}

{To illustrate the behaviour of the solution {near} $\partial \Gamma$, Figure \ref{fig:vertical and horizontal component to the right extreme cp=2}
shows the components of $-\pmb{\Psi}$, for $c_{\mathtt{S}}=1,\:c_{\mathtt{P}}=2$, with respect to the distance $r$ towards the right {end point of the segment} $(0.5,0)^\top$ for various time instants: one observes that the {singular} behaviour is independent of time, and the {numerical solutions decrease like $r^{1/2}$} for $r$ tending to zero. The plots in Figure \ref{fig:vertical and horizontal component to the right extreme cp=2} are obtained {using} the $h$ version on a ${\tilde{\beta}}$-graded mesh with $81$ nodes, {with time step} $\Delta t=0.00625$.}
\begin{figure}[h!]
\centering
{\includegraphics[width=0.65\textwidth]{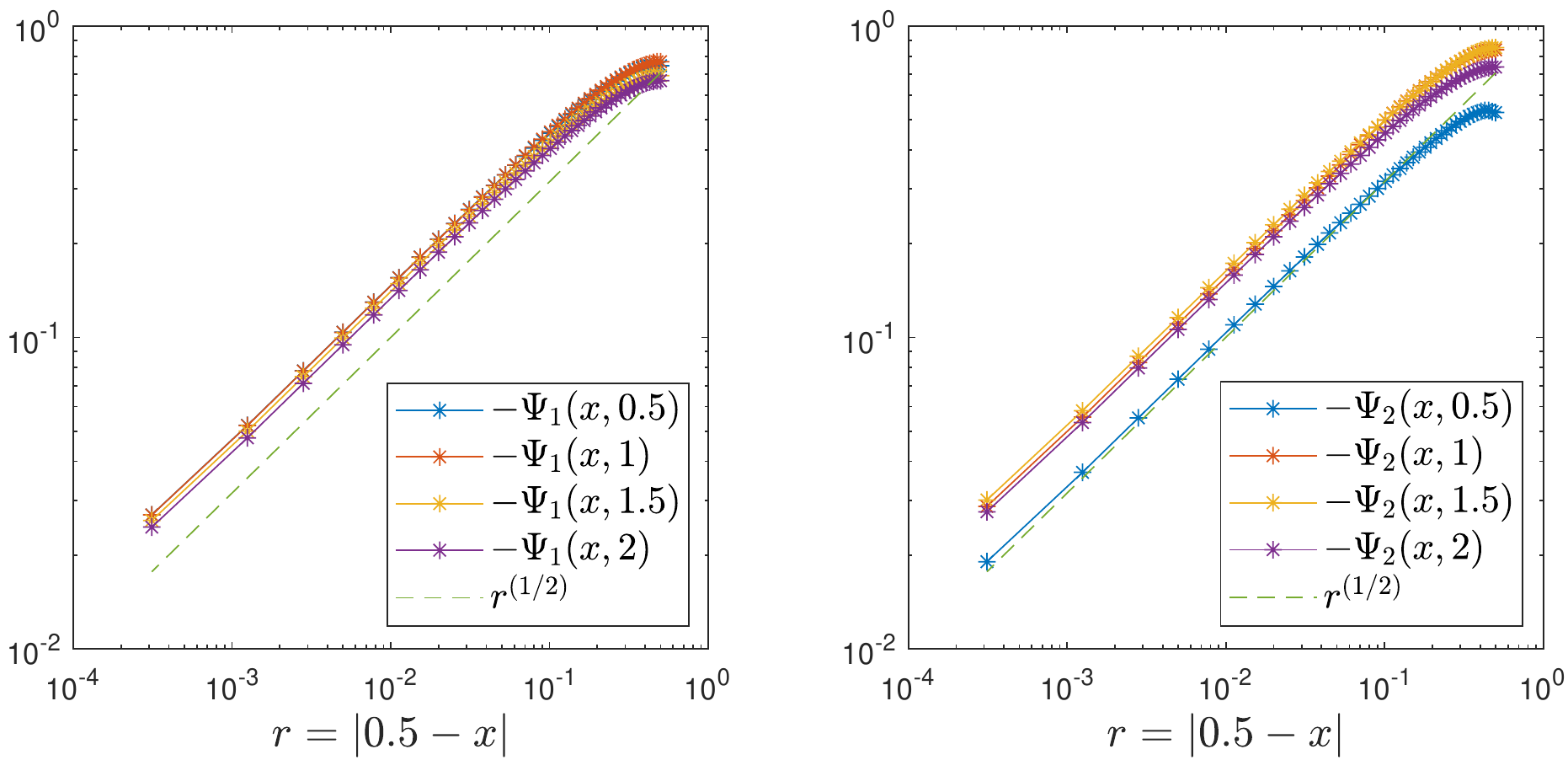}}
\caption{Asymptotic behaviour of $-\Psi_1$ and $-\Psi_2$ towards the right end of $\Gamma$ for various time instants ($c_{\mathtt{S}}=1,\:c_{\mathtt{P}}=2$). The arc is discretized by a ${\tilde{\beta}}$-graded mesh with ${\tilde{\beta}}=2$ and $81$ mesh points.}
\label{fig:vertical and horizontal component to the right extreme cp=2}
\end{figure}

For the case $c_{\mathtt{S}}=1,\:c_{\mathtt{P}}=2$, we study the decay of the error in energy norm for the approximate solution of \eqref{hypersingeqh} up to time $T=2$ analysing the value $\pmb{\beta}^\top {E}_{\mathcal{W}} \pmb{\beta}$, namely the squared energetic norm of the approximate solution, which increases towards a common benchmark value for all tested discretization methods. We refer the reader to section \ref{sec:algo} for construction details of $\pmb{\beta}$ and ${E}_{\mathcal{W}}$. The number of spatial DOF in the following, as previously, corresponds to one component of the vector solution. For the $h$ version we choose a ${\tilde{\beta}}$-graded mesh on $\Gamma$ with ${\tilde{\beta}}=1,2$ and  $10,20,40$, respectively $80$ segments. The time step $\Delta t=0.05$ in the case of $10$ segments is halved at each refinement of the spatial mesh. The log scale plot in Figure \ref{fig:energy_and_condition_number_hypersingular} shows a linear decay of the error for the $h$ version, parallel to the lines $\mathcal{O}($DOF${}^{-{\tilde{\beta}}})$. The results confirm the prediction in Corollary \ref{approxcor2}. For the $p$ version we consider a uniform discretization of the obstacle with $h=0.1$ and  a uniform time step $\Delta t = 1/(2\cdot$DOF$)$. The log scale plot shows a linear decay of the error parallel to the expected line $\mathcal{O}($DOF${}^{-2})$.  The $hp$ version with a geometrically graded mesh is considered for meshes  on $\Gamma$ with $4,6,8,10$ and $12$ segments. At each refinement of the mesh the degree $p$, starting from $2$, increases uniformly on all the space elements. The time step is chosen as $\Delta t=0.125$ for $4$ segments and halved at each iteration. Similarly to the soft scattering problems presented above the $hp$ method shows the fastest decay of the error with respect to increasing spatial DOF.
\begin{figure}[h!]
\centering
\includegraphics[width=0.45\textwidth]{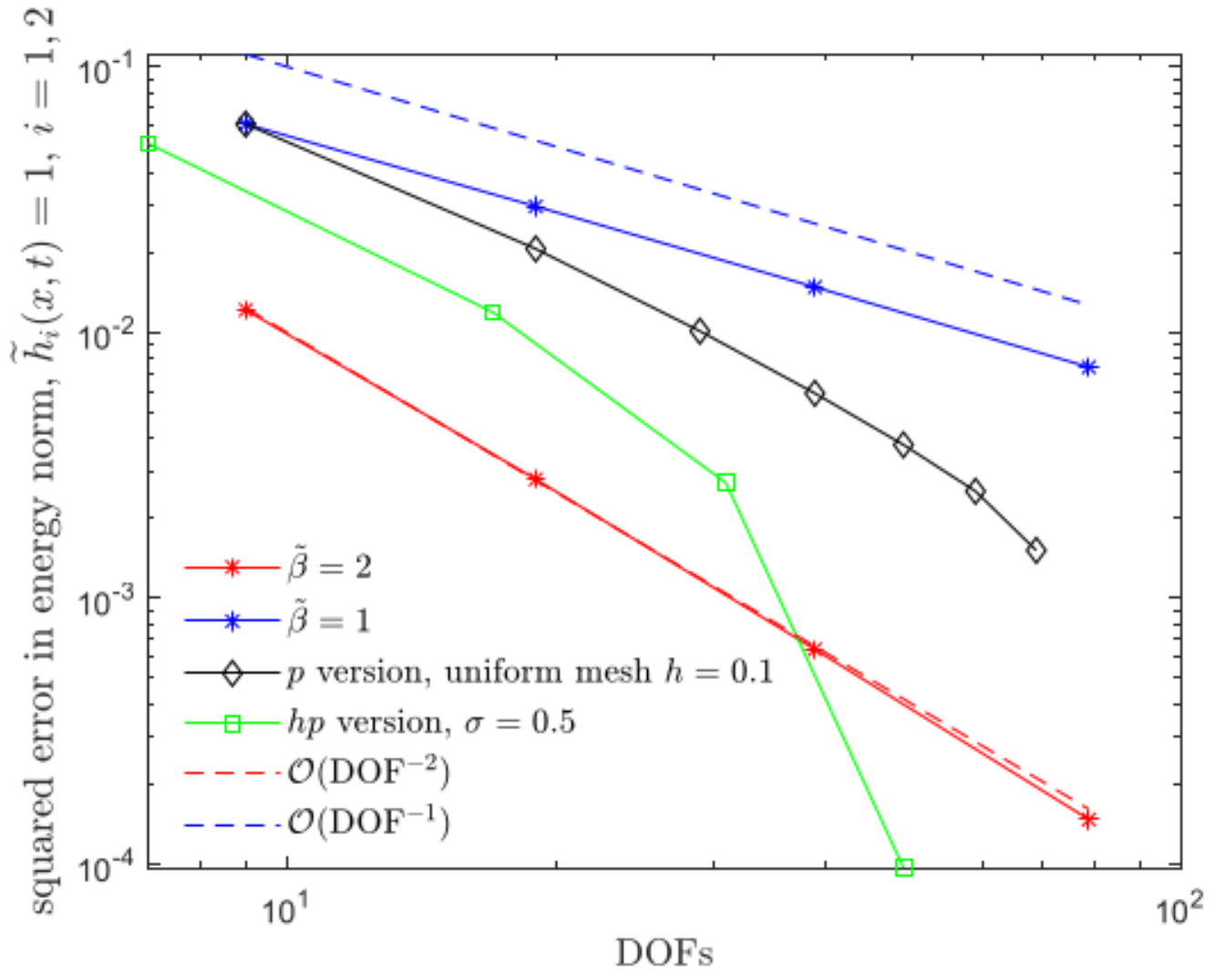}
\caption{Squared error of the energy norm calculated up to time instant $T=2$}
\label{fig:energy_and_condition_number_hypersingular}
\end{figure}

\section{Conclusions}

{In this work we initiate the study of higher-order versions of the boundary element method for linear elastodynamics, including $h$, $p$ and $hp$ versions. The asymptotic expansions for the solution obtained near geometric singularities of the domain give rise to efficient discretizations, with the same approximation rates as known for $h$, $p$ and $hp$ approximations of time independent
problems.\\
The quasi-optimal $hp$ explicit estimates in this article complement the recent analysis for the wave equation, for both finite and boundary element methods \cite{graded,hp,mueller}, and for linear elastodynamics in 2d \cite{mueller2}.
The convergence is determined by the singular behavior of the solution near the non-smooth boundary points of
the domain. Our analysis relies on the classical approximation results for time independent problems \cite{dauge}, in combination with the analysis of the leading singular terms in the time dependent problem \cite{matyu}. \\
Extensive numerical experiments for a slit and polygonal domains in 2d illustrate the quasi-optimal convergence rates and confirm the expected leading asymptotic behavior of the solution near a vertex.
On a slit the energy error $O(p^{-1})$ of the $p$ version converges with the same rate as an $h$ version on a $2$-graded mesh. For closed
polygonal domains the solution is less singular near the vertices, depending on the material parameters and the opening angle. Accordingly, higher convergence rates are obtained in both the analysis and in the numerical experiments. }

\begin{appendices}
\appendix

\section{}

In this appendix we introduce space--time anisotropic Sobolev spaces on the boundary $\Gamma$ as a convenient functional analytic setting for the analysis of the time dependent boundary integral operators. A detailed exposition may be found in \cite{hd, setup}. {Furthermore, we collect mapping properties of the integral operators $\mathcal{V}, \mathcal{W}$ in these space-time anisotropic spaces (Theorem \ref{mapthm}) and show continuity and coercivity of the associated bilinear forms (Proposition \ref{DPbounds}). The latter imply the stability of the Galerkin schemes in Section 4.} In the case of an open screen or line segment, $\partial\Gamma\neq \emptyset$, we first extend $\Gamma$ to a closed, orientable Lipschitz manifold $\widetilde{\Gamma}$. 
On $\Gamma$ we recall the usual Sobolev spaces of supported distributions:
$$\widetilde{H}^s(\Gamma) = \{u\in H^s(\widetilde{\Gamma}): \mathrm{supp}\ u \subset {\overline{\Gamma}}\}\ , \quad\ s \in \mathbb{R}\ .$$
The Sobolev space ${H}^s(\Gamma)$ is the quotient space $ H^s(\widetilde{\Gamma}) / \widetilde{H}^s({\widetilde{\Gamma}\setminus\overline{\Gamma}})$.
To define a family of Sobolev norms,  $\alpha_i$ be a partition of unity subordinate to a covering of $\widetilde{\Gamma}$ by open sets $B_i$. Given diffeomorphisms $\varphi_i$ from $B_i$ to the unit cube in $\mathbb{R}^n$, Sobolev norms are induced from $\mathbb{R}^n$, with parameter $\omega \in \mathbb{C}\setminus \{0\}$:
\begin{equation*}
 ||u||_{s,\omega,{\widetilde{\Gamma}}}=\left( \sum_{i=1}^p \int_{\mathbb{R}^n} (|\omega|^2+|\pmb{\xi}|^2)^s|\mathcal{F}\left\{(\alpha_i u)\circ \varphi_i^{-1}\right\}(\pmb{\xi})|^2 d\pmb{\xi} \right)^{\frac{1}{2}}\ .
\end{equation*}
Here, $\mathcal{F}=\mathcal{F}_{\bold{x} \mapsto \pmb{\xi}}$ denotes the Fourier transform $\mathcal{F}\varphi(\pmb{\xi}) = \int e^{-i\bold{x}\cdot\pmb{\xi}} \varphi(\bold{x})\ d\bold{x}$. Different $\omega \in \mathbb{C}\setminus \{0\}$ induce equivalent norms on $H^s(\Gamma)$, $\|u\|_{s,\omega,\Gamma} = \inf_{v \in \widetilde{H}^s(\widetilde{\Gamma}\setminus\overline{\Gamma})} \ \|u+v\|_{s,\omega,\widetilde{\Gamma}}$ and on $\widetilde{H}^s(\Gamma)$, $\|u\|_{s,\omega,\Gamma, \ast } = \|e_+ u\|_{s,\omega,\widetilde{\Gamma}}$. $e_+$ extends the distribution $u$ by $0$ from $\Gamma$ to $\widetilde{\Gamma}$.  When a specific $\omega$ is fixed, we write $H^s_\omega(\Gamma)$ for $H^s(\Gamma)$, respectively $\widetilde{H}^s_\omega(\Gamma)$ for $\widetilde{H}^s(\Gamma)$. 
The norm $\|u\|_{s,\omega,\Gamma, \ast }$ is stronger than $\|u\|_{s,\omega,\Gamma}$. 

We may now define a family of space-time anisotropic Sobolev spaces:
\begin{definition}\label{sobdef}
For {$\sigma>0$ and} $r,s \in\mathbb{R}$ define
\begin{align}
 H^r_\sigma(\mathbb{R}^+,{H}^s(\Gamma))&=\{ u \in \mathcal{D}^{'}_{+}(H^s(\Gamma)): e^{-\sigma t} u \in \mathcal{S}^{'}_{+}(H^s(\Gamma))  \textrm{ and }   ||u||_{r,s,\Gamma} < \infty \}\ , \nonumber \\
 H^r_\sigma(\mathbb{R}^+,\widetilde{H}^s({\Gamma}))&=\{ u \in \mathcal{D}^{'}_{+}(\widetilde{H}^s({\Gamma})): e^{-\sigma t} u \in \mathcal{S}^{'}_{+}(\widetilde{H}^s({\Gamma}))  \textrm{ and }   ||u||_{r,s,\Gamma, \ast} < \infty \}\ .\label{sobdef}
\end{align}
Here, $\mathcal{D}^{'}_{+}(E)$ denotes the space of all distributions on $\mathbb{R}$ with support in $[0,\infty)$, taking values in a Hilbert space $E= {H}^s({\Gamma})$, respectively  $E=\widetilde{H}^s({\Gamma})$. $\mathcal{S}^{'}_{+}(E)\subset \mathcal{D}^{'}_{+}(E)$ denotes the subspace of tempered distributions. The Sobolev spaces are endowed with the norms
\begin{align}
\|u\|_{r,s}:=\|u\|_{r,s,\Gamma}&=\left(\int_{-\infty+i\sigma}^{+\infty+i\sigma}|\omega|^{2r}\ \|\hat{u}(\omega)\|^2_{s,\omega,\Gamma}\ d\omega \right)^{\frac{1}{2}}\ ,\nonumber \\
\|u\|_{r,s,\ast} := \|u\|_{r,s,\Gamma,\ast}&=\left(\int_{-\infty+i\sigma}^{+\infty+i\sigma}|\omega|^{2r}\ \|\hat{u}(\omega)\|^2_{s,\omega,\Gamma,\ast}\ d\omega \right)^{\frac{1}{2}}\,. \label{sobnormdef}  
\end{align}
\end{definition}
They are Hilbert spaces. For $r=s=0$ they correspond to the weighted $L^2$-space with scalar product $\int_0^\infty e^{-2\sigma t} \int_\Gamma u \overline{v} d\Gamma_{\bold{x}}\ dt$. Because $\Gamma$ is Lipschitz, these spaces are independent of the choice of $\alpha_i$ and $\varphi_i$ when $|s|\leq 1$, as for standard Sobolev spaces.\\ We shall also use the norms $\|u\|_{r, s, (t_1, t_2]\times \Gamma}$ and $\|u\|_{r, s, (t_1, t_2]\times \Gamma,\ast}$ for restrictions on the time interval $(t_1,t_2]$.\\

Let now $\widetilde{\Gamma} = \partial \Omega'$ the boundary of a Lipschitz subset $\Omega \subset \R^n$ and  $\Gamma \subset \widetilde{\Gamma}$ open. Denote $\Omega = \R^n \setminus \overline{\Omega'}$.

We review the mapping properties for the weakly singular integral operator $\mathcal{V}$ and the hypersingular operator $\mathcal{W}$. 
\begin{theorem}\label{mapthm}
The single layer potential operator and the hypersingular operator are continuous for {$\sigma>0$ and $r \in \mathbb{R}$:}
\begin{align*}
\mathcal{V}&: H_\sigma^{r+1}(\mathbb{R}^+, \widetilde{H}^{-\frac{1}{2}}(\Gamma)) \to H_\sigma^{r}(\mathbb{R}^+, H^{\frac{1}{2}}(\Gamma)) \ , \ ,\\ \mathcal{K}' &: H_\sigma^{r+1}(\mathbb{R}^+, \widetilde{H}^{-\frac{1}{2}}(\Gamma)) \to H_\sigma^{r}(\mathbb{R}^+, H^{-\frac{1}{2}}(\Gamma))\ ,\\
\mathcal{K}&: H_\sigma^{r+1}(\mathbb{R}^+, \widetilde{H}^{\frac{1}{2}}(\Gamma)) \to H_\sigma^{r}(\mathbb{R}^+, H^{\frac{1}{2}}(\Gamma)) \ ,
\\ \mathcal{W} &: H_\sigma^{r+1}(\mathbb{R}^+, \widetilde{H}^{\frac{1}{2}}(\Gamma)) \to H_\sigma^{r}(\mathbb{R}^+, H^{-\frac{1}{2}}(\Gamma))\ .
\end{align*}
\end{theorem}
{This may be found in Theorem 3.1 in \cite{chud}, see also \cite{Becache1994}  for $\mathcal{W}$ in 2d, with an analogous proof. See also \cite{jr} for a recent discussion of mapping properties for the wave equation.} 

For  convenience of the reader we recall basic properties of the bilinear form 
for the Dirichlet problem {in the infinite space-time cylinder $\Gamma \times \mathbb{R}^+$}, 
\begin{align} \label{eq:bilinearform}
B_{D,\Gamma \times \mathbb{R}^+}(\pmb{ \Phi},\pmb{\tilde{\Phi}}) := \int_{\mathbb{R}^+}\int_\Gamma \mathcal{V} \partial_t \pmb{\Phi}(t,\bold{x})\ \pmb{\tilde{\Phi}}(t,\bold{x})\ d\Gamma_{\bold{x}} \, d_\sigma t \ ,
\end{align}
{where $d_\sigma t = e^{-2\sigma t} dt$, as well as the corresponding bilinear form for the Neumann problem,
\begin{align} \label{eq:bilinearformN}
B_{N,\Gamma \times \mathbb{R}^+}(\pmb{ \Psi},\pmb{\tilde{\Psi}}) := \int_{\mathbb{R}^+}\int_\Gamma \mathcal{W} \partial_t \pmb{\Psi}(t,\bold{x})\ \pmb{\tilde{\Psi}}(t,\bold{x})\ d\Gamma_{\bold{x}} \, d_\sigma t \ ,
\end{align}
}
\begin{proposition}\label{DPbounds} {Let $\sigma>0$.} \\
a) For every $\pmb{ \Phi},\pmb{\tilde{\Phi}} \in H^1_\sigma( \mathbb{R}^+, H^{-\frac{1}{2}}(\Gamma))^n$ there holds:
\begin{equation}\label{BDcontA}|B_{D,\Gamma \times \mathbb{R}^+}(\pmb{ \Phi},\pmb{\tilde{\Phi}})| \lesssim \|\pmb{ \Phi}\|_{1,-\frac{1}{2},\Gamma, \ast} \|\pmb{\tilde{\Phi}}\|_{1,-\frac{1}{2}, \Gamma,\ast}\end{equation}
and
\begin{equation}\label{BDcoerA}\|\pmb{ \Phi}\|_{0,-\frac{1}{2},\Gamma,\ast}^2 \lesssim B_{D,\Gamma \times \mathbb{R}^+}(\pmb{ \Phi},\pmb{ \Phi}) . \end{equation}
{b) For every $\pmb{ \Psi},\pmb{\tilde{\Psi}} \in H^1_\sigma( \mathbb{R}^+, H^{\frac{1}{2}}(\Gamma))^n$ there holds:
\begin{equation}\label{BNcontA}|B_{N,\Gamma \times \mathbb{R}^+}(\pmb{ \Psi},\pmb{\tilde{\Psi}})| \lesssim \|\pmb{ \Psi}\|_{1,-\frac{1}{2},\Gamma, \ast} \|\pmb{\tilde{\Psi}}\|_{1,\frac{1}{2}, \Gamma,\ast}\end{equation}
and
\begin{equation}\label{BNcoerA}\|\pmb{ \Psi}\|_{0,\frac{1}{2},\Gamma,\ast}^2 \lesssim B_{N,\Gamma \times \mathbb{R}^+}(\pmb{ \Psi},\pmb{ \Psi}) . \end{equation}}
\end{proposition}
\begin{proof}
{The inequalities \eqref{BDcontA} and \eqref{BNcontA} follow from the mapping properties in Theorem \ref{mapthm}.  The coercivity \eqref{BNcoerA} was shown in \cite{Becache1993, Becache1994} in 2d, and the proof holds verbatim in any dimension.}\\

{To show \eqref{BDcoerA}, we consider the elastic problem in the frequency domain:
\begin{equation}
\left\lbrace
\begin{array}{l}
(\lambda+\mu)\nabla(\nabla \cdot \textbf{u})+\mu\Delta\textbf{u}+\rho\omega^2\textbf{u}=\textrm{div}\:\sigma(\textbf{u})+\rho\omega^2\textbf{u}= 0, \quad \textbf{x}\in\Omega'\cup\Omega \\
\textbf{u}=\textbf{g}, \quad \textbf{x}\in \widetilde{\Gamma}.
\end{array}
\right..
\end{equation}
We assume $\mathrm{Im}\ (\omega) \geq \sigma>0$. The energetic weak formulation for the single layer equation for the traction $[\textbf{p}]=[\sigma(\textbf{u})\textbf{n}]$ in frequency domain is given by (using Parseval's identity):\\
\textit{
Find} $[\textbf{p}]\in H^{-\frac{1}{2}}_\omega(\widetilde{\Gamma})^n$ \textit{such that}
\begin{equation}\label{bilinear form trasform}
B_{D,\omega}([\textbf{p}],\overline{\pmb{q}}) = \langle-i\omega \mathcal{V}_{\omega} [\textbf{p}],\overline{\pmb{q}} \rangle_{\widetilde{\Gamma}}= \langle-i\omega \textbf{g},\overline{\pmb{q}}\rangle_{\widetilde{\Gamma}}
\end{equation}
\textit{
for all} $\pmb{q}\in H^{-\frac{1}{2}}_\omega(\widetilde{\Gamma})^n$.}\\

{It involves the single layer operator $\mathcal{V}_{\omega}$ obtained from $\mathcal{V}$ by Fourier transformation. Using Green's formula as in \cite{Becache1994}, Thm 3.1, we have
$$
\int_{\Omega'\cup\Omega} \left(\overline{\sigma (\textbf{u})} : \varepsilon(\textbf{u})-\rho \omega^2\vert\textbf{u}\vert^2\right)d \textbf{x}=\int_{\widetilde{\Gamma}} \textbf{u} \cdot \overline{[\sigma(\textbf{u})\textbf{n}]} d\widetilde{\Gamma}\equiv \langle\mathcal{V}_{\omega} [\textbf{p}],\overline{[\textbf{p}]} \rangle_{\widetilde{\Gamma}}.
$$
Now note that $|\langle-i\omega \mathcal{V}_{\omega} [\textbf{p}],\overline{[\textbf{p}]} \rangle_{\widetilde{\Gamma}}\vert\geqslant  \mathrm{Re} \:  i\overline{\omega} \langle\mathcal{V}_{\omega} [\textbf{p}],\overline{[\textbf{p}]} \rangle_{\widetilde{\Gamma}}
$
and
\begin{align}
\mathrm{Re}\:  i\overline{\omega} \langle\mathcal{V}_{\omega} [\textbf{p}],\overline{[\textbf{p}]} \rangle_{\widetilde{\Gamma}}  &= \mathrm{Re} \left( i\overline{\omega} \int_{\Omega'\cup\Omega} \overline{\sigma (\textbf{u})}: \varepsilon(\textbf{u})d \textbf{x}\right)+ \mathrm{Re} \left(-i\omega \int_{\Omega'\cup\Omega} \rho\vert\omega\vert^2\vert\textbf{u}\vert^2 d \textbf{x}\right)\nonumber\\
&=2 \mathrm{Im}  (\omega) E_{\omega}\geqslant 0,\label{passage}
\end{align}
with
$$E_{\omega}=\frac{1}{2}\int_{\Omega'\cup\Omega} \left(\overline{\sigma (\textbf{u})}: \varepsilon(\textbf{u})+ \rho\vert\omega\vert^2\vert\textbf{u}\vert^2 \right)d \textbf{x}\ .$$
Physically, $E_{\omega}$ is the energy of the displacement $\textbf{u}$, and it satisfies
\begin{equation}\label{energy 2}
E_{\omega}\geqslant C_{\sigma}\Vert \textbf{u} \Vert^2_{1,\omega,\Omega'\cup\Omega}
\end{equation}
for a positive constant $C_{\sigma}$. From \eqref{passage} and \eqref{energy 2} we deduce that
$$\vert  \langle -i\omega \mathcal{V}_{\omega} [\textbf{p}],\overline{[\textbf{p}]} \rangle_{\widetilde{\Gamma}}\vert\geqslant\widetilde{C}_{\sigma}\Vert \textbf{u} \Vert^2_{1,\omega,\Omega'\cup\Omega}\ .$$
From the trace theorem there exists a positive constant $C_{trace}$ such that
$$2C_{trace}\Vert \textbf{u} \Vert^2_{1,\omega,\Omega'\cup\Omega}\geqslant 2 \Vert \textbf{p}|_{\widetilde{\Gamma}_+}\Vert_{-1/2,\omega,\widetilde{\Gamma}}^2 + 2\Vert \textbf{p}|_{\widetilde{\Gamma}_-}\Vert_{-1/2,\omega,\widetilde{\Gamma}}^2 \geqslant  \Vert[\textbf{p}] \Vert_{-1/2,\omega,\widetilde{\Gamma}}^2.
$$
Coercivity in the frequency domain follows:
\begin{equation}\label{coerciveness frequency domain}
\vert  \langle -i\omega \mathcal{V}_{\omega} [\textbf{p}],\overline{[\textbf{p}]} \rangle_{\widetilde{\Gamma}}\vert\geqslant \frac{\widetilde{C}_{\sigma}}{2C_{trace}} \Vert[\textbf{p}] \Vert_{-1/2,\omega,\widetilde{\Gamma}}^2.
\end{equation}}
{To show \eqref{BDcoerA}, it remains to translate the coercivity \eqref{coerciveness frequency domain} from the frequency  domain into the time domain.
Integrating \eqref{bilinear form trasform} in  $\omega$ and using the Parseval identity, noting $\mathcal{F}_{\omega\rightarrow t}^{-1}\left(\widehat{\varphi}(\omega+i\sigma)\right)=\varphi(t)e^{-\sigma t}$, we get the identity
$$\int_{\mathbb{R}+i\omega_I^0}\int_{\widetilde{\Gamma}}-i\omega \mathcal{V}_{\omega} \widehat{\pmb{\Phi}}\cdot\overline{\widehat{\pmb{\Phi}}}d\widetilde{\Gamma}d\omega=\int_0^{+\infty}\int_{\widetilde{\Gamma}}e^{-2\sigma t}\frac{\partial}{\partial t}\left( \mathcal{V}\pmb{\Phi}\right)\cdot\pmb{\Phi} d\widetilde{\Gamma}dt=B_D\left(\pmb{\Phi},\pmb{\Phi} \right).$$
We now use \eqref{coerciveness frequency domain}:
$$\mathrm{Re}\: B_D\left(\pmb{\Phi},\pmb{\Phi} \right)  =    \int_{\mathbb{R}+i\sigma}\mathrm{Re}\: i\overline{\omega} \langle\mathcal{V}_{\omega} \pmb{\widehat{\Phi}},\overline{\pmb{\widehat{\Phi}}} \rangle_{\widetilde{\Gamma}} \geqslant \frac{\widetilde{C}_{\sigma}}{2C_{trace}} \int_{\mathbb{R}+i \sigma}  \Vert\pmb{\widehat{\Phi}} \Vert_{-1/2,\omega,\widetilde{\Gamma}}^2 d\omega\ .$$
Therefore $$|B_D\left(\pmb{\Phi},\pmb{\Phi} \right)| \geq \frac{\widetilde{C}_{\omega_I}}{2C_{trace}} \Vert \pmb{\Phi}\Vert_{0,-1/2,\widetilde{\Gamma}}^2\ .$$
Proposition \ref{DPbounds} follows by restricting to distributions supported in $\Gamma \subset \widetilde{\Gamma}$.}
\end{proof}

\section{}\label{sec:polygonallame}


In the following, let us describe the approach by Matyukevich and Plamenevski\v{\i}  from \cite{matyu} to prove the asymptotic expansion of the solution to the elastodynamic {Dirichlet problem  \eqref{components with hooke tensor} - \eqref{dirichlet condition} in a neighborhood of} a {non-smooth boundary} point of the domain. 
For ease of reference to the work of Plamenevski\v{\i} and coauthors, as well as \cite{hp}, this section adopts some of the notation from the analysis community, rather than the notation commonly found in numerical works. In particular, the  $\sigma>0$ from other sections in the article is here called $\gamma$, singular exponents {$\lambda_\ell$} are denoted by $i{\lambda_\ell}$, and the definition of the Fourier transform and its inverse are interchanged. {However, note that the dimensions $n$ and $m$ are interchanged compared to the specific reference \cite{matyu}, but they agree with the main body of this paper.}

%

In the following we consider two model geometries, wedge and corner, to describe the local behavior of solutions to this and more general systems near {non-smooth boundary} points {of the domain}. They are of the form  $\mathbb{D} = \mathbb{K} \times \mathbb{R}^{n-m} \subset \mathbb{R}^n$, with $m \geq 2$ and $\mathbb{K} \subset \mathbb{R}^m$ an open cone. 
We use local coordinates $\bold{x} = (\bold{y},z)$ in the wedge $\mathbb{D}$.

For $n \geq 2$ we consider the elastodynamic problem  \eqref{components with hooke tensor} - \eqref{dirichlet condition} in the space-time cylinder $\mathbb{D}\times \mathbb{R}$, written abstractly in the form:
\begin{align}\label{wedgetimedomain}
\mathcal{L}(D_{\bold{x}},D_{t})\bold{u}(\bold{x},t) &= \bold{f}(\bold{x},t),  &(\bold{x},t) \in \mathbb{D} \times \mathbb{R} \ , \\
 \bold{u}(\bold{x},t) &= \bold{g}{(\bold{x},t)}, & (\bold{x},t) \in \partial \mathbb{D} \times \mathbb{R}\ . \label{wedgetimedomain2}
\end{align}
with the matrix differential operator   $(\mathcal{L}(D_{\bold{x}},D_{t})\bold{u}(\bold{x},t))_p=\partial_t^2\bold{u}(\bold{x},t) - \sum_{k,l,q=0}^{{n}} \partial_{k} a^{kl}_{pq}(\bold{x})\partial_l u_q(\bold{x},t)$, $p=1, \dots, {n}$.

Applying the Fourier transform $\mathcal{F}_{{t} \mapsto \tau}$  leads to a parameter-dependent elliptic problem, with $\tau = \sigma - i\gamma$, $\gamma > 0$, $\sigma \in \mathbb{R}$:
\begin{equation}\label{one}
\mathcal{L}(D_{\bold{x}},\tau)\bold{v}(\bold{x},\tau) = \hat{\bold{f}}(\bold{x},\tau),  \ \ \bold{x} \in \mathbb{D}, \quad \bold{v}(\bold{x},\tau) = 0, \quad \bold{x} \in \partial  \mathbb{D} \ .
\end{equation}
We denote by $\mathcal{A}_D(\tau) = \mathcal{L}(D_{\bold{x}},\tau)$ {the closure of this operator} in $L^{2}(\mathbb{D})$. We first note a well-posedness  {result, Theorem 4.1.2 in \cite{matyu}.} 

\begin{proposition}\label{wellposedness1} For every $\hat{\bold{f}} \in L^{2}(\mathbb{D})$ and $\tau=\sigma-i \gamma$, $\sigma \in \mathbb{R}$, $\gamma>0$, there exists a unique solution $\bold{v}$ of \eqref{one}. 
Further, there exists a constant $c>0$ independent of $\tau$ and $\hat{\bold{f}}$ such that
\begin{equation}\label{apriori1}
\gamma^{2} \int_{\mathbb{D} } (|\tau|^{2}|\bold{v}(\bold{x},\tau)|^{2} + |D_{\bold{x}}\bold{v}(\bold{x},\tau)|^{2})d\bold{x}  \leq c \int_{\mathbb{D} } |\hat{\bold{f}}(\bold{x},\tau)|^{2}d\bold{x}\ .
\end{equation}
\end{proposition}
\begin{proof} {On the standard Sobolev space $H^{1}_{0}(\mathbb{D})$ we define the sesquilinear form 
\begin{equation*}
 B^\tau_D(\bold{u},\bold{v}) = \int_{\mathbb{D} } \sum_{i,j,k,l} C_{kl}^{ij}(\bold{x}) \partial_k u_{i}(\bold{x}) \partial_l \overline{v_{j}(\bold{x})} d\bold{x} - \tau^{2} \int_{\mathbb{D}} \bold{u}(\bold{x})\cdot \overline{\bold{v}(\bold{x})} d\bold{x}\ ,
\end{equation*}
where $C_{kl}^{ij}$ denotes the Hooke tensor from Section \ref{sec 2}. A key property of $B^\tau_D$ is the Korn inequality, which estimates $B^\tau_D$ in terms of the norm of $H^1(\mathbb{D})$; see Proposition 4.1.3 in \cite{matyu}: If $\tau^2 \in \C \setminus \overline{\R_+}$, then there exists $\delta = \delta(\tau)>0$ such that $|B^\tau_D(\bold{u},\bold{u})|\geq \delta \|\bold{u}; H^1(\mathbb{D})\|^2$.\\
The assertion then follows  by applying the Lax-Milgram theorem.}\end{proof}

\subsection{Solution of parameter-dependent Dirichlet problem in a cone}




For a finer analysis one performs a Fourier transform  $\mathcal{F}_{z \mapsto \xi}$ in the variable $z$ in \eqref{wedgetimedomain}, \eqref{wedgetimedomain2} and introduces polar coordinates in $\mathbb{K}$: $r = |\bold{y}|$, $\pmb{\omega} = \frac{\bold{y}}{|\bold{y}|}$. We first assume that $\bold{v}$
solves the non-homogeneous Dirichlet problem with parameters $\tau \in \mathbb{R}- i \gamma$ and $\xi \in \mathbb{R}$, 
\begin{align}\label{119}
\mathcal{L}(D_{\bold{y}}, \xi, \tau)\bold{v}(\bold{y}, \xi, \tau) &= \hat{\bold{f}}(\bold{y},\xi,\tau),\quad \bold{y} \in \mathbb{K}\\
\bold{v}(\bold{y},\xi, \tau)&=\hat{\bold{g}}(\bold{y},\xi,\tau), \quad \bold{y} \in \partial\mathbb{K}\ .
\label{119b}
\end{align}

For simplicity, we first consider the homogeneous Dirichlet problem, corresponding to $\bold{g}=0$. The corresponding statements for nonzero Dirichlet data $\bold{g}$ can be deduced from the general results for a wedge in Subsection \ref{subsectionwedge}.

\begin{proposition}[{Theorem 6.2.5,} \cite{matyu}]\label{625}
Let $\tau \in \R-i\gamma$ with $\gamma>0$. For all $\hat{\bold{f}} \in L^2(\mathbb{K})$, There exists a unique, strong solution $\bold{v}$ of \eqref{119}, \eqref{119b}, and $$\gamma^2(p^2 \|\bold{v};L^2(\mathbb{K})\|^2 + \|D_{\bold{x}} \bold{v};L^2(\mathbb{K})\|^2) \leq c \|\hat{\bold{f}};L^2(\mathbb{K})\|^2 \ .$$ Here $p = \sqrt{|\xi|^2+|\tau|^2}$, and $c$ is independent of $\xi$, $\tau$.  
\end{proposition}

Define the weighted Sobolev norms
\begin{align}
\| v; H^{s}_\beta(\mathbb{K}) \| &= \left( \sum_{|\alpha| \leq s} \int_{\mathbb{K}} r^{2(\beta + |\alpha| - s)} |D_{\bold{x}}^\alpha v|^2  \right)^{\frac{1}{2}} d\bold{x}\ , \label{96}\\
\| v; H^{s}_\beta(\mathbb{K},p) \| &= \left( \sum_{k=0}^{s} p^{2k}  \| v;H^{s-k}_\beta(\mathbb{K})  \|^{2}  \right)^{\frac{1}{2}} \ .
\end{align}
Let $\chi \in C^\infty_0(\mathbb{K})$ be a cut-off function which is $=1$ in a neighborhood of the vertex of the cone $\mathbb{K}$, and $\chi_\tau(\bold{x}) = \chi(|\tau| \bold{y})$. From Proposition \ref{625} one obtains with $p =\sqrt{|\xi|^2+|\tau|^2}$, and $c$ independent of $\xi$, $\tau$, 
\begin{equation}\label{H2estimate}\gamma^2\|\bold{v}; H^1_\beta(\mathbb{K}, p)\|^2 + \|\chi_\tau \bold{v}; H^2_\beta(\mathbb{K},p)\|^2 \leq c \left\{\|\mathcal{L}(D_{\bold{y}},\xi,\tau)\bold{v}; H^0_\beta(\mathbb{K})\|^2+ \frac{p^{2(1-\beta)}}{\gamma^2}\|\mathcal{L}(D_{\bold{y}},\xi,\tau)\bold{v}; L^2(\mathbb{K})\|\right\}.\end{equation}

Set $\Xi = \mathbb{K} \cap S^{{m}-1}$. For every $\lambda \in \mathbb{C}$ the pencil 
\begin{equation}\mathcal{A}_D(\lambda)\pmb{\varphi} = \left\{r^{2-i\lambda} \mathcal{L}(D_{\bold{y}},0,0) r^{i\lambda}\pmb{\varphi},\pmb{\varphi}|_{\partial \Xi}\right\}\label{pencildef}\end{equation} defines a map
$$\mathcal{A}_D(\lambda) : H^2(\Xi) \to L^2(\Xi) \times H^{3/2}(\partial \Xi)\ , $$ which is an isomorphism except for a discrete set of eigenvalues $\{\lambda_\ell\}$.

For the elastodynamic equation $\mathcal{L}$ has constant coefficients and is of the form $\mathcal{L}(D_{\bold{x}},D_{t})\bold{v} = \partial_t ^{2}\bold{v} + A(D_{\bold{x}})\bold{v}$ with $$A(D_{\bold{x}}) = A(D_{\bold{y}}, D_z) =  D_{k} A^{kl} D_{l}\ ,$$
where each of the $A^{kl}$ is a constant matrix $A^{kl}=(a^{kl}_{ij})_{i,j}$. The operator pencil is then given by
 \begin{equation}
 \left\{r^{2-i\lambda} A(D_{\bold{y}},0) r^{i\lambda}\pmb{\varphi},\pmb{\varphi}|_{\partial \Xi}\right\}\ . \label{pencilspec}
\end{equation}
We assume that the strip $\{ \lambda \in \mathbb{C}: m-3  \leq 2 \mathrm{Im}\ \lambda \leq m - 2 \}$ does not intersect the spectrum of $\mathcal{A}_{D}$.
For an eigenvalue $\lambda_\ell$ of $\mathcal{A}_D$ we take a power-like solution
\begin{equation}\label{powerlike} \bold{u}_\ell(\bold{y}) = r^{i\lambda_\ell} \sum_{q=0}^k \frac{1}{q!} (i \ln(r))^q \pmb{\varphi}_\ell^{(k-q)}(\pmb{\omega})\end{equation}
of the homogeneous Dirichlet problem with $\tau=0$, $\xi=0$:
\begin{align}\label{111}
\mathcal{L}(D_{\bold{y}}, 0, 0)\bold{u}(\bold{y}) &= 0,\quad \bold{y} \in \mathbb{K}\ ,\\
\bold{u}(\bold{y})&=0, \quad \bold{y} \in \partial\mathbb{K}\ .
\label{111b}
\end{align}
Here, $\{\pmb{\varphi}_\ell^{(0)}, \dots, \pmb{\varphi}_\ell^{(k)}\}$ is a Jordan chain to $\lambda_\ell$, consisting of an eigenvector $\pmb{\varphi}_\ell^{(k)}$ and generalized eigenvectors $\pmb{\varphi}_\ell^{(0)}, \dots, \pmb{\varphi}_\ell^{(k-1)}$. Let $\kappa_1 \geq \kappa_2 \geq \dots \geq \kappa_J$ denote the partial multiplicities of the $\lambda_\ell$ , and let $\{\pmb{\varphi}_\ell^{(0,j)}, \dots, \pmb{\varphi}_\ell^{(\kappa_j-1,j)} : j=1, \dots, J\}$ be a canonical system of Jordan chains. The functions
\begin{equation}\label{Jordanchain}
\bold{u}^{(k,j)}_\ell(\bold{y}) = r^{i \lambda_\ell} \sum_{q=0}^k \frac{1}{q!} (i \ln(r))^q \pmb{\varphi}_\ell^{(k-q,j)}(\pmb{\omega}),
\end{equation}
 where $k=0, \dots, \kappa_j - 1$ and $j=1,\dots, J$, constitute a basis in the space of power-like solutions corresponding to $\lambda_\ell$.
 

\begin{remark}
In special geometries the spectral problem for $\mathcal{A}_D$ admits an explicit solution. See {Section \ref{regularity2d}} for a discussion of the eigenvalues and eigenfunctions in the case of a polygon, Section \ref{regularitywedge} for an edge, and Section \ref{regularitycone} for a circular cone. 
\end{remark}

Let  $\bold{V}_{\ell}^{(k,j)}$ be the infinite series of dual functions satisfying the homogeneous equations \eqref{111}, \eqref{111b}, and let $\bold{V}_{\ell, {M}}^{(k,j)}$ be its truncation after ${M}$ terms.




The dual vector functions
\begin{equation}\label{dualfunctions}
\bold{v}^{(k,j)}_\ell(\bold{y}) = r^{i \overline{\lambda_\ell} - (m-2)} \sum_{q=0}^k \frac{1}{q!} (i \ln(r))^q \pmb{\psi}^{(k-q,j)}_\ell(\pmb{\omega}),
\end{equation}
form a basis in the space of power-like solutions to \eqref{111}, \eqref{111b} that correspond to the eigenvalue $\overline{\lambda_\ell} + i(m-2)$. The bases match under specific orthogonality and normalization conditions (see, for example, (114) in \cite{matyu}), respectively \cite{nazarov}. 

Denote by $\{\bold{u}_\ell^{k,j}\}$, $\{\bold{v}_\ell^{k,j}\}$ the matched bases of power-like solutions of \eqref{111}, \eqref{111b}. Next we consider the homogeneous problem with parameters $\tau \in \mathbb{R}- i \gamma$ and $\xi \in \mathbb{R}^{{n-m}}$, corresponding to \eqref{119}, \eqref{119b}, 
\begin{align}\label{116}
\mathcal{L}(D_{\bold{y}}, \xi, \tau)\bold{v}(\bold{y}, \xi, \tau) &= 0,\quad \bold{y} \in \mathbb{K}\\
\bold{v}(\bold{y},\xi, \tau)&=0, \quad \bold{y} \in \partial\mathbb{K}\ .
\label{116b}
\end{align}
Substituting $\bold{u}_\ell^{(k,j)}$ in \eqref{116}, \eqref{116b}, we construct the formal series 
\begin{equation}\label{117}
\bold{U}_\ell^{(k,j)}(\bold{y},\xi,\tau) = \sum_{q=0}^\infty r^{i\lambda_\ell + q} \bold{P}{(k,j)}_q(\pmb{\omega},\xi,\tau,\ln(r))
\end{equation}
satisfying \eqref{116}, \eqref{116b}. Here $\bold{P}{(k,j)}_q$ are polynomials in $\xi,\tau,\ln(r)$, with coefficients smoothly depending on $\pmb{\omega} \in \Xi$. Replacing $\{\bold{u}_\ell^{k,j}\}$ by $\{\bold{v}_\ell^{k,j}\}$, we obtain the formal series
\begin{equation}\label{118}
\bold{V}_\ell^{(k,j)}(\bold{y},\xi,\tau) = \sum_{q=0}^\infty r^{i(\overline{\lambda_\ell}+i{(m-n-2)}) + q} \bold{Q}{(k,j)}_q(\pmb{\omega},\xi,\tau,\ln(r)),
\end{equation}
satisfying \eqref{116}, \eqref{116b}. The functions $\bold{Q}{(k,j)}_q$ again obey analogous properties to $\bold{P}{(k,j)}_q$.

In reference \cite{matyu} the formal series $\bold{U}_\ell^{(k,j)}$, $\bold{V}_\ell^{(k,j)}$ 
are constructed for these bases.

 Consider now \eqref{119}, \eqref{119b} with $\chi \bold{v}\in H^2_\beta(\mathbb{K})$, $\hat{\bold{f}} \in H^0_\beta(\mathbb{K}) \cap H^0_\gamma(\mathbb{K})$, for $\gamma<\beta$. As above, $\chi \in C^\infty_0(\mathbb{K})$ denotes a cut-off function which is $=1$ in a neighborhood of the vertex of the cone $\mathbb{K}$. If the line $\{\lambda \in \mathbb{C} : \mathrm{Im} \ \lambda = \gamma + \frac{m}{2}-2\}$ does not intersect the spectrum of the pencil $\mathcal{A}_D$, then we have
$$ \bold{v} = \chi \sum c_\ell^{(k,j)} \bold{U}^{(k,j)}_{\ell,{M}} + \bold{h}\ ,$$
where the remainder $\bold{h}$ is such that $\chi \bold{h} \in H^2_\gamma(\mathbb{K})$. Here $\bold{U}^{(k,j)}_{\ell, {M}}$ is the partial sum of the series $\bold{U}^{(k,j)}_{\ell}$ containing {M} terms such that $\chi r^{i\lambda_\ell + ({M}+1)} \bold{P}_{{M}+1}^{(k,j)} \in H^2_\gamma(\mathbb{K})$. The asymptotic formula for $\bold{v}$ involves the summands corresponding to the eigenvalues of the pencil in the strip $\{\lambda \in \mathbb{C} : \mathrm{Im} \ \lambda \in (\gamma + \frac{m}{2}-2,\beta + \frac{m}{2}-2)\}$, so that $\chi \bold{U}^{(k,j)}_{\ell,{M}} \in H^2_\beta(\mathbb{K})$ and $\chi \bold{U}^{(k,j)}_{\ell,{M}} \not \in H^2_\gamma(\mathbb{K})$

To state the main result for the  expansion of the parameter-dependent problem near the vertex of the cone $\mathbb{K}$, we introduce the following function spaces:
\begin{align*}
\| \bold{v}; DH_\beta(\mathbb{K}, \xi, \tau) \| = \left( \gamma^2 \|\bold{v}; H^1_\beta(\mathbb{K},p)\|^2 + \|\chi_p \bold{v}; H^2_\beta(\mathbb{K},p)\|^2 \right)^{\frac{1}{2}}\ ,\\
\| \hat{\bold{f}}; RH_\beta(\mathbb{K}, \xi, \tau) \| = \left( \|\hat{\bold{f}}; H^0_\beta(\mathbb{K})\|^2 + p^{2(1-\beta)} \gamma^{-2}\|\hat{\bold{f}}; L^2(\mathbb{K})\|^2 \right)^{\frac{1}{2}} \ ,
\end{align*}
where $p = \sqrt{|\xi|^2 + |\tau|^2}$ and $\tau = \sigma - i \gamma$ ($\sigma \in \mathbb{R}$, $\gamma>0$).
By Proposition \ref{625} and \eqref{H2estimate}, the operator $\mathcal{L}(D_{\bold{y}},\xi,\tau)$ from Problem \eqref{119}, \eqref{119b}, defines a continuous map $\mathcal{L}(D_{\bold{y}},\xi,\tau): DH_\beta(\mathbb{K}, \xi, \tau) \to  RH_\beta(\mathbb{K}, \xi, \tau)$.

In \cite{matyu}, Matyukevich and Plamenevski\v{\i} investigate the dependence of properties of $\mathcal{L}(D_{\bold{y}},\xi,\tau)$ on $\beta$. Let $1>\beta_1>\beta_2> \dots$ be numbers in $(-\infty,1]$ such that every line $\{\lambda \in \mathbb{C} : \mathrm{Im}\ \lambda = \beta_r+\frac{m}{2}-2\}$ contains at least one eigenvalue of the pencil $\mathcal{A}_D$.

Matyukevich and Plamenevski\v{\i} obtain the following results:

\begin{theorem}[{Theorem 6.3.5, } \cite{matyu}]
Suppose that $\beta \in (\beta_1,1]$, $\gamma>0$ and $\hat{\bold{f}} \in RH_\beta(\mathbb{K}, \xi, \tau)$. Then \eqref{119}, \eqref{119b} with right hand side $\hat{\bold{f}}$ admits a unique solution $\bold{v}$ satisfying 
$$\| \bold{v}; DH_\beta(\mathbb{K}, \xi, \tau) \| \leq c \| \hat{\bold{f}}; RH_\beta(\mathbb{K}, \xi, \tau) \|,$$
where $c$ is independent of $(\xi,\tau)$.
\end{theorem}

\begin{theorem}[{Proposition 6.4.1,} \cite{matyu}]\label{thm95}
Suppose $\gamma>0$, $\beta \in(\beta_{r+1}, \beta_r)$, $0<\beta_r-\beta<1$, $\hat{\bold{f}} \in RH_\beta(\mathbb{K}, \xi, \tau)$ and $$(\hat{\bold{f}}, \bold{w}_{\ell}^{(k,j)}(\cdot, \xi,\overline{\tau}))_{L^2(\mathbb{K})} = 0$$
for all $\bold{w}_{\ell}^{(k,j)}$ corresponding to eigenvalues of $\mathcal{A}_D$ in the strip $\{\mathrm{Im} \ \lambda \in (\beta_{r+1}+\frac{m}{2}-2, \beta_1+\frac{m}{2}-2)\}$.
Then the solution $\bold{v}$ of \eqref{119}, \eqref{119b}, admits the representation
$$\bold{v}(\bold{y},\xi,\tau) = \chi(p\bold{y}) \sum_{\ell} \sum_{k,j}c^{(k,j)}_\ell(\xi,\tau) \bold{u}_\ell^{(k,j)}(\bold{y}) + \bold{v}_0(\bold{y},\xi,\tau)\ .$$
Here the outer summmation over $\ell$ sums over all eigenvalues $\lambda_\ell$ of the pencil with $\mathrm{Im}\ \lambda = \beta_r+\frac{m}{2}-2$, while the inner summation sums over a basis $\{\bold{u}_\ell^{(k,j)}\}$ of power-like solutions as in \eqref{powerlike} corresponding to $\lambda_\ell$. The remainder $\bold{v}_0$ belongs to $DH_\beta(\mathbb{K}, \xi, \tau)$. 

There holds $$c_\ell^{(k,j)}(\xi,\tau) = p^{i\lambda_\ell} \sum_q \frac{1}{q!} (i \ln(p))^q d_\ell^{(k+q,j)}(\xi,\tau)\ ,$$
with $$d_\ell^{(k,j)}(\xi,\tau) = p^{-2}\left(\hat{\bold{f}}(p^{-1} \cdot,\xi,\tau), \bold{w}_\ell^{(k,j)}(\cdot,\xi /p, \bar{\tau}/p)\right)_{L^2(\mathbb{K})} \ .$$ Moreover there holds $$\|\bold{v}_0; DH_\beta(\mathbb{K}, \xi, \tau) \| \leq c \|\hat{\bold{f}}; RH_\beta(\mathbb{K}, \xi, \tau) \|, \ \ \ |d_\ell^{(k,j)}(\xi,\tau)|\leq c p^{\beta+\frac{m}{2}-2}\|\hat{\bold{f}}; RH_\beta(\mathbb{K}, \xi, \tau) \|,$$
with a constant $c$ independent of $\xi$, $\tau$ and $\hat{\bold{f}}$.
\end{theorem}

\subsection{Solution of a parameter-dependent Dirichlet problem in a wedge}\label{subsectionwedge}

By means of an inverse  Fourier transform {$\mathcal{F}_{\xi \mapsto z}^{-1}$} in the dual edge variable $\xi$, we obtain  results for the general Dirichlet problem in the wedge $\mathbb{D}$,
 \begin{align}\label{75}
\mathcal{L}(\bold{x},D_{\bold{x}},\tau)\bold{v}(\bold{x},\tau) = \hat{\bold{f}}(\bold{x},\tau), \quad \bold{x} \in \mathbb{D},\\
\bold{v}(\bold{x},\tau) = \hat{\bold{g}}(\bold{x},\tau), \quad \bold{x} \in \partial \mathbb{D}\ ,\label{76}
 \end{align}
{the problem in the frequency domain} corresponding to \eqref{wedgetimedomain}, \eqref{wedgetimedomain2}.

The regularity of the solutions is described in the following weighted Sobolev spaces on $\D = \mathbb{K}\times \R^{n-m}$. In $\D$, one uses the coordinates $\bold{x}=(\bold{y},z)$ and introduces polar coordinates in $\mathbb{K}$: $r = |\bold{y}|$, $\pmb{\omega} = \frac{\bold{y}}{|\bold{y}|}$. Define
\begin{align}\label{Hsb}
\| u; H^{s}_\beta(\mathbb{D}) \| = \left( \sum_{|\alpha| \leq s} \int_{\mathbb{D}} r^{{2(}\beta + |\alpha| - s{)}} |D_{\bold{x}}^\alpha u|^2  \right)^{\frac{1}{2}}\ ,\\
\| u; H^{s}_\beta(\mathbb{D},p) \| = \left( \sum_{k=0}^{s} p^{2k}  \| u;H^{s-k}_\beta(\mathbb{D})  \|^{2}  \right)^{\frac{1}{2}} \ .
\end{align}
Corresponding spaces $H^{s}_\beta(\partial\mathbb{D})$ and $H^{s}_\beta(\partial \mathbb{D},p)$ on $\partial \mathbb{D}$ are obtained as trace spaces for $H^{s}_\beta(\mathbb{D})$, respectively $H^{s}_\beta(\mathbb{D},p)$. 


The basic existence result is given by:
\begin{proposition}[{Theorem 4.2.2}, \cite{matyu}]\label{9.6}
Suppose that the wedge $\mathbb{D}$  is admissible in the sense of \cite{matyu}, $\{ \hat{\bold{f}},\hat{\bold{g}} \} \in L^{2}(\mathbb{D}) \times H^{1}(\partial \mathbb{D})$ and $\tau = \sigma - i \gamma$, $\sigma \in \mathbb{R}$, $\gamma > 0$. Then there exists a unique strong solution $\bold{v}$ of \eqref{75} and \eqref{76}. Furthermore, there exists a constant $c>0$ independent of $\tau$ such that 
 \begin{equation*}
\gamma^{2}  \| \bold{v}, H^{1} (\mathbb{D}, |\tau|)  \|^{2} + \gamma \|\bold{p}(\bold{v}), L^{2}(\partial\mathbb{D}) \|^2 \leq c\left(\| \hat{\bold{f}}\|^2_{L^{2}(\mathbb{D})} + \gamma \| \hat{\bold{g}}, H^{1}(\partial \mathbb{D}, |\tau|)  \|^2\right)\ .
 \end{equation*}
\end{proposition}
 
Higher regularity has been obtained by Matyukevich and Plamenevski\v{\i} in the spaces $H^{s}_\beta(\mathbb{D})$. Following \cite{matyu} we only state the result for homogeneous boundary conditions. 
\begin{proposition}[{Proposition 5.1.1,} \cite{matyu}] \label{9.7} Let $\beta  \leq 1$. Assume $\mathrm{Im}\ \lambda = \beta + \frac{m}{2}-2$ does not intersect the spectrum of $\mathcal{A}_{D}$. Then for $\bold{v} \in H^{2}_{\beta}(\mathbb{D},1) \cap H^{1}_{\beta=0}(\mathbb{D})$ with $\mathcal{L}(D_{\bold{x}},0)\bold{v} \in L^{2}(\mathbb{D})$ 
there holds
\begin{align}\label{9.7estimate}
&\|\chi_\tau \bold{v}, H^2_\beta(\D, |\tau|)\|^2 + \gamma^2 \|\bold{v}, H^1_{\beta}(\D, |\tau|)\|^2 \\ &\qquad \leq c \left\{\|\mathcal{L}(D_{\bold{x}},\tau)\bold{v}, H^0_\beta(\D)\|^2 + |\tau|^{2(1-\beta)} \gamma^{-2} \|\mathcal{L}(D_{\bold{x}},\tau)\bold{v}, L^2(\D)\|^2 \right\} \ ,
\nonumber \end{align}
where $\chi_\tau(\bold{x}) = \chi(|\tau| \bold{y})$ for some {$\chi \in C^\infty_0(\overline{\mathbb{K}})$} which is $=1$ in a neighborhood of the vertex of the cone $\mathbb{K}$. The constant $c$ is independent of $\bold{v}$, $\tau = \sigma-i\gamma$, $\sigma \in \R$, $\gamma>0$.
\end{proposition}
A corresponding result for the wave equation with inhomogeneous boundary conditions has been considered in \cite{plamenevskii}, Formula (7), but we omit the more involved statement. 

The proof in \cite{matyu} is based on three steps: (i) estimates far from the edge, (ii) estimates near the edge, (iii) the global a priori estimate \eqref{apriori1}.

\subsection{Solution of a time-dependent problem in a wedge; non-homogeneous boundary conditions}

We now present results for the time-dependent system \eqref{wedgetimedomain}, \eqref{wedgetimedomain2}, with constant coefficients, obtained from the frequency-domain results via the inverse Fourier transform. They are stated in terms of the following weighted function spaces in the space-time cylinder $\mathcal{Q} = \D \times \R$, with coordinates $\bold{x}=(\bold{y},z) \in \D$ and parameter $q>0$:
$$\|w; H^s_\beta(\mathcal{Q})\| = \left(\sum_{|\alpha| \leq s}  \int_{\R} \int_{\D}r^{2(\beta-s+|\alpha|)} |D_{\bold{x},t}^\alpha w(\bold{x},t)|^2 \ d\bold{x} \ dt\right)^{1/2} \ ,$$
$$\|w; H^s_\beta(\mathcal{Q},q)\| = \left(\sum_{k=0}^{s} q^{2k} \|w; H^{s-k}_\beta(Q)\|^2\right)^{1/2}\ .$$
If $\gamma>0$, we set $w^\gamma(\bold{x},t):= \exp(-\gamma t) w(\bold{x},t)$ and define $$\|w; V^s_\beta(\mathcal{Q},\gamma)\| = \|w^\gamma; H^s_\beta(\mathcal{Q},\gamma)\|\ .$$
The corresponding spaces on the boundary $\partial \mathcal{Q}$ are defined as the trace spaces of $H^s_\beta(Q)$, respectively $V^s_\beta(\mathcal{Q},\gamma)$.

\begin{definition}
Assume $(\bold{f},\bold{g}) \in V^0_0(\mathcal{Q},\gamma) \times V^{3/2}_0(\partial \mathcal{Q},\gamma)$, and let $\bold{v}$ be a strong solution to \eqref{75}, \eqref{76} in $\mathbb{D}$ with right hand side $(\hat{\bold{f}}, \hat{\bold{g}})$. Then $\bold{u}(\bold{y},z,t) = \mathcal{F}^{-1}_{{(\xi,\tau) \to (z,t)}} \bold{v}(\bold{y},\xi,\tau)$ is called a strong solution of \eqref{wedgetimedomain}, \eqref{wedgetimedomain2}.
\end{definition}

Proposition \ref{9.6} implies that for any $(\bold{f},\bold{g}) \in V^0_0(\mathcal{Q},\gamma) \times V^{3/2}_0(\partial \mathcal{Q},\gamma)$ with $\gamma>0$  the problem {\eqref{wedgetimedomain}, \eqref{wedgetimedomain2}} admits a unique strong solution and
$$\gamma^2 \|\bold{{u}}; V^1_0(\mathcal{Q},\gamma)\|^2 +\gamma \|\bold{p}(\bold{{u}}), V^0_\gamma(\partial\mathbb{D},\gamma) \|^2 \leq c\left(\|\bold{f}; V^0_0(\mathcal{Q},\gamma)\|^2+\gamma \|\bold{g}; V^{3/2}_0(\partial \mathcal{Q},\gamma)\|^2\right)\ ,$$
for a constant $c>0$ independent of $\gamma$.

Let $\chi \in C^\infty(\overline{\mathbb{K}})$ be a cut-off function which is identically $1$ in a neighborhood of the conical point $0$. Define \begin{equation}\label{Xdef}X \bold{u}(\bold{y},z,t) = \mathcal{F}^{-1}_{(\xi,\tau) \to (z,t)} \chi(p\bold{y}) \mathcal{F}_{(z',t') \to (\xi,\tau)} \bold{u}(\bold{y},z',t')\end{equation}
and
\begin{equation}\label{Lambdadef}\Lambda^\mu \bold{u}(\bold{y},z,t) = \mathcal{F}^{-1}_{\tau \to t} |\tau|^\mu \mathcal{F}_{t'\to \tau} \bold{u}(\bold{y},z,t')\ .\end{equation}
Higher regularity theorems involve the following norms in $Q$: For $\beta \in \mathbb{R}$ and $\gamma>0$ 
\begin{align}\|\bold{v}; DV_{\beta}(\mathcal{Q},\gamma)\| &= \left(\gamma^2 \|\bold{v}; V^1_\beta(\mathcal{Q},\gamma)\|^2 + \|X\bold{v}; V_\beta^2(\mathcal{Q},\gamma)\|^2+ \gamma \|\partial_\nu \bold{v}; V^0_\beta(\partial \mathcal{Q},\gamma) \|^2\right)^{1/2}\ , \label{DVQ}\\
\|\bold{f}; RV_{\beta}(\mathcal{Q},\gamma)\| &= \left( \|\bold{f}; V^0_\beta(\mathcal{Q},\gamma)\|^2 + \gamma^{-2}\|\Lambda^{1-\beta}\bold{f}; V^0_0(\mathcal{Q},\gamma)\|^2\right)^{1/2}\ . \label{RVQ}\\
\|(\bold{f},\bold{g}); \mathcal{R}V_{\beta}(\mathcal{Q},\gamma)\| &= \left( \|\bold{f}; RV_\beta(\mathcal{Q},\gamma)\|^2 + \|X\bold{g}; V_\beta^{3/2}(\partial \mathcal{Q},\gamma)\|^2+ \gamma \|\bold{g}; V^1_0(\partial \mathcal{Q},\gamma) \|^2\right.\nonumber \\ &\qquad \left. + \gamma^{-1} \|\Lambda^{1-\beta}\bold{g}; V^1_0(\partial \mathcal{Q},\gamma) \|^2\right)^{1/2}\ . \label{RRVQ}\end{align}
More generally, one may introduce for $q \in \mathbb{N}_0$
$$\|\bold{f}; RV_{\beta,q}(\mathcal{Q},\gamma)\| = \left( \sum_{j=0}^q \gamma^{-2j}\|\Lambda^{j}\bold{f}; V^{q-j}_{\beta+q-j}(\mathcal{Q},\gamma)\|^2 + \gamma^{-2(1+q)}\|\Lambda^{1-\beta+q}\bold{f}; V^0_0(\mathcal{Q},\gamma)\|^2\right)^{1/2}\ ,$$
and similarly $\mathcal{R}V_{\beta,q}(\mathcal{Q},\gamma)$ and $DV_{\beta,q}(\mathcal{Q},\gamma)$.

The following result may then be found in Theorem 7.4, \cite{matyu}, for  {${\bf{g}}=0$} and $q=0$. It may be extended to inhomogeneous boundary conditions and $q>0$ using the arguments in \cite{kprussian}.
\begin{theorem}\label{thm9.9} Suppose $q \in \mathbb{N}_0$, $\gamma>0$ and $(\bold{f},\bold{g}) \in \mathcal{R}V_{\beta,q}(\mathcal{Q},\gamma)$.
a) If $\beta \in (\beta_1,1)$, the strong solution $\bold{u}$ to \eqref{wedgetimedomain},  \eqref{wedgetimedomain2} belongs to $DV_{\beta,q}(\mathcal{Q},\gamma)$ and there exists $c>0$ independent of $\gamma$ such that $$\|\bold{u}; DV_{\beta,q}(\mathcal{Q},\gamma)\|
\leq c \|(\bold{f},\bold{g}); \mathcal{R}V_{\beta,q}(\mathcal{Q},\gamma)\|\ .$$
b) If $\beta \in (\beta_{r+1},\beta_r)$, then there exists a solution $\bold{u}$ to \eqref{wedgetimedomain},  \eqref{wedgetimedomain2} if and only if for all $\xi \in \R^{n-m}$, for all $\tau \in \R - i\gamma$ and for all $\bold{w}_\ell^{k,j}$ corresponding to eigenvalues $\lambda_\ell$ of $\mathcal{A}_D$ with $\mathrm{Im}\ \lambda \in[\beta_r + \frac{m}{2}-2, \beta_1 + \frac{m}{2}-2]$, \begin{equation}\label{orthorelation}(\hat{\bold{f}}(\cdot, \xi,\tau), \bold{w}_\ell^{k,j}(\cdot, \xi, \overline{\tau}))_{L^2(\mathbb{K})} + (\hat{\bold{g}}(\cdot, \xi,\tau), {\pmb p}(\bold{w}_\ell^{k,j})(\cdot, \xi, \overline{\tau}))_{L^2(\partial \mathbb{K})}= 0\ .\end{equation}
\end{theorem}

We can now state the main result of this section, which gives the asymptotics of the time-dependent problem in a neighborhood of the edge. It may be found in Theorem 7.5, \cite{matyu}, for  {${\bf{g}}=0$} and $q=0$. {The extension to inhomogeneous boundary data ${\bf{g}}$ follows as in Section \ref{regularitysection}: choose an extension $\widetilde{{\bf g}}$ in the domain with Dirichlet trace ${\bf{g}}$. Theorem 7.5, \cite{matyu} then assures an asymptotic expansion of the function ${\bf U} = {\bf u}-\widetilde{{\bf g}}$, which satisfies homogeneous boundary conditions. The expansion of ${\bf u} = {\bf U} + \widetilde{{\bf g}}$ then follows.}

\begin{theorem}\label{conetheorem}
Suppose $\gamma>0$ and $(\bold{f},\bold{g}) \in \mathcal{R}V_{\beta,q}(\mathcal{Q},\gamma)$ for $\beta \in (\beta_{r+1},\beta_r)$ with $0<\beta_r-\beta<1$. Assume that for all $\xi \in \R^{n-m}$, for all $\tau \in \R - i\gamma$ and for all $\bold{w}_\ell^{k,j}$ corresponding to eigenvalues $\lambda_\ell$ of $\mathcal{A}_D$ with $\mathrm{Im}\ \lambda \in[\beta_r + \frac{m}{2}-2, \beta_1 + \frac{m}{2}-2]$ the relation \eqref{orthorelation} holds.
Then the solution $\bold{u}$ to \eqref{wedgetimedomain}, \eqref{wedgetimedomain2} admits an asymptotic expansion 
\begin{equation}\label{asymptoticexpansion} \bold{u}(\bold{y},z,t) = \sum_\ell \sum_{k,j} (X \tilde{c}_\ell^{k,j})(\bold{y},z,t) \bold{u}_\ell^{k,j}(\bold{y}) + \bold{u}_0(\bold{y},z,t)\ ,\end{equation} 
with $\bold{u}_0 \in DV_{\beta,q}(\mathcal{Q},\gamma)$.
Here the first sum is over all eigenvalues $\lambda_\ell$ with $\mathrm{Im}\ \lambda =\beta_r + \frac{m}{2}-2$, while the second sum is  over all generalized eigenfunctions $\bold{u}_\ell^{k,j}$ corresponding to $\lambda_\ell$.
The coefficients $\tilde{c}_\ell^{k,j}(z,t)$ are defined by
$$\tilde{c}_\ell^{k,j} = \mathcal{F}^{-1}_{(\xi,\tau)\to (z,t)} c_\ell^{k,j}, $$
where
\begin{equation}\label{cformula}c_\ell^{k,j} = p^{i\lambda_\ell} \sum_q \frac{1}{q!} (i \ln p)^q d_\ell^{(k+q,j)}(\xi,\tau), \end{equation}
 and, with $p = \sqrt{|\xi|^2+|\tau|^2}$ and $\bold{w}_\ell^{k,j}$ as in Theorem \ref{thm9.9},
$$d_\ell^{(k+q,j)}(\xi,\tau) = p^{-2}(\hat{\bold{f}}(p^{-1}\cdot, \xi,\tau), \bold{w}_\ell^{k,j}(\cdot, \xi/p, \overline{\tau}/p))_{L^2(\mathbb{K})} + p^{-1} (\hat{\bold{g}}(p^{-1}\cdot, \xi,\tau), {\pmb p}(\bold{w}_\ell^{k,j})(p^{-1}\cdot, \xi/p, \overline{\tau}/p))_{L^2(\partial \mathbb{K})}.$$
Moreover, the following estimates hold: $\|e^{-\gamma t} \tilde{d}_\ell; H^{2-\frac{m}{2}-\beta}(\R^{n-m+1})\| \leq c \|(\bold{f},\bold{g}); \mathcal{R}V_{\beta,q}(\mathcal{Q},\gamma)\|$\\ and $\|\bold{u}_0; DV_{\beta,q}(\mathcal{Q},\gamma)\| \leq c \|(\bold{f},\bold{g}); \mathcal{R}V_{\beta,q}(\mathcal{Q},\gamma)\|$.
\end{theorem}
The explicit formulas show that for $f$ smooth in time also the coefficients $d_\ell$ will be smooth in time. 

Analogous results for the Neumann problem may be obtained in a similar way, see \cite{kokotov2,matyu}. The boundary condition affects the {corresponding} stencil $\mathcal{A}_N$ and consequently its eigenvalues $i \lambda_\ell$ and singular functions $\bold{w}_\ell^{k,j}$. 

\section{}

We 
recall certain auxiliary results from \cite{graded}, which are used in the proofs of Theorem \ref{approxtheorem2} and Theorem \ref{approxtheorem1}. 
\begin{lemma}[{\cite{graded}, Lemma 3}]\label{lemma3.2}
Let $\Gamma,\, \Gamma_j\; (j=1,\dots ,N)$ be Lipschitz domains with $\overline{\Gamma} = \bigcup\limits_{j=1}^N \overline{\Gamma}_j$, {$s \in [-1,1]$ and $r \in \mathbb{R}$. Then for all
$\tilde{u}\in H^r_\sigma(\mathbb{R}^+,\widetilde{H}^s(\Gamma))$, $u\in H^r_\sigma(\mathbb{R}^+,H^s(\Gamma))$,} \begin{align}
	\sum\limits_{j=1}^N \| u\|^2_{r,s,\Gamma_j} \leq \| u\|^2_{r,s,\Gamma} \label{3.21a}\ ,\qquad 
	\| \tilde{u}\|^2_{r,s,\Gamma, \ast} \leq \sum\limits_{j=1}^N \| \tilde{u}\|^2_{r,s,\Gamma_j, \ast}\ .
\end{align}
\end{lemma}

\begin{lemma}[{\cite{graded}, Lemma 8}]\label{lemma3.3}
	Let $I_j = [0, h_j]$, {$r\in \mathbb{R}$}, $0\leq s_j\leq 1$, $f_2\in\widetilde{H}^{-s_2}(I_2)$, $f_1 \in \widetilde{H}^r(\mathbb{R}^+, H^{-s_1}(I_1)$. Then there holds
	\begin{equation*}	
		\| f_1(t,x) f_2(y) \|_{r, -s_1 - s_2, I_1\times I_2, \ast} \leq \| f_1\|_{r, -s_1, I_1, \ast} \| f_2\|_{\tilde{H}^{-s_2}(I_2)} \ .
	\end{equation*}
\end{lemma}
A similar result holds in the positive Sobolev norms:
\begin{lemma}[{\cite{graded}, Lemma 9}]\label{lemma3.3a}
	Let $I_j = [0, h_j],\; 0\leq s\leq 1,\; f_2\in\widetilde{H}^{s}(I_2)$, $f_1 \in \widetilde{H}^s(\mathbb{R}^+, H^{s}(I_1)$. Then there holds
	\begin{equation*}	
		\| f_1(t,x) f_2(y) \|_{r, s, I_1\times I_2, \ast} \leq \| f_1\|_{r, s, I_1, \ast} \| f_2\|_{\tilde{H}^{s}(I_2)}\ .
	\end{equation*}
\end{lemma}
\begin{lemma}[{\cite{graded}, Lemma 10}]\label{lemma3.4}
	Let {$0 \leq r \leq \rho \leq q+1$}, $-1\leq s\leq 0$, $R=[0,h_1]\times[0,h_2],\; u\in H^{\rho}([0,\Delta t), H^1(R))$, $\Pi_t^q u$ the orthogonal projection onto piecewise polynomials in $t$ of order $q$, $\Pi_{x,y}^0 u=\frac{1}{h_1 h_2}\int\limits_R u(t,x,y)  dy\, dx$. Then for $p =  \Pi_t^q \Pi_{x,y} u$ we have 
	\begin{align}
		\| u-p\|_{r,s,R,\ast}&\lesssim \nonumber  (\Delta t)^{\rho-r}{\max\{h_1,h_2,\Delta t \}^{-s}}\|\partial_t^\rho u\|_{L^2([0,\Delta t)\times R)} \nonumber\\ & \qquad +  \max\{h_1,h_2,\Delta t \}^{-s}\left(h_1 \| u_x\|_{L^2([0,\Delta t) \times R)}  + h_2 \| u_y\|_{L^2({[0,\Delta t) \times}R)} \right).\label{3.22}
	\end{align}
	If $u(t,x,y) =u_1(t,x)u_2(y),\; u_1\in H^{{\rho}}( [0,\Delta t), H^1([0,h_1])), \; u_2\in H^1( [0,h_2])$ then 
	\begin{equation*}
		\| u-p \|_{r,s,R,\ast}\lesssim (\Delta t)^{\rho-r}{\max\{h_1,h_2,\Delta t \}^{-s}}\|\partial_t^\rho u\|_{L^2([0,\Delta t)\times R)} +\left(h_1^{1-s}\| u_x\|_{{L^2}([0,\Delta t) \times R)} + h_2^{1-s}\| u_y\|_{{L^2}([0,\Delta t) \times R)} \right).
	\end{equation*}
\end{lemma}

\begin{lemma}[{\cite{graded}, Lemma 11}]\label{TobiasTRACE}
Let $Q = [0,h_1]\times [0,h_2],u\in H^3([0,\Delta t)\times Q),p$ the bilinear interpolant of $u$ at the vertices of $Q$. Then there holds {for $r\in \mathbb{R}$}
\begin{align}
	\| u-p\|_{r,0,[0,\Delta t)\times Q} &\lesssim\max\{h_1, \Delta t\}^2\| u_{xx}\|_{r,0,[0,\Delta t)\times Q} + \max\{h_2, \Delta t\}^2\| u_{yy}\|_{r,0,[0,\Delta t)\times Q}\nonumber \\  & \qquad +  ( \max\{h_1, \Delta t\}^2+ \max\{h_2, \Delta t\}^2 )\| u_{tt}\|_{r,0,[0,\Delta t)\times Q} \nonumber \\ & \qquad + \max\{h_1, \Delta t\}^2 \max\{h_2, \Delta t\}\| u_{xxy}\|_{r,0,[0,\Delta t)\times Q}  \label{bilinearinterpolant2D1} \\
	\|(u-p)_x\|_{r,0,[0,\Delta t)\times Q} &\lesssim\max\{h_1, \Delta t\}\| u_{xx}\|_{r,0,[0,\Delta t)\times Q} + \max\{h_1, \Delta t\}\| u_{xt}\|_{r,0,[0,\Delta t)\times Q} \nonumber \\ &\qquad+ \max\{h_2, \Delta t\}^2\| u_{xyy}\|_{r,0,[0,\Delta t)\times Q}  \label{bilinearinterpolant2D2}
\end{align}

\end{lemma}

\end{appendices}

\section*{Acknowledgements}

This research was supported through the ``Oberwolfach Research Fellows'' program in 2020.





\end{document}